\documentclass[a4paper,10pt]{article}
\usepackage{float}
 \usepackage{latexstyle}
 \usepackage{bbm}
 \usepackage[section]{placeins}
\usepackage{mathtools}
 
\usepackage{amsmath}
\usepackage{caption}

\makeatletter
\renewcommand*\env@matrix[1][*\c@MaxMatrixCols c]{%
  \hskip -\arraycolsep
  \let\@ifnextchar\new@ifnextchar
  \array{#1}}
\makeatother

\makeatletter
\AtBeginDocument{%
  \expandafter\renewcommand\expandafter\subsection\expandafter{%
    \expandafter\@fb@secFB\subsection
  }%
}
\makeatother

\title{On the role of total variation in compressed sensing}

\author{Clarice Poon\thanks{cmhsp2@cam.ac.uk}\\ Department of Applied Mathematics and Theoretical Physics,\\ University of Cambridge}

\date{July, 2014; Revised April, 2015}

\begin{document}

\maketitle

\begin{abstract}
This paper considers the problem of recovering a one or two dimensional discrete signal which is approximately sparse in its  gradient from an incomplete subset of its  Fourier coefficients which have been corrupted with noise. We prove that in order to obtain a reconstruction which is robust to noise and stable to inexact gradient sparsity of order $s$ with high probability, it suffices to draw $\ord{s \log N}$ of the available Fourier coefficients uniformly at random. However, we also show that if one draws $\ord{s \log N}$ samples in accordance with a particular distribution which concentrates on the low Fourier frequencies, then the stability bounds which can be guaranteed are optimal up to $\log$ factors. Finally, we prove that in the one dimensional case where the underlying signal is gradient sparse and its sparsity pattern satisfies a minimum separation condition, to guarantee exact recovery with high probability, for some $M<N$, it suffices to draw $\ord{s\log M\log s}$ samples uniformly at random from the Fourier coefficients whose frequencies are no greater than $M$.
\end{abstract}
\section{Introduction}
This paper revisits the theory behind one of the first instances of compressed sensing: the recovery of a gradient sparse signal from  a small subset of its discrete Fourier data. This problem was first studied in \cite{candes2006robust}, where it was shown that one can achieve exact recovery with highly incomplete measurements by drawing the measurements uniformly at random. We recall the main result of \cite{candes2006robust} in Theorem \ref{thm:original}.  In the following result and throughout this paper, for $N\in\bbN$ and for each $p\geq 1$, $\nm{\cdot}_p$ denotes the norm over the complex vector space of $\bbC^N$ defined by $\nm{z}_p^p := \sum_{j=1}^N \abs{z_j}^p$ for each $z\in\bbC^N$.  Also, given $\Delta \subset \bbZ$, let $\rP_\Delta$ denote the projection matrix, which restricts a vector to its entries indexed by $\Delta$.
\begin{theorem}[\cite{candes2006robust}]\label{thm:original}
Let $N\in\bbN$ and let $\rD:\bbC^N\to \bbC^N$ be such that for any $z\in\bbC^N$,  $\rD z = (z_j - z_{j+1})_{j=1}^N$ where $z_{N+1} := z_1$.   Define $\nmu{z}_{TV} := \nmu{\rD z}_1$. 
Let $\rA\in\bbC^{N\times N}$ be the discrete Fourier transform on $\bbC^N$, such that given $z\in\bbC^N$,
\be{\label{eq:DFT}
\rA z = \left(\sum_{j=1}^N z_j e^{2\pi i k j/N}\right)_{k=-\lfloor  N/2 \rfloor +1}^{\lceil  N/2 \rceil}.
}
Let $x\in\bbC^N$ be gradient $s$-sparse, i.e. $\abs{\br{ j: \abs{(\rD z)_j}\neq 0}} = s$ and let $\epsilon\in (0,1)$.
Suppose that $\Omega = \Omega'\cup \br{0}$ where $\Omega'\subset \br{-\lfloor N/2 \rfloor +1,\ldots, \lceil  N/2 \rceil}$  consists of $m$ indices chosen uniformly at random with
$$
m\geq C \cdot s\cdot  \left(\log(N)+ \log(\epsilon^{-1})\right)
$$
for some numerical constant $C$. Then, with probability exceeding $1-\epsilon$, $x$ is the unique solution to
\be{\label{eq:orig}
\min_{z\in\bbC^N} \nmu{z}_{TV} \text{ subject to } \rP_\Omega \rA z = \rP_\Omega \rA x.
}
\end{theorem}

This result can be easily extended to two dimensions and was significant because of its close links to practical applications -- the total variation norm is widely used in imaging applications \cite{chambolle2010introduction, rudin1992nonlinear} since
natural images are generally assumed to be compressible in their gradient, and
furthermore, many imaging devices can be modelled as sampling the Fourier transform of an unknown
object of interest. The original motivation behind Theorem \ref{thm:original} was a reconstruction problem in parallel-beam tomography from \cite{delaney1996fast}, however, this result has since generated much interest for other applications where one directly samples the Fourier transform, such as magnetic resonance imaging \cite{lustig2007sparse} and also applications linked to the Radon transform, such as electron tomography \cite{Leary201370} and radio interferometry\cite{wiaux2009compressed}. As mentioned in the cited works, although the latter two applications are associated with the Radon transform instead of the Fourier transform, the Fourier slice theorem can be exploited to model the sampling process as samples of the Fourier transform.

Theorem \ref{thm:original} suggests that a significant saving in the data acquisition process can be achieved since for $s \ll N$, this sampling cardinality of $\ord{s \log N}$ is  significantly smaller than the number specified by its Nyquist rate. This result was rather spectacular because prior to this, the frequently discussed approach of filtered backprojection algorithms led only to reconstructions with a large number of artefacts. Note also that while solving (\ref{eq:orig}) perfectly recovers $x\in\bbC^N$, if $\nm{\cdot}_{TV}$ was replaced with $\nm{\rD \cdot }_2$ such that we simply solve a Tikhonov regularization problem, the solution necessarily belongs to an $m$ dimensional subspace \cite[Theorem 13.1]{mallat2008wavelet} if $m$ is the number of samples and one cannot expect an exact reconstruction.

However, in order to fully understand the role of total variation in compressed sensing for such practical applications, there are two immediate challenges.
\begin{enumerate}
\item For some unknown signal $x\in\bbC^N$, it is more realistic to assume that we do not observe
$\rP_\Omega \rA x$, but $y$ such that $\nmu{\rP_\Omega \rA x - y}_2\leq \delta \sqrt{m}$ with $m=\abs{\Omega}$ and for some noise level $\delta >0$. Furthermore, often  $x$ is compressible only in its gradient rather than sparse, i.e., $x$ can be approximated by $\rP_\Delta x$ for some $\Delta\subset \br{1,\ldots, N}$ with $\abs{\Delta}=s \ll N$.
So, in practice, the following minimization problem is solved in place of (\ref{eq:orig}).
\be{\label{eq:noise_tv}
\min_{z\in\bbC^N} \nmu{z}_{TV} \text{ subject to } \nmu{\rP_\Omega \rA z - y}_2 \leq \sqrt{m}\cdot \delta.
}
An immediate question is whether the uniform random choice of $\Omega$ from Theorem \ref{thm:original}  guarantees robust and stable recovery of $x$. (Robustness and stability refer to control over the reconstruction error by $\delta$ and $\nmu{\rP_\Delta^\perp \rD x}_1$ respectively.)
\item Since the publication of \cite{candes2006robust}, there has been much empirical work on how to choose $\Omega$ such that we minimize its cardinality, while retaining properties of stability and robustness. In practice, $\Omega$ is often chosen in accordance with some \textit{variable density distribution} \cite{ lustig2007sparse, lustig2008compressed, benning2014phase}, which concentrates more on low Fourier frequencies and less on high Fourier frequencies. Furthermore, in practice, $\Omega$ is not chosen in the uniform random manner as specified by Theorem \ref{thm:original}, because empirically, it has been observed to result in  inferior reconstructions when compared with variable density sampling schemes. Thus, it is of interest to derive theoretical statements to understand the reconstruction qualities of solutions to (\ref{eq:noise_tv}) with a non-uniform choice of $\Omega$.
\end{enumerate}

This paper will derive theoretical results for the first question and partially address the second question. 
Note that the finite differences operator considered in \cite{candes2006robust} was defined with periodic boundary conditions and this formulation will be assumed throughout this paper. However, the use of periodic boundary conditions is not essential in practice and is used here  only to simplify the analysis (we will exploit the fact that the periodic gradient operator satisfies a commutative property with the discrete Fourier transform) and because it allows for significant simplifications in the numerical implementation  of (\ref{eq:noise_tv}). See \cite{goldstein2009split} for algorithmic details.

\subsection*{Related results and overview}
 Prior work relating to the use of total variation in compressed sensing for the stable and robust recovery of signals in two or higher dimensions include \cite{needell2013stable, needell2013near}. Their work considered the recovery of gradient sparse signals by solving (\ref{eq:noise_tv}) where $\rA$ is a matrix which satisfies a restricted isometry property  (defined in Section \ref{sec:stab_rip})  when composed with the discrete Haar transform. Although this property is not satisfied by the discrete Fourier transform, \cite{ward2013stable} demonstrated that one can exploit these results to derive recovery results for the case of weighted Fourier samples. More recently, recovery guarantees for total variation minimization from random Gaussian samples in the one dimensional case have also been derived \cite{cai2013guarantees}.  
 
However, to date, there have been few works directly analysing the use of total variation when sampling the Fourier transform and the purpose of this paper is to extend the result of \cite{candes2006robust} to include the case of inexact gradient sparsity and noisy Fourier measurements. The main results are presented in Section \ref{sec:main_1}.
We prove that reconstructions obtained through uniform random sampling are robust to noise and stable to inexact sparsity, although the recovery estimates are not optimal. We also prove that if uniform random sampling is combined with one particular type of variable density sampling (which was introduced in \cite{ward2013stable}), the reconstructions are, up to $\log$ factors, guaranteed to be robust and optimally stable. This suggests that one of the benefits of variable density sampling is added stability and Section \ref{sec:num_stab} will present some numerical examples to support this claim. These two results are proved for the recovery of  one or two dimensional signals, although the techniques  are  applicable in higher dimensions also. 

In contrast to the results of \cite{needell2013stable, needell2013near,ward2013stable}, the results of this paper are not concerned with universal recovery where we guarantee the recovery of \emph{all} gradient $s$-sparse signals from one random sampling set $\Omega$. Instead, we derive  results for the recovery of one specific signal from  a random choice of $\Omega$. For this reason, the proofs in this paper do not rely completely on the restricted isometry property and we require only $\ord{s\log N}$ samples for recovery up to sparsity level $s$, as opposed to $\ord{s\log^5N\log^3s}$  samples as derived in \cite{ward2013stable}.
 
This paper will also consider the recovery of one dimensional signals  
whose gradient sparsity pattern satisfies some minimum separation condition from low frequency Fourier samples. This is presented in Theorem \ref{thm:min_sep_thm}. Due to the close relationship between the discrete Fourier transform and the discrete gradient operator, this result is closely related to the idea of super-resolution, which considers the recovery of a sum of diracs from its low frequency Fourier samples \cite{candes2013super,candes2014towards,tang2012compressive}. Even though super-resolution is studied in an infinite dimensional setting, the  proof of Theorem \ref{thm:min_sep_thm} will make use of finite dimensional versions of the results in \cite{candes2013super,candes2014towards,tang2012compressive}.

\section{Main results}

\subsection{Near-optimal sampling and error bounds}\label{sec:near_optimal}
This section presents results for the recovery of one dimensional and two dimensional vectors, showing how one can recover elements of $\bbC^N$ (or $\bbC^{N\times N}$) from $\ord{s\log N}$ samples, with accuracy up to the best gradient $s$-sparse approximation. Throughout, given $a,b\in\bbR$, let $a \lesssim b$ denote $a\leq C \cdot b$, for some numerical constant $C$ which is independent of all variables under consideration.
\begin{theorem}\label{thm:near_optimal}
For $N = 2^J$ with $J\in\bbN$, let $\rA$ be the discrete Fourier transform  and  let $\rD$ be the discrete gradient operator on $\bbC^N$ from Theorem \ref{thm:original}.
Let $\epsilon\in (0,1)$ and let $\Delta \subset \br{1,\ldots, N}$ with $\abs{\Delta}= s$. 
Let $x\in\bbC^N$ .
 Let  $\Omega_1,\Omega_2 \subset \br{-N/2+1,\ldots, N/2}$ be such that 
 $\abs{\Omega_1} = \abs{\Omega_2}=m$ with 
 $$
 m \gtrsim s \cdot \log(N)(1 + \log(\epsilon^{-1})).
 $$
Let  $\Omega_1$ be chosen uniformly at random, and let $\Omega_2 = \br{k_1,\ldots, k_{m}}$ consist of $m$ indices which are independent and identically distributed (i.i.d.)  such that  for each $j=1,\ldots, m$,
$$
\bbP(k_j = n) = p(n), \quad p(n) = C \left(\log(N)\max\br{1, \abs{n}}\right)^{-1}, \quad n=-N/2+1,\ldots, N/2,
$$
where $C$ is an appropriate constant such that $p$ is a probability measure.   Let $\Omega = \Omega_1\cup \Omega_2$ and suppose that $y\in\bbC^N$ is such that $\nm{y-\rP_\Omega \rA x}_2\leq  \sqrt{m} \cdot \delta$ for some $\delta\geq 0$.
Let $\hat x$ be a minimizer of 
$$
\min_{z\in\bbC^N} \nmu{z}_{TV} \text{ subject to } \nmu{\rP_\Omega \rA z - y}_2 \leq \sqrt{m}\cdot \delta.
$$
Then with probability exceeding $1-\epsilon$, 
$$
\nmu{\rD x  - \rD \hat x}_2 \lesssim \left(\delta \sqrt{s} + \cL_2 \cdot \frac{\nmu{\rP_\Delta^\perp \rD x}_1}{\sqrt{s}}\right),
\qquad
\frac{\nm{x-\hat x}_2}{\sqrt{N}} \lesssim   \cL_1\cdot  \left(\frac{\delta}{\sqrt{s}}  + 
 \cL_2 \cdot \frac{   \nmu{\rP_\Delta^\perp \rD x}_1}{s}\right),
$$
where 
$
\cL_1= \log^2(s) \log(N)\log(m)$ and $ \cL_2 = \log(s) \log^{1/2}(m).
$
\end{theorem}

\begin{remark}\label{rem:optimal}

 To address the optimality of this result, first observe that in the case where there is exact sparsity of level $s$ and no noise, this result guarantees exact recovery from $\ord{s\log N}$ samples, which is the optimal sampling cardinality for gradient $s$-sparse signals \cite{candes2006robust}. In the presence of noise and inexact sparsity, recall from \cite[Theorem 9.14]{mallat2008wavelet} that if $f\in BV[0,1)$ (the space of bounded variation functions which is formally defined in Definition \ref{def:BV}), then the optimal 
error-decay rate for all bounded variation functions by any
type of nonlinear approximation $\tilde f$ from $s$ samples is $\nmu{\tilde f - f}_{L^2[0,1)} = \ord{\nm{f}_V \cdot s^{-1}}$.  Ignoring the contribution from the noise level $\delta$ and the $\log$ factors, the error decay obtained by our theorem is bounded by $\ord{\nm{x}_{TV} \cdot s^{-1}}$ and this is optimal by comparison to the optimal error bounds achievable for bounded variation function. Thus, one can improve upon this result only by removing the $\log$ factors in the error bound.

Note also that the factor of $N^{-1/2}$ on the left hand side of the error bound on the recovered signal naturally arises from the discretization of functions when one links functions of infinite dimensions with their discrete counterparts (see for example, \cite{leveque2007finite}). In general, given $v\in\bbC^{N^d}$, the discrete $\ell^p$ norm is defined to be
$$
\nm{v}_{p, \mathrm{discrete}}^p := \sum_{j\in [N]^d} \abs{v_j}^p N^{-d},
$$
where $[N] = \br{1,\ldots, N}$. Furthermore, the gradient operator considered in the context of discretized functions is $\tilde \rD = N \rD$, where $\rD$ is the finite differences operator defined previously. In the case of $d=1$, given $v\in\bbC^N$,
$$
\nm{v}_{2, \mathrm{discrete}} = N^{-1/2} \nm{v}_2, \qquad \nmu{\tilde \rD v}_{1, \mathrm{discrete}} = \nm{\rD v}_1.
$$
So the occurrence of the $\sqrt{N}$ above is natural.  Note, however, that there is no discrepancy of $N^{1/2}$ in the case of $d=2$: by letting $\tilde \rD = N \rD$, where $\rD$ is now the differences operator for two dimensional vectors (defined in (\ref{eq:discrete_2d_TV}), given any  $v\in\bbC^{N^2}$,
$$
\nm{v}_{2, \mathrm{discrete}} = N^{-1} \nm{v}_2, \qquad \nmu{\tilde \rD v}_{1, \mathrm{discrete}} = N^{-1} \nm{\rD v}_1.
$$
\end{remark}

\begin{remark}
The restriction of $N=2^J$ arises because part of this result will exploit the relation between Haar wavelet coefficients and total variation. However,  this restriction is likely to be an artefact of the proof techniques applied in this paper and not required in practice.
\end{remark}

\begin{remark}
Throughout this paper, drawing $m$ samples uniformly at random refers to sampling \textit{without} replacement. Note however that the samples indexed by $\Omega_2$ are chosen independently and are not necessarily unique.
\end{remark}

To state the two dimensional result, we define the two dimensional $\ell^p$ norm, the discrete Fourier transform, and the discrete gradient operator for two dimensional vectors. For $p\geq 1$, given any $z\in\bbC^{N\times N}$, let $\nm{z}_p^p  = \sum_{j=1}^N\sum_{k=1}^N \abs{z_{k,j}}^p$.
Let $\rA$ be the discrete Fourier transform on $\bbC^{N\times N}$ such that given $z\in\bbC^{N\times N}$,
\be{\label{def:2d_DFT}
\rA z = \left(\sum_{j_1=1}^N\sum_{j_2=1}^N z_{j_1,j_2} e^{2\pi i (j_1 k_1 +j_2 k_2)/N}\right)_{k_1,k_2=-\lfloor  N/2 \rfloor +1}^{\lceil  N/2 \rceil}
}
Define the vertical gradient operator as
$$
\rD_1: \bbC^{N\times N} \to \bbC^{N\times N}, \quad
x \mapsto (x_{j+1,k}-x_{j,k})_{j,k=1}^N
$$
with $x_{N+1,k} = x_{1,k}$ for each $k=1,\ldots, N$ and the horizontal gradient operator as
$$
\rD_2: \bbC^{N\times N} \to \bbC^{N\times N}, \quad
x \mapsto (x_{j,k+1}-x_{j,k})_{j,k=1}^N
$$
with $x_{j,N+1} = x_{j,1}$ for each $j=1,\ldots, N$. Now define the gradient operator $\rD: \bbC^{N\times N}\to \bbC^{N\times N}$ as
\be{\label{eq:discrete_2d_TV}
\rD x = \rD_1 x+ i \rD_2 x,
} and the isotropic total variation (semi) norm as
$$
\nmu{x}_{TV} = \nmu{\rD x}_1.
$$
Given any $\Lambda \subset \bbZ^2$, and $x\in\bbC^{N\times N}$, $\rP_\Lambda: \bbC^{N\times N} \to \bbC^{N \times N}$ is the projection operator such that $\rP_{\Lambda} x$ is the restriction of $x$ to its entries indexed by $\Lambda$.

\begin{theorem}\label{cor:unif_samp_2D_opt}
Let $N = 2^J$ for some $J \in\bbN$. Let $x\in\bbC^{N\times N}$. Let $\epsilon\in (0,1)$, and let $\Delta \subset \br{1,\ldots, N}^2$ with $\abs{\Delta}= s$. Let $\Omega = \Omega_1\cup \Omega_2$ where $\Omega_1,\Omega_2 \subset \br{-\lfloor  N/2 \rfloor +1,\ldots, \lceil  N/2 \rceil}^2$ consist of  $m$ indices each, with
$$
m \gtrsim   s\cdot  \left(1+\log(\epsilon^{-1})\right)\cdot \log\left(N\right).
$$  Let $\Omega_1$ be chosen uniformly at random, and let $\Omega_2 = \br{k_1,\ldots, k_m}$ consist of  i.i.d. indices such that  for each $j=1,\ldots, m$,  and $n,m = -N/2+1,\ldots, N/2$,
$$
\bbP(k_j = (n, m)) = p(n, m), \quad p(n,m) = C' \left( \log(N)\max\br{1, \abs{n}^2 + \abs{m}^2} \right)^{-1},
$$
where $C'>0$ is such that $p$ is a probability measure. Let $\Omega = \Omega_1\cup \Omega_2$ and suppose that $y\in\bbC^N$ is such that $\nm{y-\rP_\Omega \rA x}_2\leq  \sqrt{m} \cdot \delta$ for some $\delta\geq 0$.
Then, with probability exceeding $1-\epsilon$, any minimizer $\hat x$ of 
\be{\label{eq:tvcs_2d}
\min_{z\in\bbC^{N\times N}} \nmu{z}_{TV} \text{ subject to } \nmu{\rP_\Omega \rA z - y}_2 \leq \sqrt{m}\cdot \delta
}
satisfies
\bes{\nmu{\rD x  - \rD \hat x}_2 \lesssim \left(\delta \cdot \sqrt{s} + \cL_2 \cdot \frac{\nmu{\rP_\Delta^\perp \rD x}_1}{\sqrt{s}}\right),\qquad
\nmu{x-\hat x}_2
\lesssim   \cL_1\cdot \left(   \delta + \cL_2 \cdot \frac{\nmu{\rP_\Delta^\perp\rD x}_1}{\sqrt{s}}\right),
}
where
$
\cL_1 =  \log(s)\log(N^2/s)\log^{1/2}(N)\log^{1/2}(m)
$, and $\cL_2 = \log^{1/2}(m)\log(s)$.

\end{theorem}

\begin{remark}
Up to the $\log$ factors, the error bound is typical of compressed sensing results.
To understand the optimality of this result, we first recall from \cite{candes2006near} (which used results of Kashin \cite{kashin1977diameters} and Garnaev and Gluskin \cite{ garnaev1984widths}) that given any reconstruction method $\rF$ and any $s\log(N/s)$ non-adaptive linear measurements of a signal $x\in\bbC^N$, which we will denote by $\rS(x)$, we have that
$$\sup_{x:\nm{x}_1 \leq 1} \br{\nm{x-\rF(y)}_2 : y=\rS(x)} \geq \frac{C}{\sqrt{s}}
$$ 
for some constant $C >0$, and the reconstruction error bound is necessarily bounded from below by a constant times $s^{-1/2}$. Thus, as argued in  \cite{needell2013stable}, the best error bound which one can hope for is
$$
\nm{x-\hat x}_2 \leq C  \nm{\rP_\Delta^\perp \rD x}_1 s^{-1/2},
$$ otherwise we would arrive at a contradiction by using $\nm{\rD x-\rD \hat x}_2 \lesssim \nm{x-\hat x}_2$.  Therefore, the error bounds of Theorem \ref{cor:unif_samp_2D_opt} can only be improved by removing $\log$ factors.
\end{remark}

\begin{remark}
The results of this section 
 consider the case where the sampling set is a combination of uniform random sampling and sampling in accordance with a decaying distribution. This decaying distribution was also studied in \cite{ward2013stable} for the recovery of vectors from their Fourier coefficients using total variation regularization, however, there are substantial differences between our results and the results in \cite{ward2013stable}. First, the results in \cite{ward2013stable} considered a nonstandard noise model, where the Fourier transform is multiplied by a weighted diagonal matrix and do not cover the recovery of one dimensional vectors. Second, even in the two dimensional case, their results guarantee recovery only up to gradient sparsity level $s$ from $\ord{s\log^5(N) \log^3(s) }$ samples, as opposed to $\ord{s\log(N)}$ samples in our results.

\end{remark}

\subsection{Stable and robust recovery from uniform random sampling}\label{sec:main_1}

The previous section demonstrated how one can combine  uniform random sampling and variable density sampling to guarantee recovery which is optimal up to $\log$ factors. This section will show that uniform random sampling on its own can still achieve stable and robust recovery, albeit with non-optimal error estimates. Even though it is open as to whether  better error bounds are possible if one restricts to uniform random sampling, we will present some numerical examples in Section \ref{sec:num_stab} to suggest that any  improvement over the results of this section will be limited.
\begin{theorem}\label{thm:unif_samp}
For $N\in\bbN$, let $\rA$ be the discrete Fourier transform  and  let $\rD$ be the discrete gradient operator on $\bbC^N$ from Theorem \ref{thm:original}. Let $x\in\bbC^N$.
Let $\epsilon\in (0,1)$, and let $\Delta \subset \br{1,\ldots, N}$ with $\abs{\Delta}= s$. 
Let $x\in\bbC^N$  and let $\Omega = \Omega'\cup \br{0}$ where $\Omega' \subset \br{-\lfloor  N/2 \rfloor +1,\ldots, \lceil  N/2 \rceil}$ consists of  $m$ indices chosen uniformly at random with
\be{\label{eq:num_samples_unif}
m \gtrsim   s\cdot \left(1+ \log( \epsilon^{-1})\right)\cdot \log\left(N\right) 
}
for some numerical constant $C$. Suppose that $y = \rP_\Omega \rA x + \eta$ where $\norm{\eta}_2 \leq \sqrt{m}\cdot \delta$.
Then with probability exceeding $1-\epsilon$, any minimizer $\hat x$ of 
$$
\min_{z\in\bbC^N} \nmu{z}_{TV} \text{ subject to } \nmu{\rP_\Omega \rA z - y}_2 \leq \sqrt{m}\cdot \delta$$
satisfies
\bes{\nmu{\rD x  - \rD \hat x}_2 \lesssim \left(\delta\cdot \sqrt{s} + \cL \cdot \frac{\nmu{\rP_\Delta^\perp \rD x}_1}{\sqrt{s}}\right),\qquad
\frac{\nmu{x-\hat x}_2}{\sqrt{N}}
\lesssim  \left(  \delta \cdot \sqrt{s}+  \cL\cdot\nmu{\rP_\Delta^\perp\rD x}_1\right),
}
where  $\cL =  \log^{1/2}(m) \log(s)$. 
\end{theorem}

Note that although the  bound for the signal error is no longer optimal, the  bound on the  gradient error is still optimal up to $\log$ factors by the optimality discussion of Remark \ref{rem:optimal}.
For the case of recovering two dimensional vectors, we have the following result.

\begin{theorem}\label{cor:unif_samp_2D}
Let $N\in\bbN$.
Let $\rA$ be the discrete Fourier transform, let $\rD$ be the discrete gradient operator, and let $\nm{\cdot}_{TV}$ be the total variation norm for two dimensional vectors in $\bbC^{N\times N}$. Let $x\in\bbC^{N\times N}$.
Let $\epsilon\in (0,1)$, and let $\Delta \subset \br{1,\ldots, N}^2$ with $\abs{\Delta}= s$. Let $\Omega = \Omega'\cup \br{0}$, where $\Omega' \subset \br{-\lfloor  N/2 \rfloor +1,\ldots, \lceil  N/2 \rceil}^2$ consists of  $m$ indices chosen uniformly at random with
$$
m \gtrsim  s\cdot  \left(1+\log(\epsilon^{-1})\right)\cdot \log\left(N\right).
$$ 
Suppose that $y = \rP_\Omega \rA x + \eta$ where $\norm{\eta}_2 \leq \sqrt{m}\cdot \delta$.
Then, with probability exceeding $1-\epsilon$, any minimizer $\hat x$ of 
\bes{
\min_{z\in\bbC^{N\times N}} \nmu{z}_{TV} \text{ subject to } \nmu{\rP_\Omega \rA z - y}_2 \leq \sqrt{m}\cdot \delta
}
satisfies
\bes{\nmu{\rD x  - \rD \hat x}_2 \lesssim \left(\delta\cdot \sqrt{s} + \cL \cdot \frac{\nmu{\rP_\Delta^\perp \rD x}_1}{\sqrt{s}}\right),\qquad
\nmu{x-\hat x}_2 \lesssim \left(  \delta \cdot \sqrt{s}+  \cL\cdot \nmu{\rP_\Delta^\perp\rD x}_1\right),
}
where  $\cL = \log^{1/2}(m) \log(s)$.
\end{theorem}

\subsection{Sampling from low Fourier frequencies}\label{sec:main_2}
The final result of this paper considers the reconstruction  of one dimensional vectors when we sample only the low Fourier frequencies. We first require a definition.

\begin{definition} \label{def:min_sep}
Given $N\in\bbN$ and $\Delta = \br{t_1,\ldots, t_s} \subset \br{1,\ldots, N}$ with $t_1<t_2<\cdots <t_s$, let $t_0 = -N+t_s$. Then, the minimum separation distance is defined to be 
$$
\nu_{\min}(\Delta, N) = \min_{j=1}^s  \frac{\abs{t_j - t_{j-1}}}{N}.
$$
\end{definition}

The following result essentially demonstrates that when the large discontinuities of  the underlying signal to be recovered are sufficiently large apart, then we need only to sample from low Fourier frequencies.

\begin{theorem}\label{thm:min_sep_thm}
Let $N\in\bbN$.
Let $\rA$ and $\nmu{\cdot}_{TV}$ be as in Theorem \ref{thm:unif_samp}. Let $x\in\bbC^N$.
Let $\epsilon\in [0,1]$, and let $M\in\bbN$ be such that $N/4 \geq M\geq 10$.  Suppose that $\nu_{\min}(\Delta, N) = \frac{1}{M}$.
Let  $\Omega = \Omega'\cup \br{0}$, where $\Omega' \subset \br{-2M,\ldots, 2M}$ consists of  $m$ indices chosen uniformly at random with
\be{\label{eq:num_samples_min_sep}
m \gtrsim \max\br{ \log^2\left(\frac{M}{\epsilon}\right), \, \log\left(\frac{N}{\epsilon}\right), \, s\cdot \log \left(\frac{s}{\epsilon}\right)\cdot \log\left(\frac{M}{\epsilon}\right) }.
}
Then with probability exceeding $1-\epsilon$, any minimizer $\hat x$ of (\ref{eq:noise_tv}) with $y = \rP_\Omega x + \eta$ and $\nm{\eta}_2\leq \delta\cdot\sqrt{m}$ satisfies
\be{\label{eq:error1}
\frac{\nmu{x-\hat x}_2}{\sqrt{N}}
\lesssim \frac{N^2}{ M^2} \cdot \left(  \delta \cdot s+ \sqrt{s}\cdot \nmu{\rP_\Delta^\perp\rD x}_1\right).
}
Furthermore, if $m=4M+1$, then the error bound (\ref{eq:error1}) holds with probability 1.
\end{theorem}

\section{Explanation of the results}\label{sec:numerics}
In this section, we will explain the significance of the main results through the presentation of several numerical experiments.
The numerical algorithm used throughout this section is the split Bregman method described in \cite{goldstein2009split}\footnote{Matlab code can be downloaded from http://www.ece.rice.edu/\textasciitilde  tag7/Tom\_Goldstein/Split\_Bregman.html}. The relative error of a reconstruction $R$ from samples of some underlying signal $I$ is defined as $\norm{R - I}_2/\norm{I}_2$.
We use the standard definition of signal to noise ratio (SNR) of a perturbed signal $\hat x = x + e$ and say that the SNR of $e$ relative to $x$ is $\mathrm{SNR} = 10\log_{10}\left(\nm{x}_2/\nm{e}_2\right)$.

\subsection{Uniform random sampling is robust and stable}
This section presents some numerical examples to demonstrate the stability and robustness of uniform random sampling. We first consider the \textit{robustness} of the reconstructions of the two gradient sparse signals in Figure \ref{fig:signals} (each of length 512). 
The graphs in Figure \ref{fig:robust} shows the relative errors of solutions of (\ref{eq:noise_tv}), given noisy measurements of the form $y = \rP_\Omega \rA x + e$ for different SNR values of $e$ relative to $ \rP_\Omega \rA x$, where $e$ is the noise vector drawn at random in accordance with a uniform random distribution,  $x$ is a sparse signal from Figure \ref{fig:signals} to be recovered, and $\Omega$ is drawn uniformly at random such that it includes zero and indexes 10\% of the 512 possible samples. 

To illustrate the \textit{stability} of solving (\ref{eq:noise_tv}), we consider the reconstruction of approximately gradient sparse signals of the form $x+e$, where $x$ is one of the gradient sparse signals of Figure \ref{fig:signals} and $e$ is a random perturbation. We will consider the reconstructions of signals of this form for different SNR values of $e$ relative to $x$ such that $\mathrm{SNR} = 10\log_{10}\left(\nm{x}_2/\nm{e}_2\right)$.
The graphs in Figure \ref{fig:inexact} shows the relative errors (the clean signal is now considered to be $x+e$ which is not perfectly sparse) against the different SNR values when reconstructing  from samples of the form $y = \rP_\Omega \rA(x+e)$, where $\Omega$ is drawn uniformly at random such that it includes zero and indexes 10\% of the 512 possible samples. 

 Figure \ref{fig:peppers_noise} illustrates the use of uniform random sampling in the two dimensional case, where we reconstruct a 512 by 512 test image from 35\% of its noise corrupted Fourier coefficients, chosen uniformly at random.
%

\begin{figure}[htb]
\begin{center}
\begin{tabular}{c@{\hspace{-12pt}}c@{\hspace{-12pt}}c}
$x_1$ & $x_2$ & Zoom of $x_2$\\
\includegraphics[width = 0.34\textwidth, trim=1.2cm 0.8cm .2cm 0.4cm, clip=true]{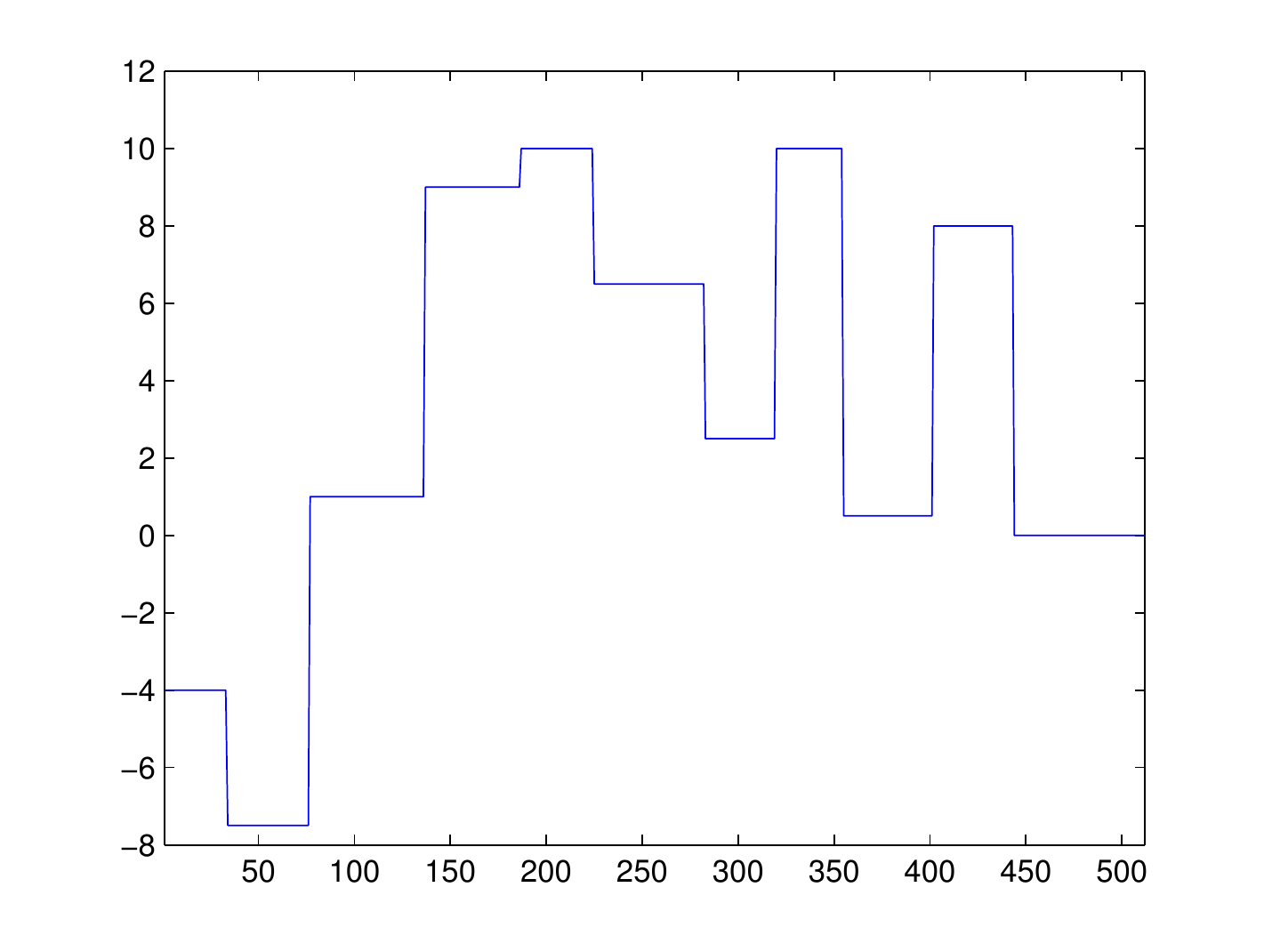}&
\includegraphics[width = 0.34\textwidth, trim=1.2cm 0.8cm .2cm 0.4cm, clip=true]{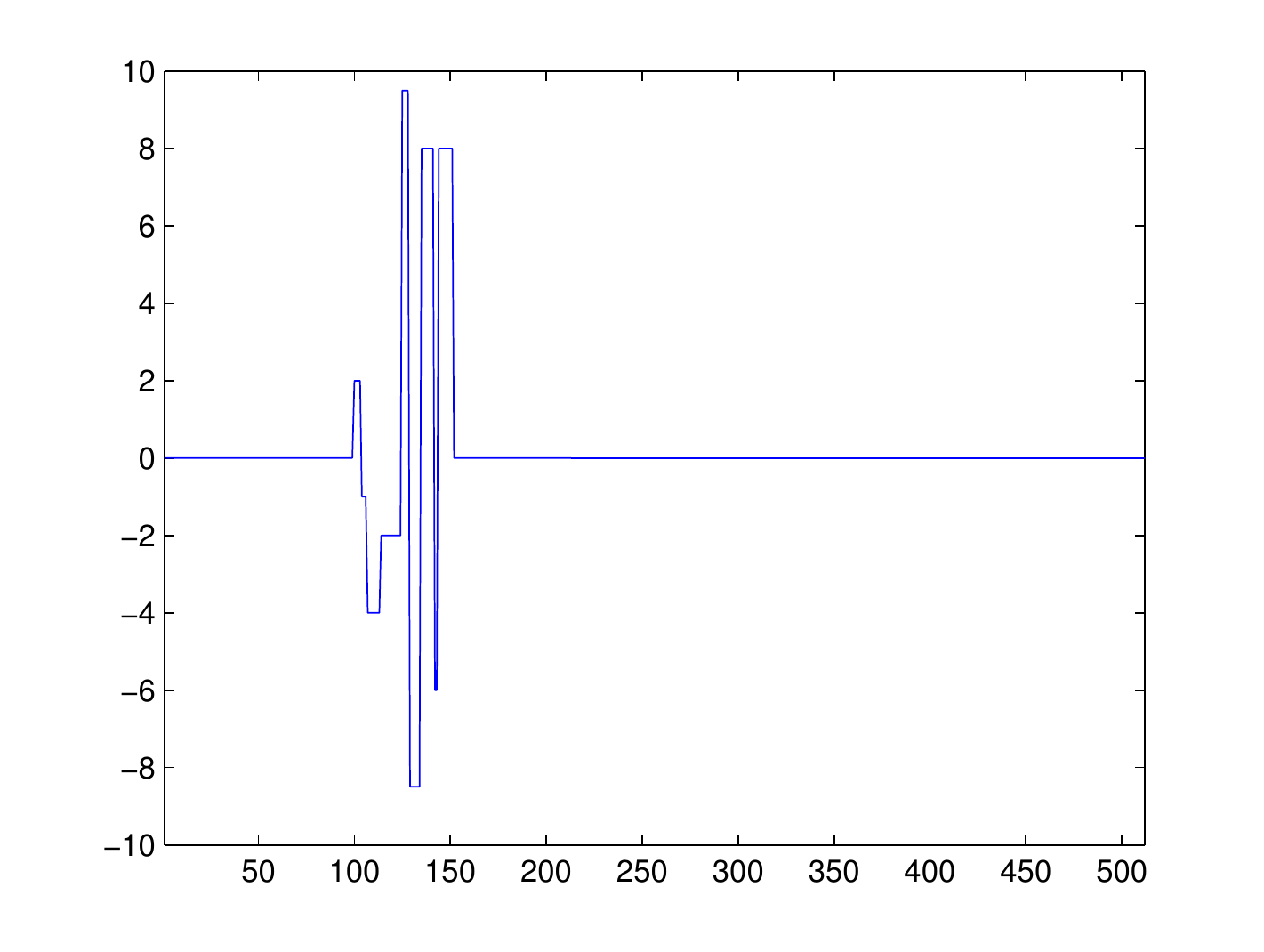}&
\includegraphics[width = 0.34\textwidth, trim=1.2cm 0.8cm .2cm 0.4cm, clip=true]{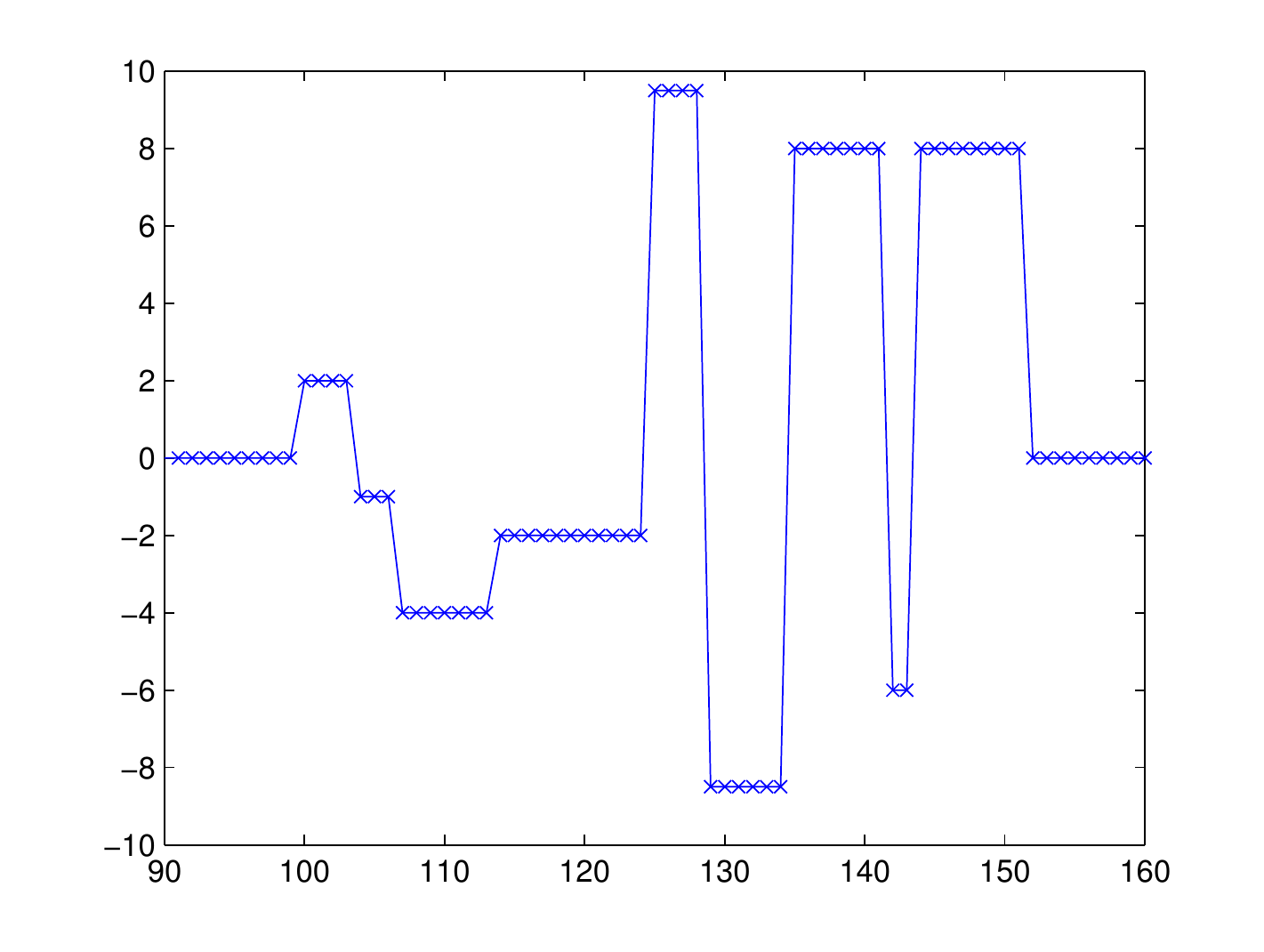}
\end{tabular}
\end{center}
\caption{Two signals consisting of 512 values (left and centre). For the zoom of $x_2$ (right), the value of the signal at each index is marked with a cross. \label{fig:signals}}
\end{figure}

 \begin{figure} 
\begin{center}
\begin{tabular}{c@{\hspace{0pt}}c}
\includegraphics[width = 0.45 \textwidth]{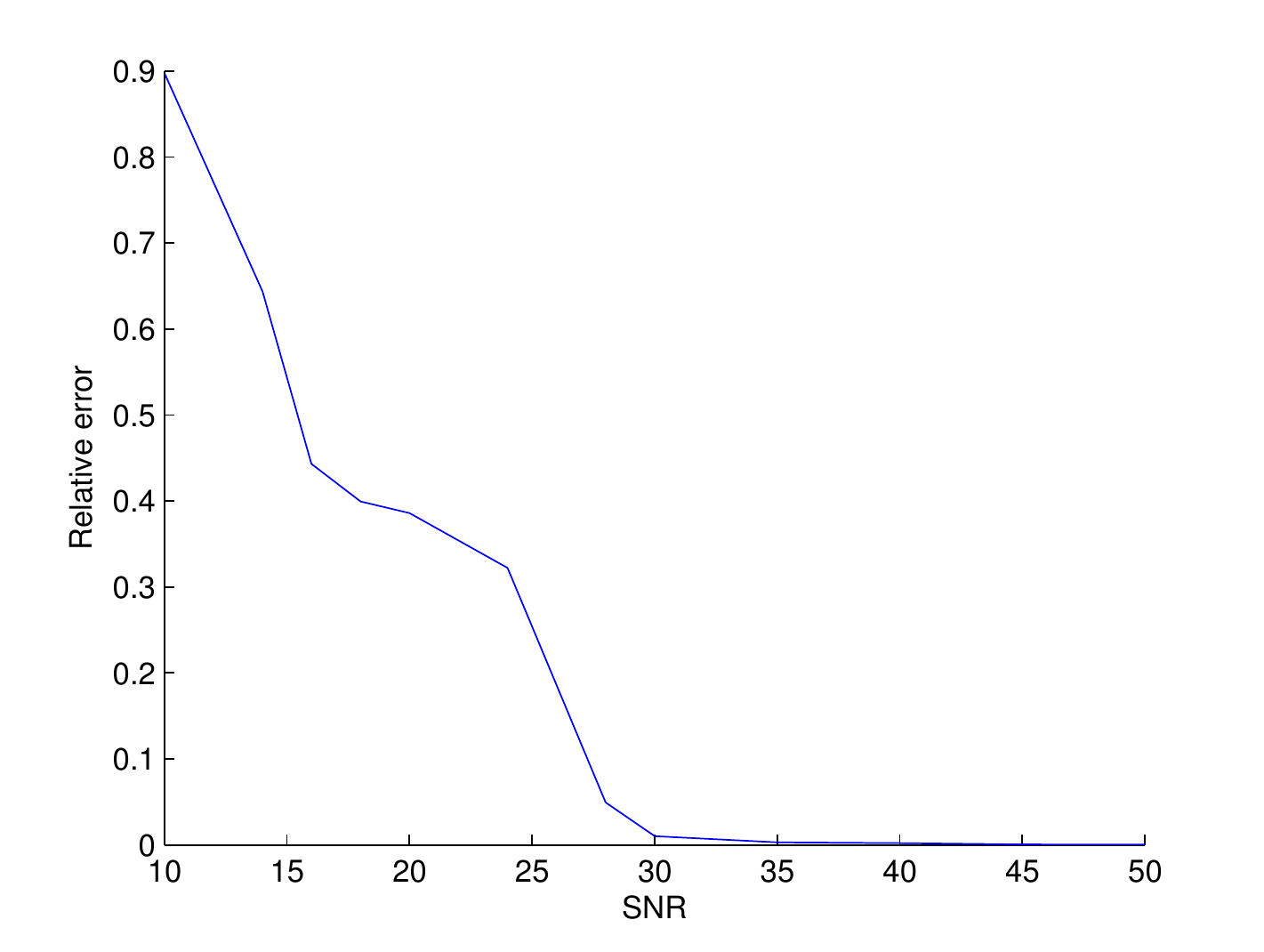}&
\includegraphics[width = 0.45\textwidth]{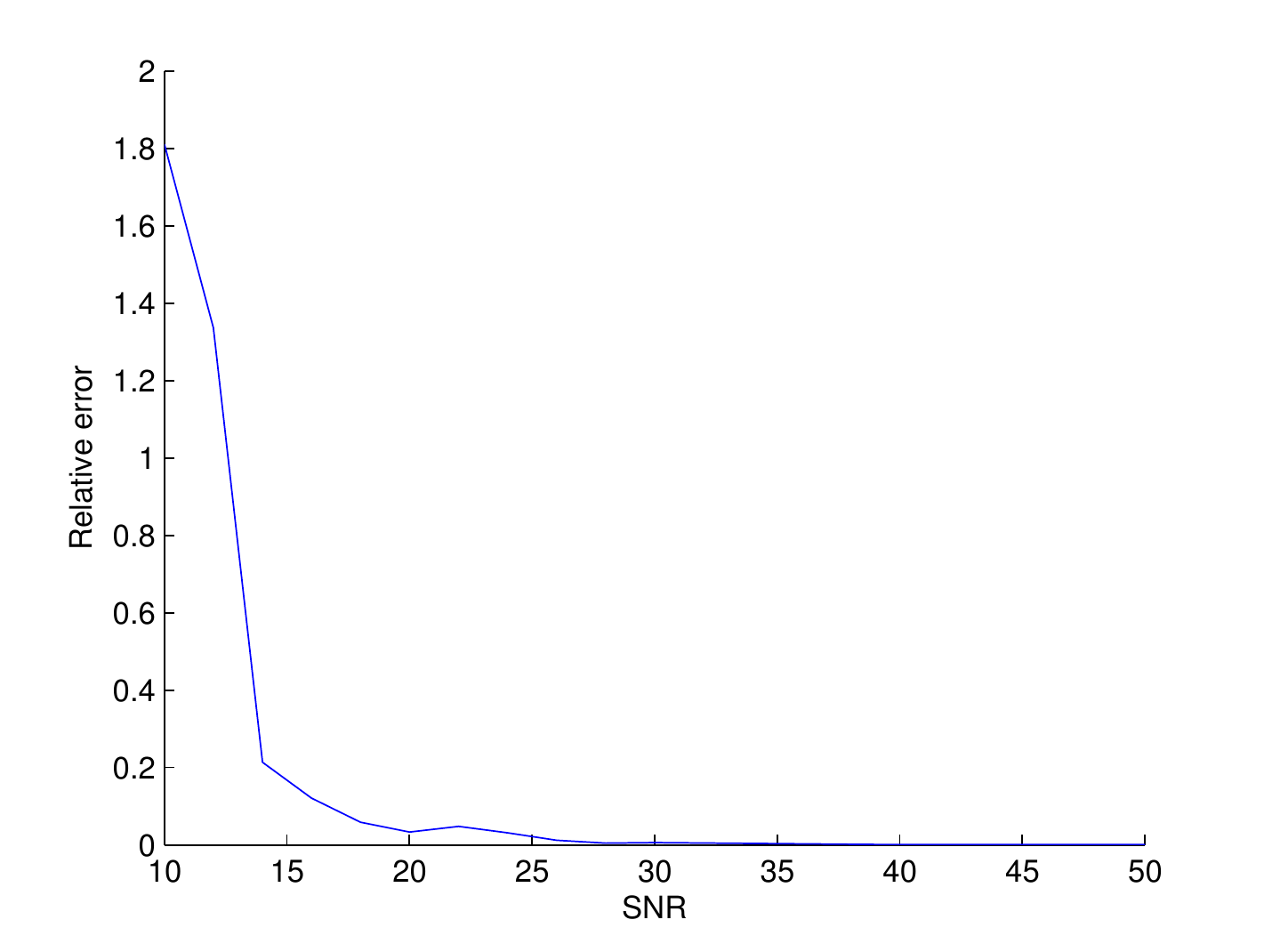}
\end{tabular}
\end{center}
\caption{Relative errors for reconstructions of $x_1$ (left) and $x_2$ (right) from noisy samples. The original signals $x_1$ and $x_2$ are shown in Figure \ref{fig:signals}.\label{fig:robust}}
\end{figure}

\begin{figure} 
\begin{center}
\begin{tabular}{c@{\hspace{0pt}}c}
\includegraphics[width = 0.45\textwidth]{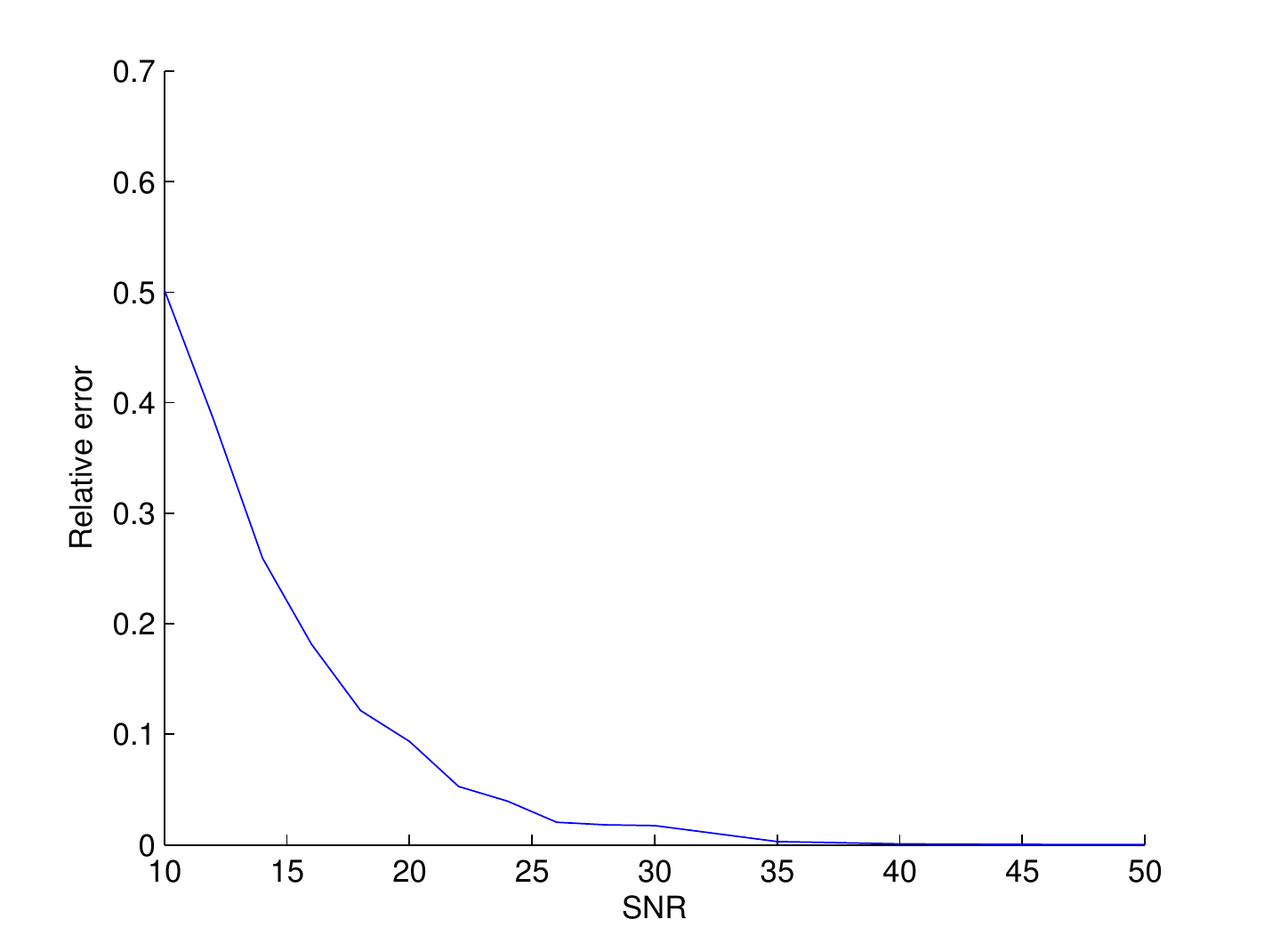}&
\includegraphics[width = 0.45 \textwidth]{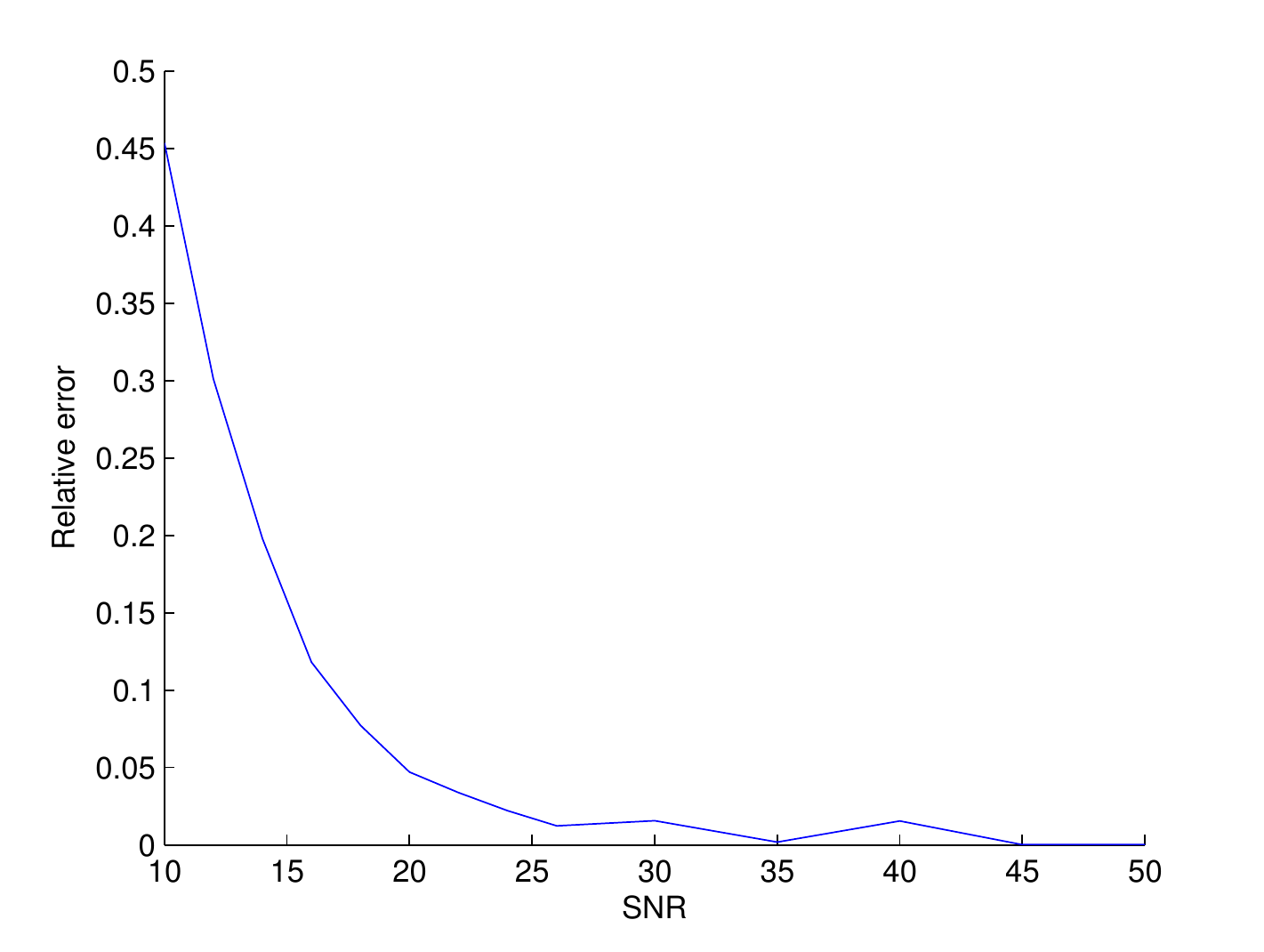}
\end{tabular}
\end{center}
\caption{Relative errors for reconstructions of perturbations of $x_1$ (left) and $x_2$ (right). The original signals $x_1$ and $x_2$ are shown in Figure \ref{fig:signals}.\label{fig:inexact}}
\end{figure}

\begin{figure} 
\begin{tabular}{c@{\hspace{2pt}}c@{\hspace{2pt}}c@{\hspace{2pt}}c}
Original & $SNR = \infty,\, \epsilon_{rel} = 0.20$ & $SNR = 10,\, \epsilon_{rel} = 0.22$ &$SNR = 5, \, \epsilon_{rel} = 0.46$ \\
\includegraphics[width = 0.24 \textwidth]{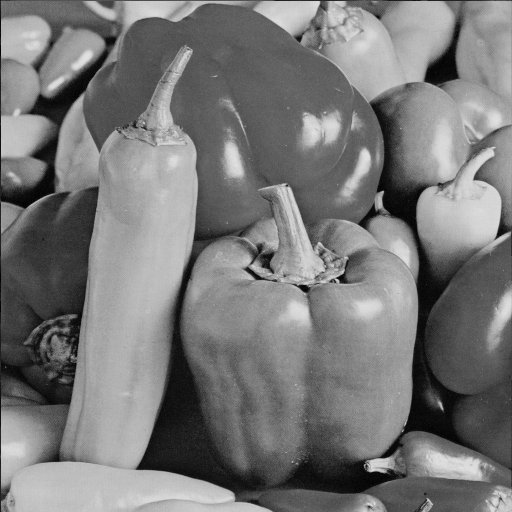}
&\includegraphics[width = 0.24 \textwidth]{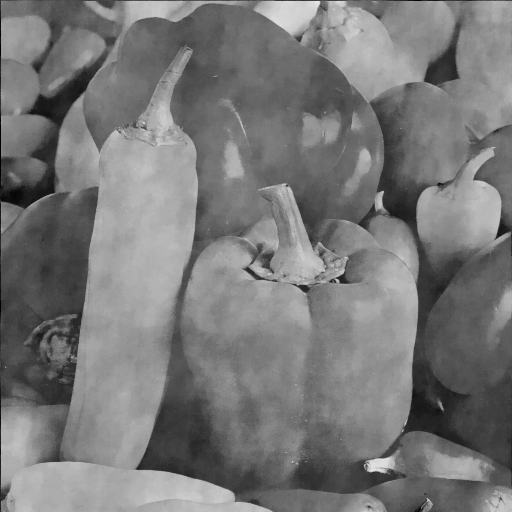}&
\includegraphics[width = 0.24 \textwidth]{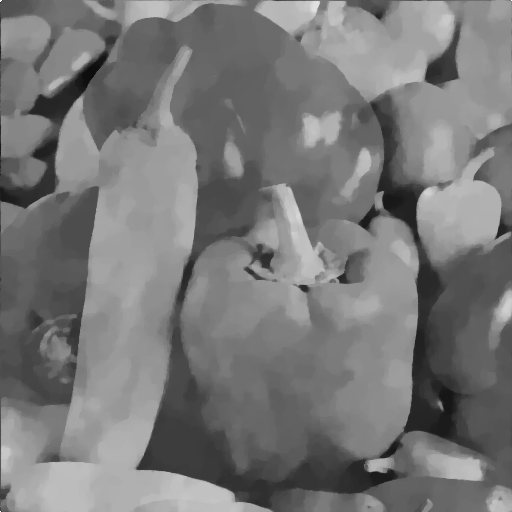}&
\includegraphics[width = 0.24 \textwidth]{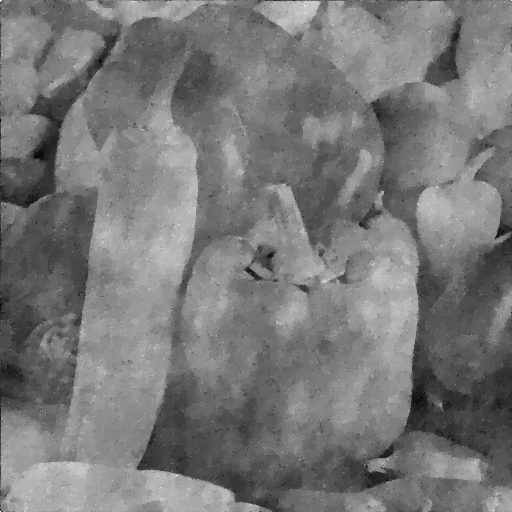}
\end{tabular}
\caption{Reconstruction of a $512\times 512$ test image from 35\% of its noise corrupted Fourier coefficients, chosen uniformly at random. \label{fig:peppers_noise}}
\end{figure}
\subsection{Stability}\label{sec:num_stab}
The theoretical results of Section \ref{sec:main_1} demonstrated that uniform random sampling maps are stable to inexact sparsity. To recap, for the recovery of elements of $\bbC^N$ up to sparsity level $s$, the following two sampling strategies were considered in Theorem \ref{thm:unif_samp} and Theorem \ref{thm:near_optimal}  respectively.
\begin{itemize}
\item[(S1)] Draw $\ord{s\log N}$ samples uniformly at random.
\item[(S2)] Draw $\ord{s\log N}$ samples  uniformly at random. Then, add another $\ord{s\log N}$ i.i.d. samples such that 
$$
\bbP(X_j = n) = p(n), \quad p(n) = C\left( \log(N)\max\br{1, \abs{n}}\right)^{-1}, \quad n=-N/2+1,\ldots, N/2,
$$
where $X_j$ indexes the $j^{th}$ samples drawn in this second phase of sampling and $C$ is such that $p$ is a probability measure.
\end{itemize}
As discussed, the provable error bounds obtained with sampling strategy (S1) are suboptimal, whereas, by adding the samples which are chosen in accordance with the nonuniform distribution in (S2), one can guarantee near-optimal error bounds. A natural question to ask is whether sampling  in accordance with a nonuniform distribution actually improves stability, or whether the improved stability between the theorems of Section \ref{sec:near_optimal} and Section \ref{sec:main_1} is simply an artefact of the proofs in this paper. To empirically address this question,  consider  the following experiment.

Given $N\in\bbN$, a gradient sparse vector $x\in\bbR^N$ and an inexact sparsity level $S$, let us perturb $x$  by a randomly generated vector $h\in\bbR^N$, where $h$ is such that  $S = 10\log_{10}\left(\nm{x}_2/\nm{h}_2\right)$. Note that this is the SNR of $h$ relative to $x$, and smaller values of $S$ represent larger magnitudes of perturbations.  We now consider the reconstruction of the approximately sparse signal $x+h$ from $\rP_\Omega \rA (x+h)$. The sampling set $\Omega$ will be such that its cardinality is $\lceil 0.15 N\rceil$, and it is either chosen uniformly at random (as described in Theorem \ref{thm:unif_samp}), which we will denote by $\Omega_U$;  or as a union of a uniform random sampling set and a variable density sampling set (as described in Theorem \ref{thm:near_optimal}), which we will denote by $\Omega_P$. 

We performed this experiment for perturbations of two  sparse signals, shown in Figure \ref{fig:stab_signals}, and the relative errors of reconstructing  the approximately sparse versions of theses signals via solving  (\ref{eq:noise_tv}) with $\delta = 0$ (since we are investigating stability rather than robustness) are shown in Figure \ref{fig:graphs_stab}. Observe that  both samplings with $\Omega_P$ and $\Omega_U$ exhibit stability with respect to inexact sparsity, since the relative errors all decay as the SNR values increase. However, the relative errors obtained when sampling with $\Omega_P$ are much lower, suggesting that one of the benefits offered by dense sampling around the zero frequency is increased stability. Finally, it is perhaps interesting to note that the results of both Theorem \ref{thm:near_optimal} and Theorem \ref{thm:unif_samp} guarantee optimal error bounds (up to $\log$ factors) on the \textit{recovered gradient}, and Figure \ref{fig:stab_gradient} confirms this result by showing that there is no substantial difference between the error on the recovered gradient between $\Omega_U$ and $\Omega_P$. This is illustrated in Figure \ref{fig:grad_vs_signal} which shows the recovered signals and recovered gradients of this experiment when $x_1$ in Figure \ref{fig:stab_signals} has been perturbed by $h$ with an SNR of 17. Note that while reconstruction obtained via (S2) is clearly superior to the reconstruction obtained via (S1), the difference in the quality of the recovered gradients is far less substantial.  So, experimentally, it appears as though dense sampling at low frequencies will significantly improve the stability of the recovered signal, although not the stability of the recovered gradient.

This improvement in stability is particularly visible in two dimensions - consider the recovery of the $256\times 256$ test image shown in Figure \ref{fig:peppers_reconsr}. This figure shows the reconstruction from when the sampling set is the two dimensional analogue of either (S1) or (S2) (as described in Theorem \ref{cor:unif_samp_2D_opt}). The reconstructions are obtained by solving (\ref{eq:tvcs_2d}) with $\delta = 0$, so we consider only the sparsity stability rather than  noise robustness. The 
improvement in reconstruction quality is substantial, and, as suggested by this section and our theoretical result, one possible reason for this is that additional samples at low frequencies are required for optimal stability.

\begin{figure} 
\begin{center}
\begin{tabular}{@{\hspace{-12pt}}c@{\hspace{-12pt}}c@{\hspace{-12pt}}c}
\includegraphics[width = 0.33\textwidth, trim=1.2cm 1.2cm 1.2cm 0cm, clip=true]{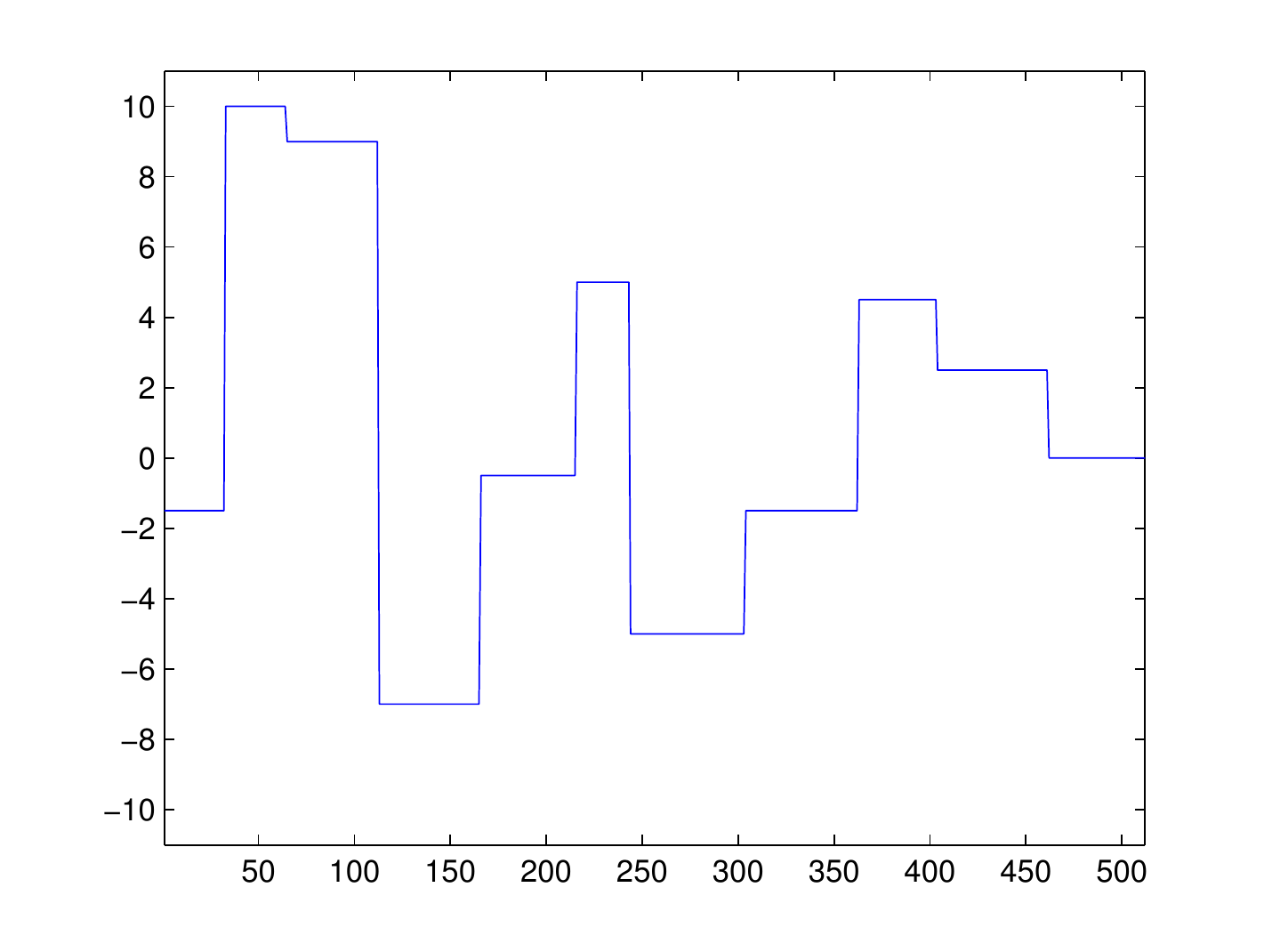}
\includegraphics[width = 0.33\textwidth, trim=1.2cm 1.2cm 1.2cm 0cm, clip=true]{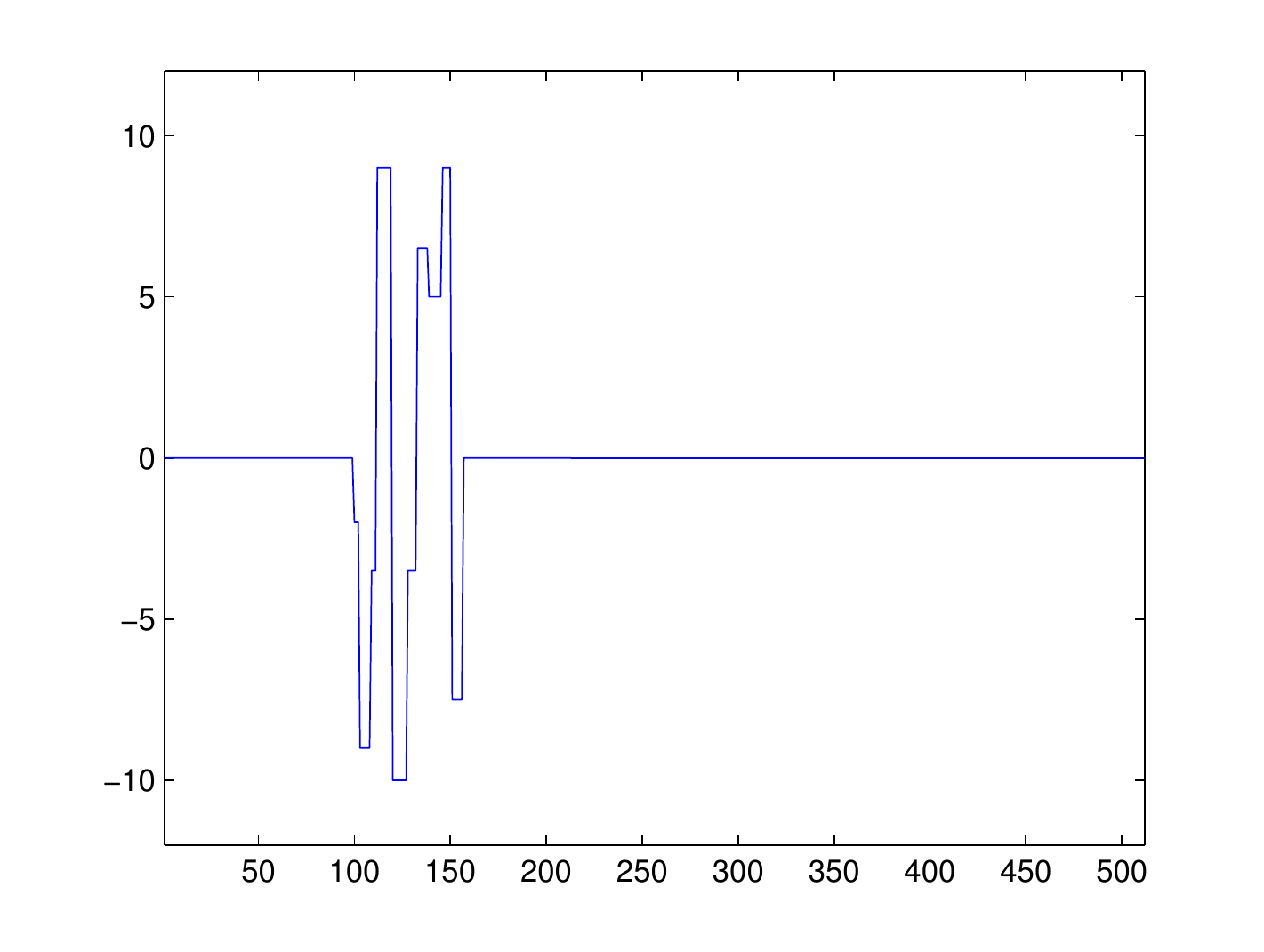}
\includegraphics[width = 0.33\textwidth, trim=1.2cm 1.2cm 1.2cm 0cm, clip=true]{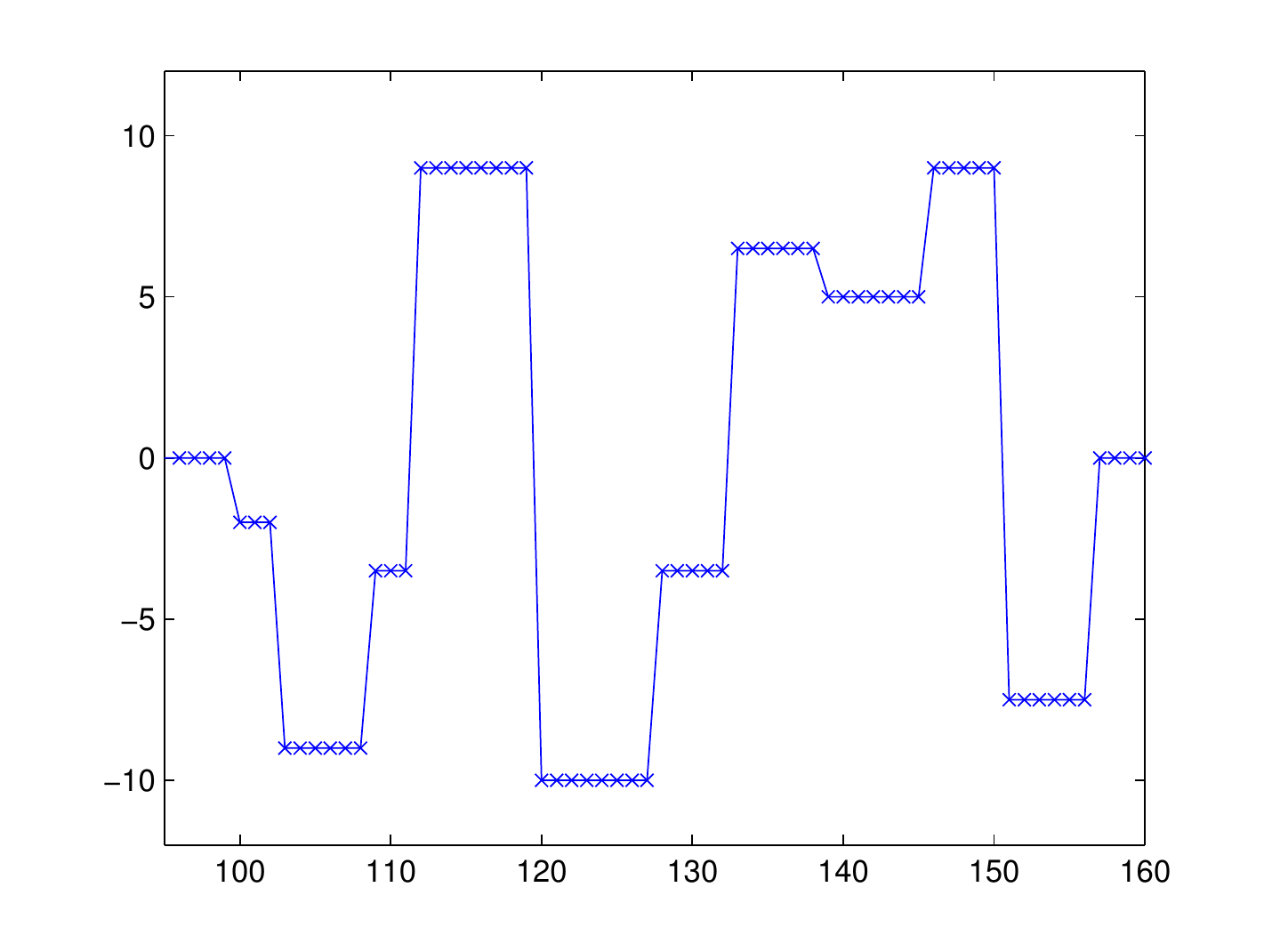}
\end{tabular}
\end{center}
\caption{The coarse signal, $x_1$, which will be perturbed (left). The fine signal $x_2$ which will be perturbed (centre), and a zoom  of $x_2$ (right) on indices between  90 and 160. For clarity, the values of $x_2$ on each index is marked by a cross. \label{fig:stab_signals}}
\end{figure}

\begin{figure} 
\begin{center}
\begin{tabular}{@{\hspace{-12pt}}c@{\hspace{-12pt}}c@{\hspace{-12pt}}c}
\includegraphics[width = 0.4\textwidth]{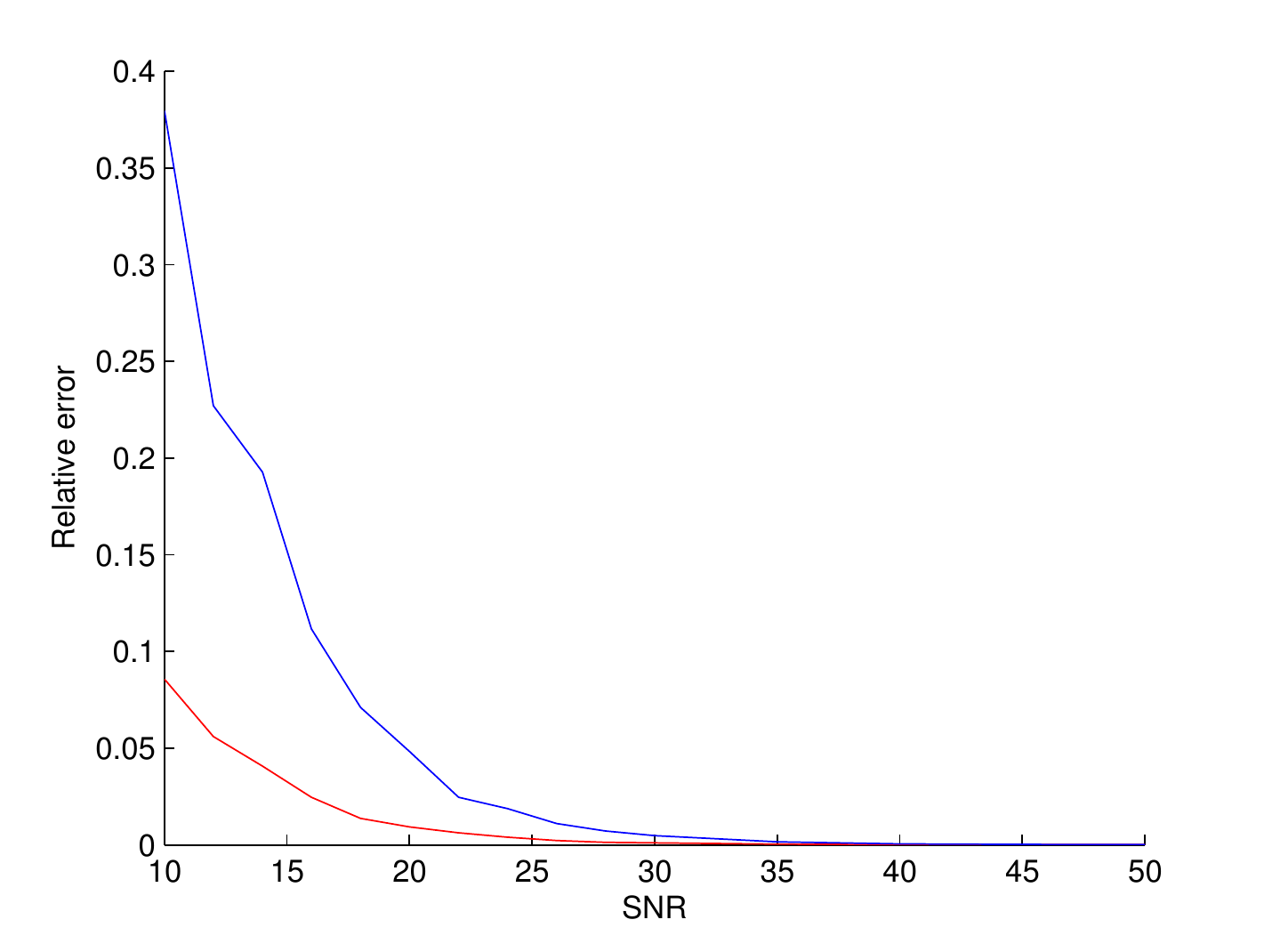}
\includegraphics[width = 0.4\textwidth]{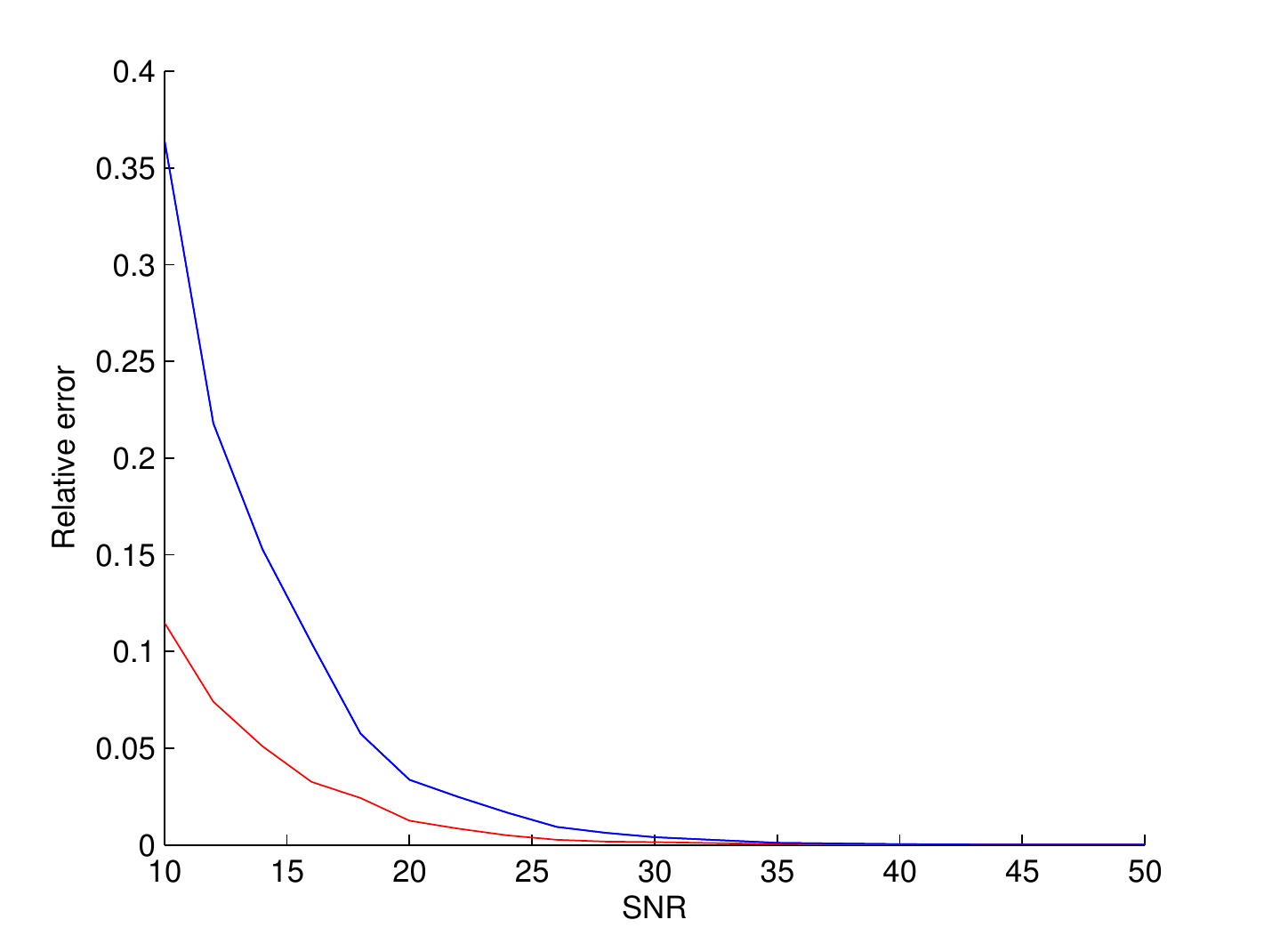}
\end{tabular}
\end{center}
\caption{Left:  plot of relative error $\nm{\hat x - (x+h)}_2/\nm{x+h}_2$ against $\mathrm{SNR} = 10\log_{10}\left(\nm{x}_2/\nm{h}_2\right)$, where $x:=x_1$ is shown in Figure \ref{fig:stab_signals} and $\hat x$ is the reconstruction.  The blue line refers to choosing $\Omega:=\Omega_U$. The red line refers to choosing $\Omega :=\Omega_P$. Right: the equivalent plot for $x:=x_2$. \label{fig:graphs_stab}}
\end{figure}

\begin{figure} 
\begin{center}
\includegraphics[width = 0.4\textwidth]{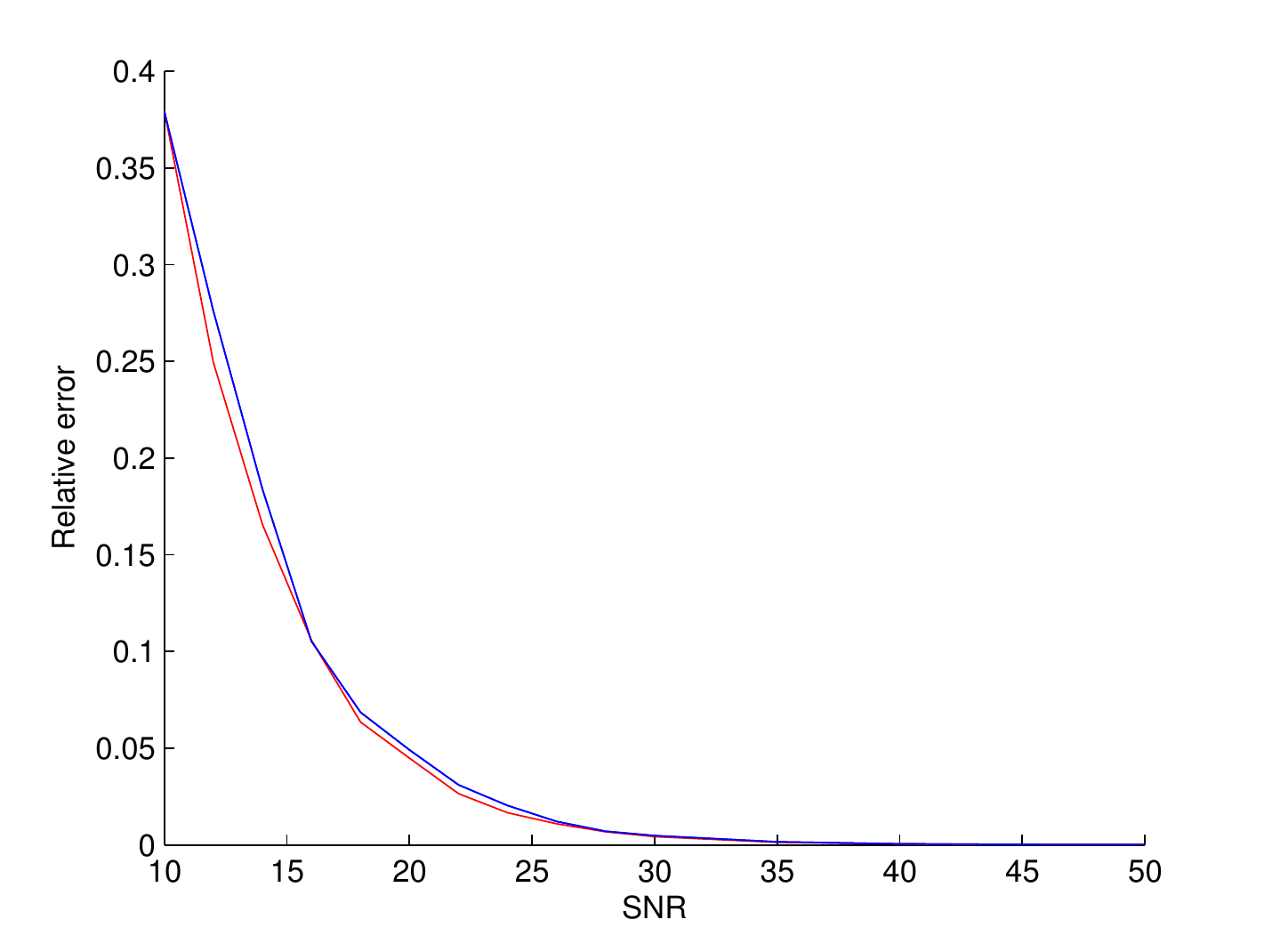}
\includegraphics[width = 0.4\textwidth]{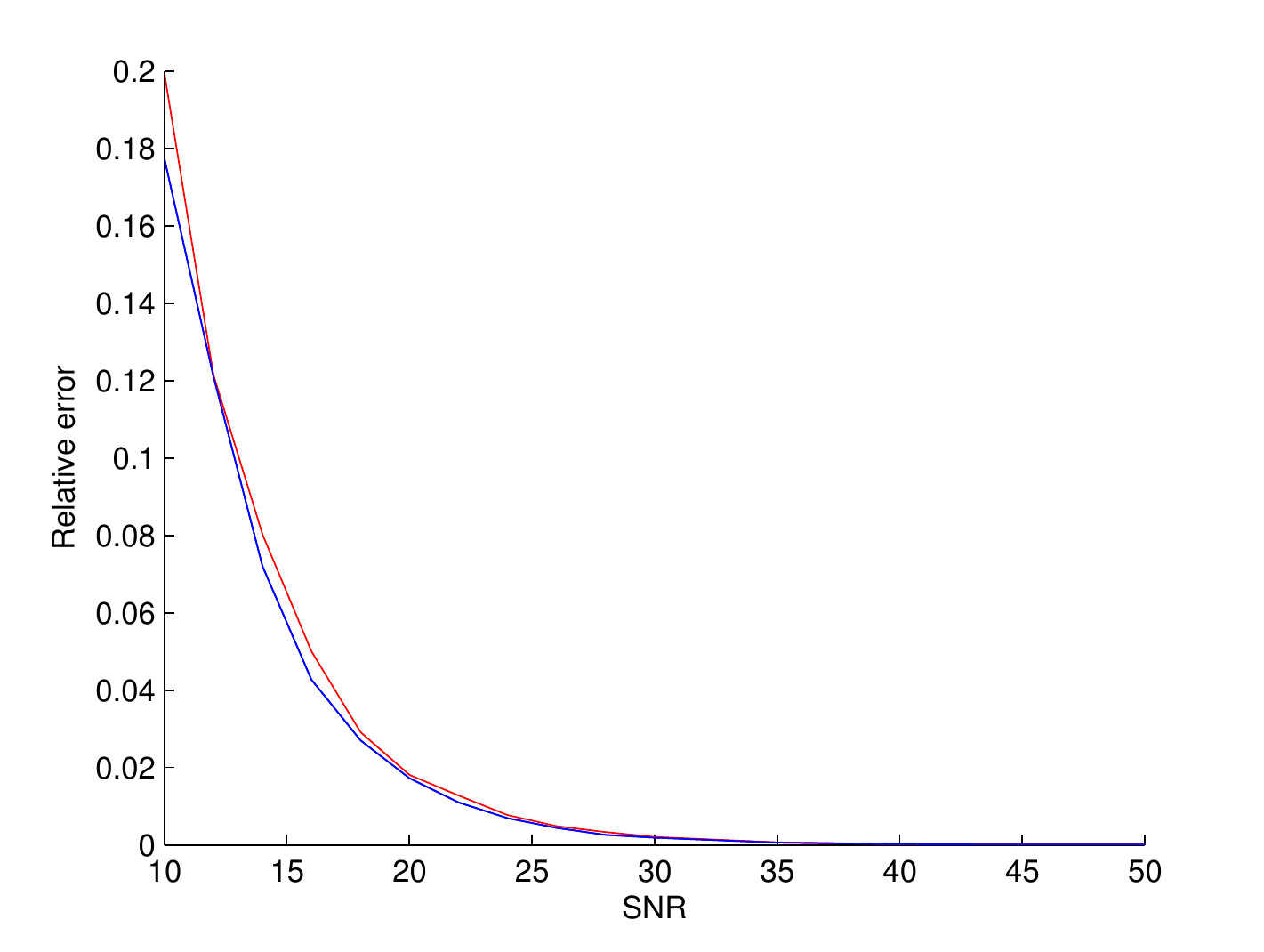}
\end{center}
\caption{Left: y-axis shows the relative error of the recovered gradient: $\nm{\rD(x+h-\hat x)}_2/\nm{\rD (x+h)}_2$, where $\hat x$ is the recovered signal, $x :=x_1$ is the sparse signal shown in Figure \ref{fig:stab_signals}, and $h$ is the perturbation. The x-axis shows the SNR of $h$ relative to $x$. The blue line corresponds to sampling uniformly at random, and the red line corresponds to uniform plus variable sampling of Theorem \ref{thm:near_optimal}. Right: same as the left graph, except that $x:=x_2$, where $x_2$ is as shown in Figure \ref{fig:stab_signals}. \label{fig:stab_gradient}}

\end{figure}

\begin{figure} 
\begin{center}
\begin{tabular}{@{\hspace{-12pt}}c@{\hspace{-12pt}}c@{\hspace{-12pt}}}
\includegraphics[width = 0.45\textwidth]{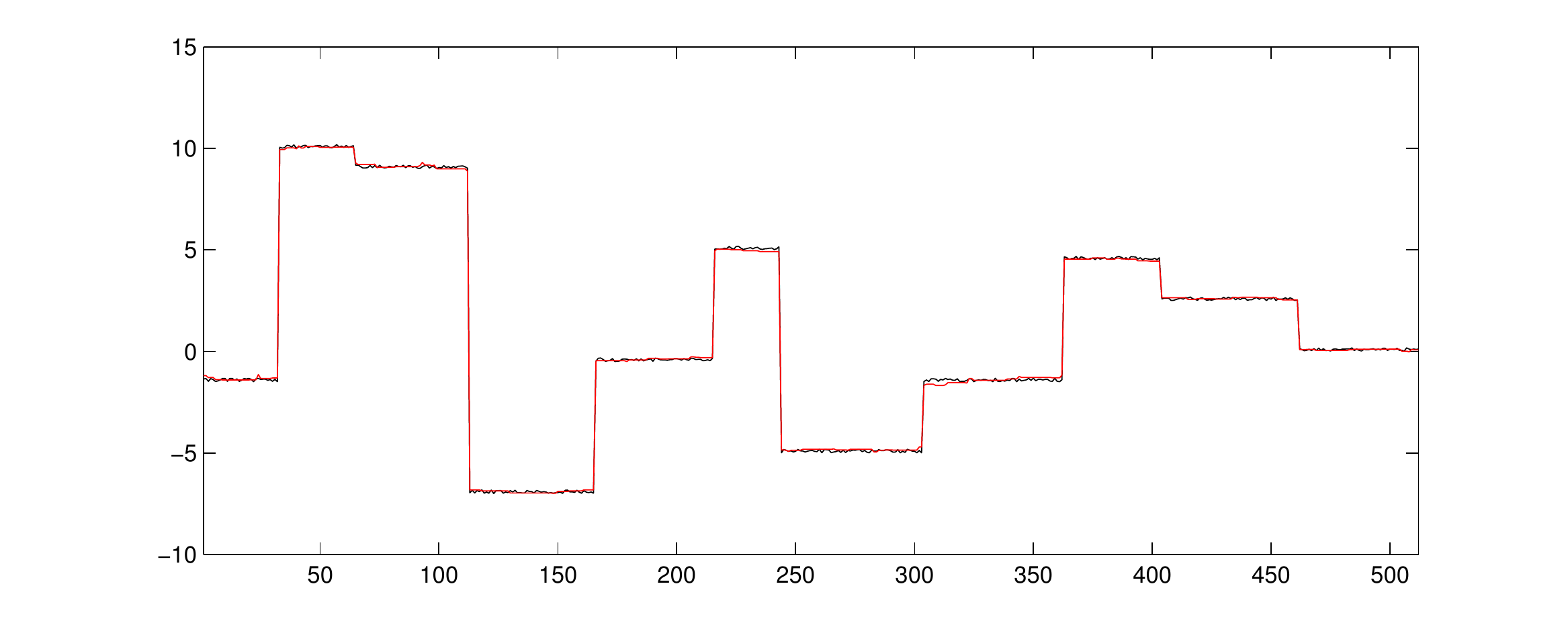}&
\includegraphics[width = 0.45\textwidth]{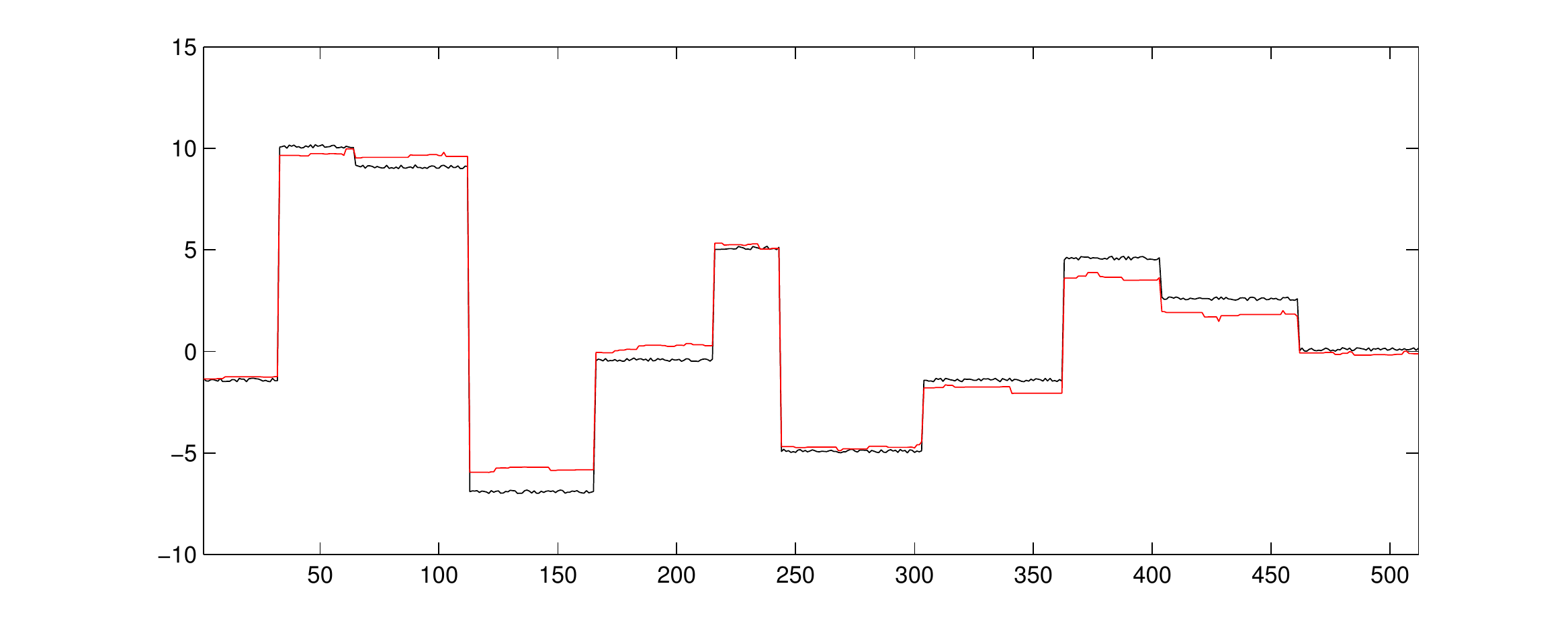}\\
\includegraphics[width = 0.45\textwidth]{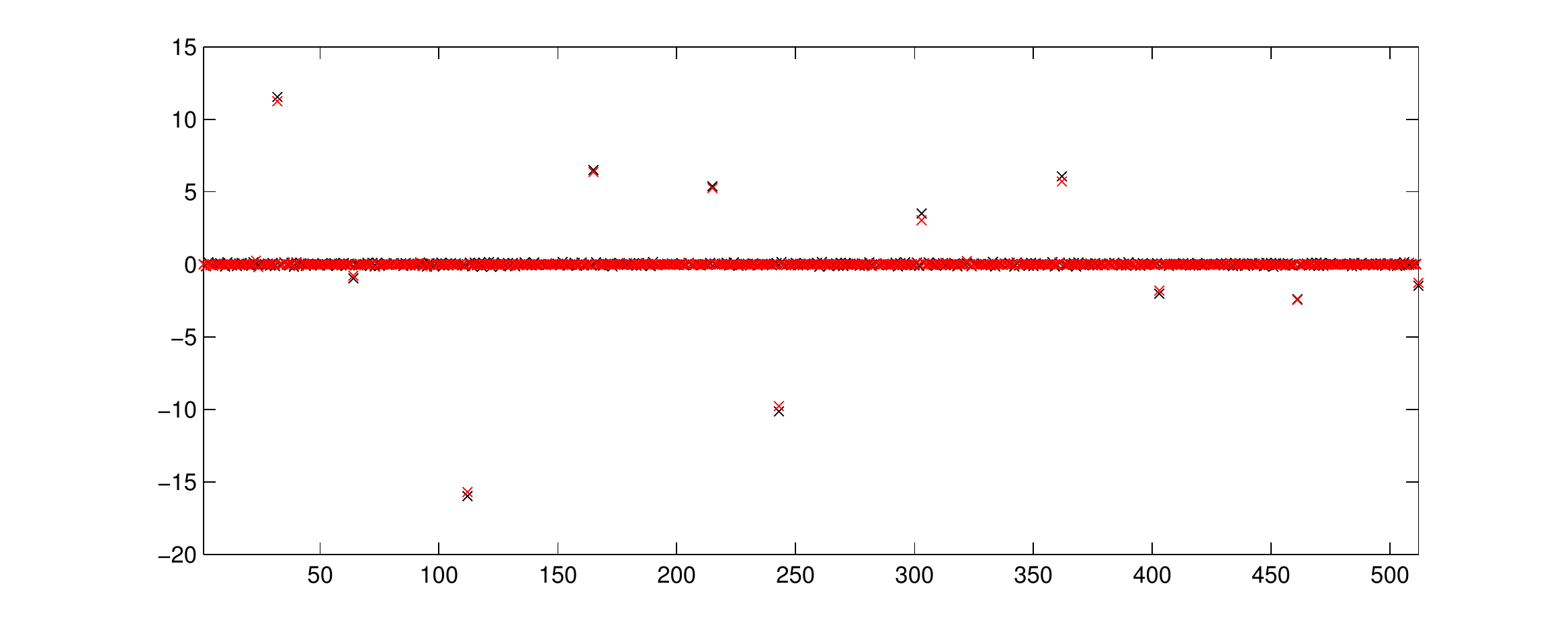}&
\includegraphics[width = 0.45\textwidth]{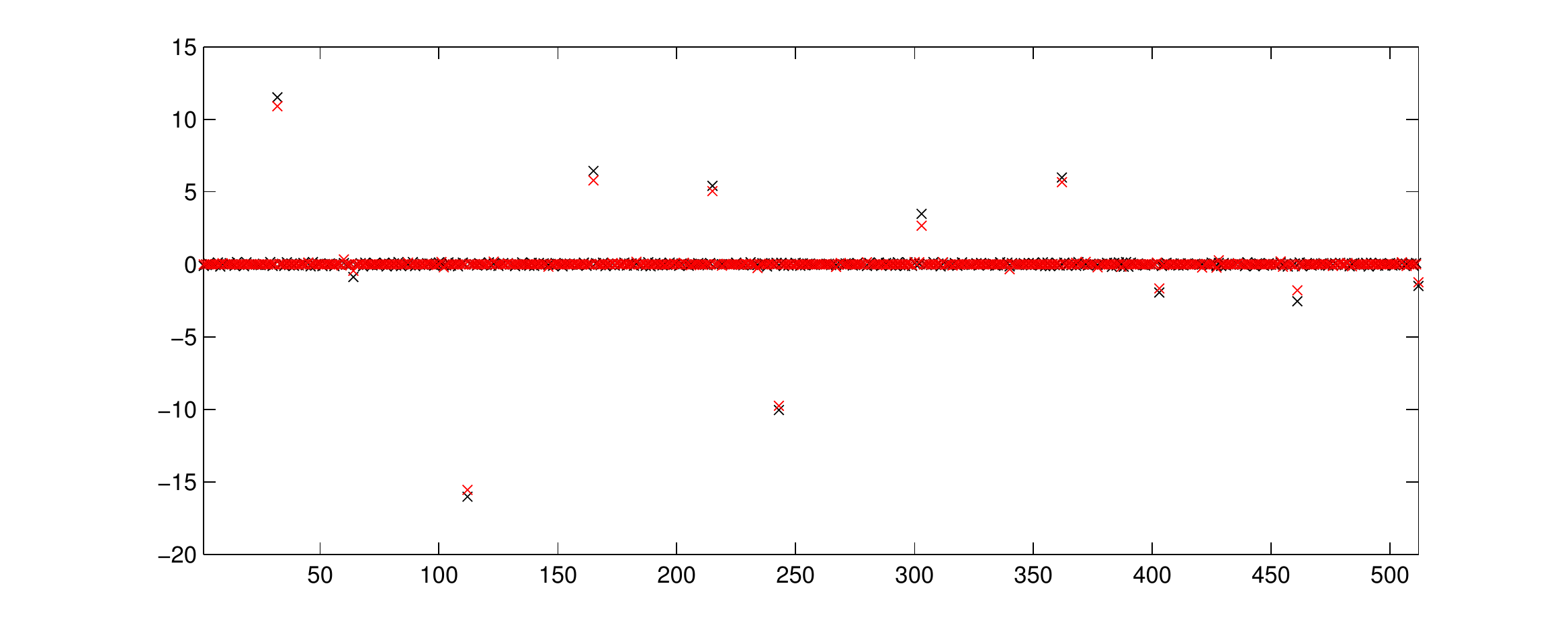}
\end{tabular}
\end{center}
\caption{Top left: the approximately sparse signal to be recovered (in black) and the signal recovered from sampling in accordance with (S2). Top right: the gradient of the approximately sparse signal to be recovered (in black) and the gradient  of the signal recovered from sampling in accordance with (S2) (in red). Each value in the gradient vectors is marked with `x' for clarity.  Bottom left: the approximately sparse signal to be recovered (in black) and the signal recovered from sampling in accordance with (S1) (in red). Bottom right: the gradient of the approximately sparse signal to be recovered (in black) and the gradient  of the signal recovered from sampling in accordance with (S1) (in red). Each value in the gradient vectors is marked with `x' for clarity. \label{fig:grad_vs_signal}}
\end{figure}

\begin{figure}
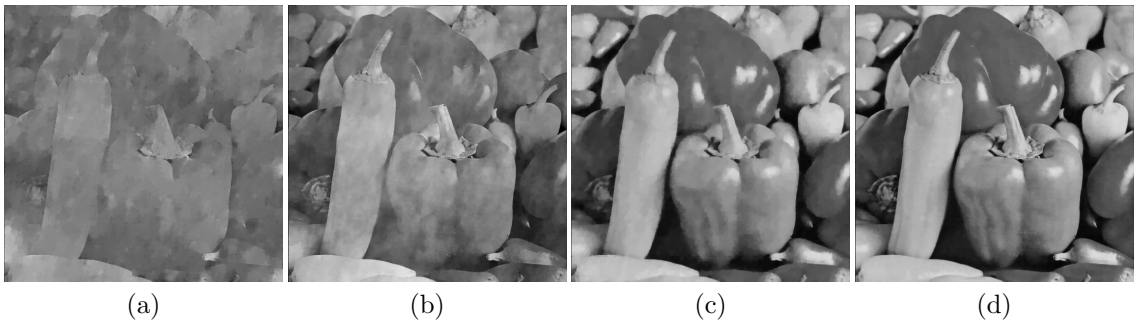
 
\begin{center}
\begin{tabular}{c@{\hspace{2pt}}c@{\hspace{2pt}}c@{\hspace{2pt}}c}
\includegraphics[width = 0.23\textwidth]{{{unifrand1_0.3318}.png}}&
\includegraphics[width = 0.23\textwidth]{{{unifrand5_0.2465}.png}}&
\includegraphics[width = 0.23\textwidth]{{{powerunifrand1_0.0683}.png}}&
\includegraphics[width = 0.23\textwidth]{{{powerunifrand5_0.0385}.png}}\\
(a) &(b) &(c) &(d)
\end{tabular}
\end{center}
\caption{(a) \& (b): reconstructions from sampling uniformly at random (plus the zero frequency sample); (a) is from sampling at 10\%, with a relative error of 0.3318; (b) is from sampling at 30\% samples with a relative error of 0.2465. (c) \& (d): reconstruction from  sampling in the semiuniform random manner described in Theorem \ref{cor:unif_samp_2D_opt}; (c) is from sampling at 10\%  with a relative error of 0.0683l (d) is from sampling at 30\% with a relative error of 0.0385.\label{fig:peppers_reconsr}}
\end{figure}

\subsection{The price of randomness}\label{sec:price}

The previous section demonstrated that one of the benefits of dense sampling at low frequencies is increased stability to inexact sparsity. This section will demonstrate that another benefit is the recovery of certain types of signals up to sparsity level $s$ from sub-$\ord{s\log N}$ samples. More specifically, although $\ord{s\log N}$ is the optimal sampling cardinality for $s$-sparse signals \cite{candes2006robust}, and this can be attained through drawing samples uniformly at random,  when one is interested in a subset of the possible $s$-sparse signals (e.g. signals whose discontinuities are sufficiently far apart), it may be unnecessary to pay the price of this $\log$ factor.  

\subsubsection{Sampling high frequencies is not necessary for coarse signals}
Theorem \ref{thm:unif_samp} and Theorem \ref{thm:min_sep_thm} offer additional insight into why dense sampling at low frequencies can outperform uniform random sampling.  Theorem \ref{thm:unif_samp} tells us that one can recover a gradient $s$-sparse signal from $\ord{s\log N}$ samples by choosing the samples uniformly at random regardless of where the nonzero gradient entries occur. Moreover, such a recovery is stable to inexact sparsity and robust to noise. However, under an additional assumption that the separation of the nonzero gradient entries is at least $1/M$, Theorem \ref{thm:min_sep_thm} stipulates that we can sample uniformly at random from the first $4M$ samples at a slightly smaller sampling order of $\ord{s \log(M) \log(s)}$ (although the provable stability and robustness bounds are worse by a factor of $\sqrt{s}$, where $s$ is the approximate sparsity). This first suggests that an understanding of the gradient structure of the underlying signal can lead to sampling patterns which will outperform uniform random sampling. Second, in order to recover an $s$-sparse signal of length $N$, one requires $\ord{s\log(N)}$ random samples and this sampling cardinality is sharp for sparse signals \cite{candes2006robust}. Thus, although such a statement guarantees the recovery of any $s$-sparse signal, there is a price of $\log(N)$ associated with the randomness introduced. However, suppose that our signal of interest (denote by $x$) is  of length $N$ and is $M$-sparse in its gradient and that these nonzero gradient entries have minimum separation of $1/M$. Then, Theorem \ref{thm:min_sep_thm} tells us that $x$ can be perfectly recovered from its first $4M+1$ Fourier samples of lowest  frequencies. Note that there is no randomness in the choice of sampling set $\Omega$, and the cardinality of $\Omega$ is linear with respect to sparsity. Observe also that a uniform random choice of $\Omega$ is guaranteed to result in  accurate reconstructions and allow for significant subsampling only if  $M\log(N)\ll N$. So in the case that $M \geq N/\log(N)$, it will be better to choose $\Omega$ to index the first $M$ samples, rather than draw the samples uniformly at random. 
 
 \subsubsection{A numerical example}
To illustrate the remarks above, consider the recovery of $x_1$ of length $N=512$ shown on the left of Figure \ref{fig:signals}. It can be perfectly recovered by solving the following minimization problem with $\Omega$ indexing the first 20 Fourier frequencies.  This accounts for $3.9\%$ of the available Fourier coefficients. For simplicity, we will present this experiment without adding noise to the samples, although similar results can be observed if noise is added. 
\begin{equation*}
\min_{z \in \mathbb{C}^{N}} \norm{z}_{TV} \text{ subject to }  \rP_\Omega \rA z = \rP_\Omega \rA x_1.
\end{equation*}
 The result of repeating this experiment over 5 trials with $\Omega$ taken to be 3.9\%, 7\%, 10\% of the  available indices, drawn uniformly at random, is shown in Table \ref{t:signal_unif}. By sampling uniformly at random, we cannot achieve exact recovery from drawing only $3.9\%$ and it is only when we sample at 10\% that we obtain exact recovery across all 5 trials.

 \begin{table} 
\begin{center}
{\small{
\begin{tabular}{| c| c| c| c|c|}

\hline
 
Trial	&		 Sampling	&	 Sampling	&	Sampling	&	 Sampling\\	
	&	 3.9\%	&	 7\%	&	9\%	& 10\%	\\\hline
1	&	0.9096	&	0.2150	&	0& 0	\\
2	&	0.7739	&	0.1915	&	0& 0	\\
3	&	0.4388	&	0		&	0.1132& 0	\\
4	&	0.7287	&	0.4396	&	0.1603& 0	\\
5	&	0.7534	&	0.3044	&	0& 0	\\\hline
\end{tabular}
   }}
  \end{center}
   \caption{Relative error of reconstructions obtained by sampling the Fourier transform of Signal 1 uniformly at random.}
\label{t:signal_unif}
 \end{table}

 \subsubsection{The need for further investigation: Structured sampling}
 To conclude, we present a numerical example to show that despite the advances in the theoretical understanding of total variation regularization for compressed sensing, there is still room for substantial improvement. Consider the following reconstruction of the resolution chart of size $528\times 500$ in Figure \ref{fig:1d_recons} from $6.5\%$ of its available Fourier coefficients using different sampling maps:
 \begin{itemize}
 \item[(i)] (Uniform sampling) $\Omega_U$ indexes  samples drawn uniformly at random,
 \item[(ii)] (Low frequency sampling) $\Omega_L$  indexes the samples of lowest Fourier frequencies, 
 \item[(iii)] (Uniform $+$ power law) $\Omega_P$  is chosen in accordance with Theorem \ref{cor:unif_samp_2D_opt}, 
 \item[(iv)] (Multilevel sampling) $\Omega_V = \Omega_{V,1}\cup \Omega_{V,2}$,  is  constructed such that $\Omega_{V,1}$ indexes all samples with frequencies no greater than $5$,  and $\Omega_{V,2}$ is constructed by first dividing up the available indices into $L$ levels in  increasing order of frequency such that $\bbP(X_j = k) = C \cdot \exp(-(bn/L)^a)$ for some appropriate constant $C$ such that we have a probability measure, $X_j$ is the $j^{th}$ element of $\Omega_{V,2}$ and $k$ belongs to the $L^{th}$ level. In this experiment, we chose $L=25$, $a=2.2$, and $b=6.5$.
 \end{itemize}

\subsubsection*{Conclusion of the experiment} 

 The following observations can be made from the sampling maps and the reconstructions    shown in Figure \ref{fig:chart_test}. 
 \begin{itemize}
\item[(i)] (Uniform sampling) Uniform random sampling yields a high relative error.
\item[(ii)] (Low frequency sampling) Sampling only the low Fourier frequencies recovers only the coarse details.
\item[(iii)] (Uniform $+$ power law) Concentrating on low Fourier frequencies but also sampling high Fourier frequencies allowed for the recovery of both the coarse and fine details.
 \item[(iv)] (Multilevel sampling)  Similarly to (iii), this allowed for the recovery of both the coarse and fine details, but the reconstruction is substantially better than that of the uniform $+$ power law.
\end{itemize}
  
 So, uniform random sampling maps are applicable only in the case of extreme sparsity due to the price of a $\log$ factor, while either fully sampling  or subsampling the low frequencies will be applicable when we aim to only recover low resolution components of the underlying signal. This suggests that variable density sampling patterns are successful because they accommodate for a combination of these two scenarios -- when there are both high and low resolution components which we want to recover and some sparsity -- sampling fully at the low frequencies will allow for recovery of coarse details without the price of a $\log$ factor, while increasingly subsampling at high frequencies will allow for the recovery of fine details up to a $\log$ factor. 
 One can essentially repeat this experiment for any natural image to observe the same phenomenon: by choosing the samples uniformly at random, we will be required to sample more than is necessary. 

Note that the theoretical results (Theorem \ref{thm:min_sep_thm}) of this paper provide only a very basic understanding of how the distribution of the Fourier coefficients favours the recovery of certain types of signals and can allow for sub-$\ord{s\log N}$ samples. However, Figure \ref{fig:chart_test} suggests that there exists a much deeper connection between the gradient sparsity structure of a signal and the distribution of the Fourier samples, and a thorough understanding of this connection could lead to more efficient sampling strategies.

\begin{figure}  
\centering
\includegraphics[width=0.42\textwidth]{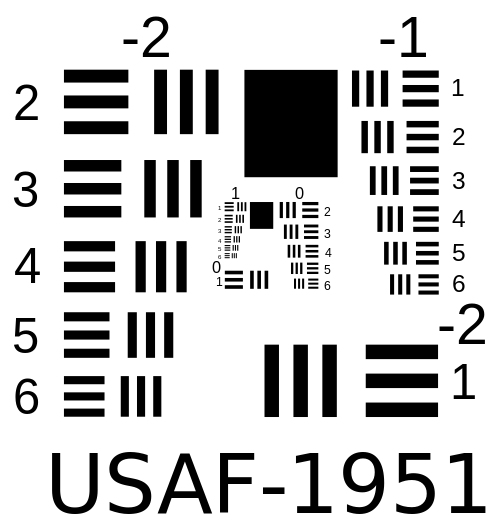}
\caption{The 1951 USAF resolution test chart of size $528\times 500$.
\label{fig:1d_recons}}
\end{figure}

\begin{figure}  
\begin{center}
\small
$\begin{array}{cccc}
 \Omega_{U}& \Omega_{L} &  \Omega_{V} & \Omega_{P} \\
 \includegraphics[ width=0.22\textwidth]{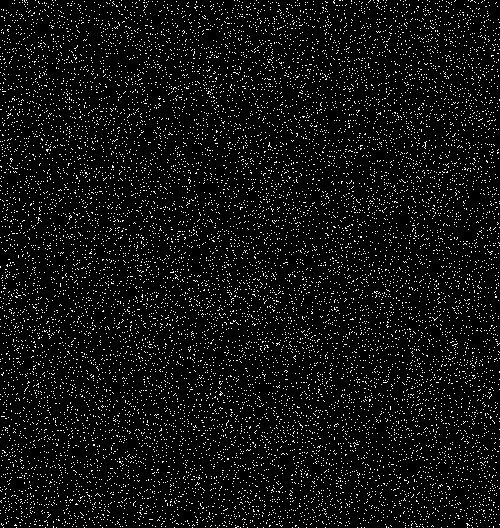}&
\includegraphics[ width=0.22\textwidth]{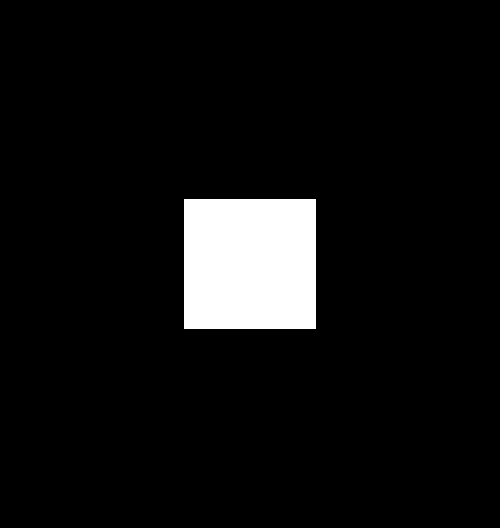} &
\includegraphics[ width=0.22\textwidth]{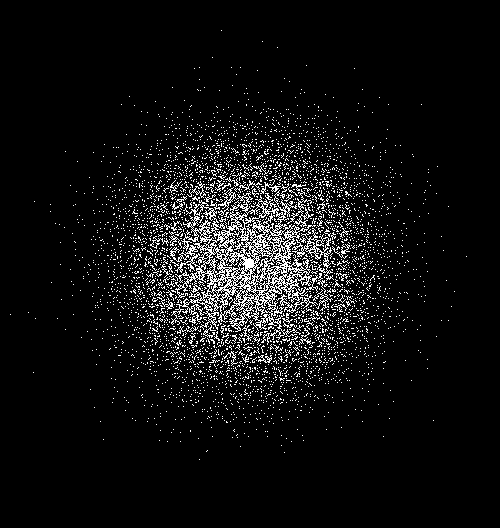}&
\includegraphics[ width=0.22\textwidth]{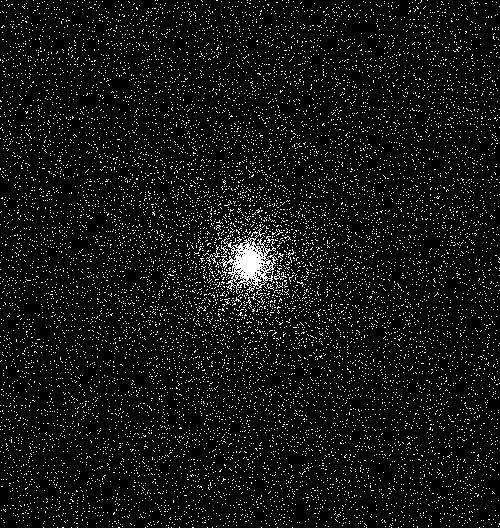} 
\\\\
\text{Reconstruction from } \Omega_{U}&
\text{Reconstruction from } \Omega_{L} &
\text{Reconstruction from }\Omega_{V} &\text{Reconstruction from } \Omega_{P}\\
\epsilon_{rel} = 0.3909& \epsilon_{rel} = 0.0698&
\epsilon_{rel} = 0.0273 & \epsilon_{rel} = 0.0940\\
\includegraphics[ width=0.22\textwidth]{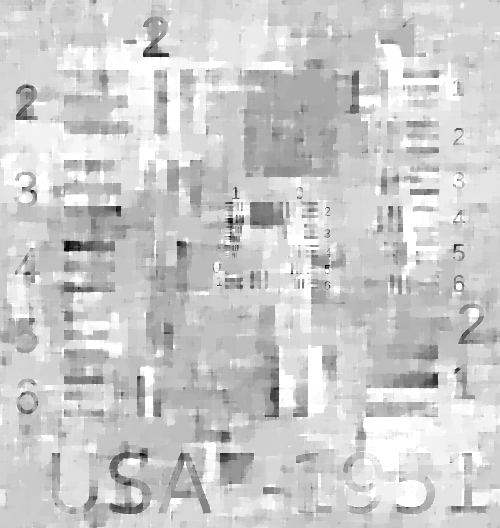}&
\includegraphics[ width=0.22\textwidth]{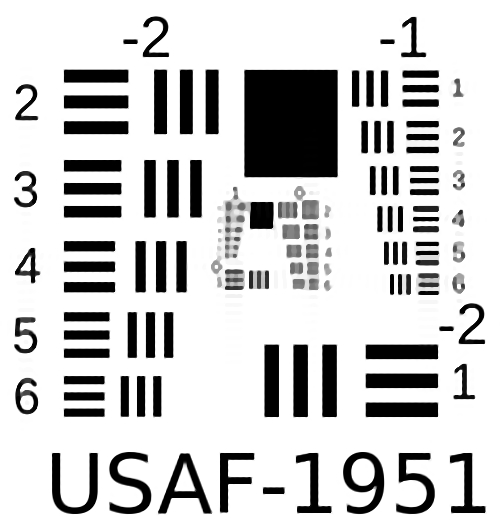} &
\includegraphics[ width=0.22\textwidth]{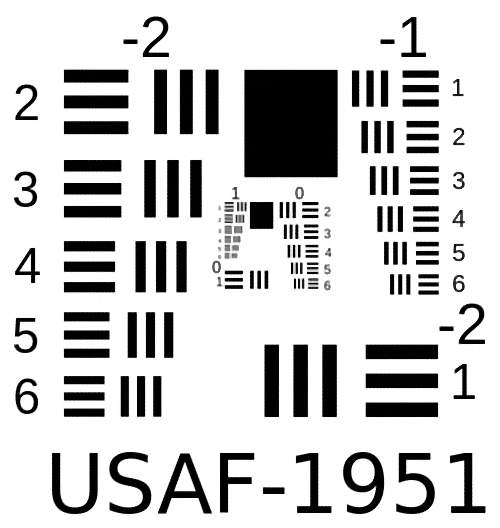} &
\includegraphics[ width=0.22\textwidth]{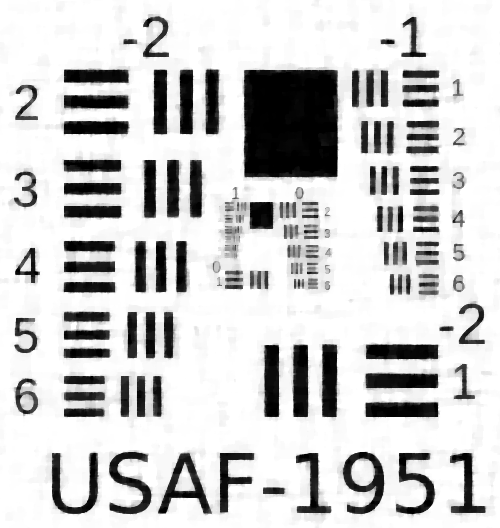} 
\\
\includegraphics[ width=0.22\textwidth]{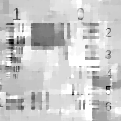}&
\includegraphics[ width=0.22\textwidth]{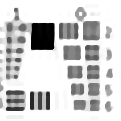} &
\includegraphics[ width=0.22\textwidth]{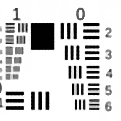} &
\includegraphics[ width=0.22\textwidth]{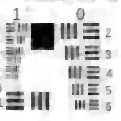} 
\end{array}$
\caption{
The top row shows the Fourier sampling maps, each indexing 6.5\% of the available Fourier samples. The middle row shows the reconstructed images and their relative errors.  The bottom row zooms in on  the reconstructions  for pixels in $[180,300]\times[220, 340]$.
\label{fig:chart_test}}
\end{center}
\end{figure}

\subsubsection*{Relationship to the recovery of wavelet coefficients from Fourier samples}
 Theorem \ref{thm:min_sep_thm} shows that under an additional assumption on the minimum separation distance on the sparsity structure of the underlying signal, we may draw the samples from only samples of low Fourier frequencies. Furthermore, the last sentence of Theorem \ref{thm:min_sep_thm} implies that if  the underlying signal has gradient support $\Delta$ with $\nu_{\min}(\Delta, N) = M^{-1}$, then the number of samples required for perfect recovery is $\ord{M}$, and in this case, there is no probability or $\log$ factor involved. As explained in this section, this explains why sampling only low frequency coefficients can lead to superior reconstruction quality when compared with sampling at random. 
This result is reminiscent of the result from \cite{gs_l1} which shows that the first $M$ Fourier coefficients of lowest frequencies will stably recover the first $cM$ wavelet coefficients  of lowest dilation factors for some  constant $c\leq 1$. Furthermore, this recovery can be achieved by solving an appropriate $\ell^1$ minimization problem. Note that this is a linear relationship between the number of samples and the number of recovered wavelet coefficients and there is no $\log$ factor involved. The recovery of wavelet coefficients from Fourier samples is also another application of compressed sensing in which variable density sampling patterns are preferred over uniform random sampling patterns. The work of \cite{adcockbreaking}  provides analysis to explain this phenomenon and shows that the strength of variable density sampling patterns exists because they combine the linear relationship between the Fourier samples of lowest frequencies and the wavelets of lowest dilation factors, and exploit the benefits of randomness at the expense  of $\log$ factors in the sampling cardinality.

\section{Proofs}

Throughout this section, given $a,b\in\bbR$, $a \lesssim b$ denotes $a \leq C \cdot b$ for some numerical constant $C$ and $a\gtrsim b$ denotes $a\geq C \cdot b$ for some numerical constant $C$. Given $x\in\bbC^N$, $\sgn(x) \in \bbC^N$ is such that $\sgn(x)_j = x_j/\abs{x_j}$ if $x_j\neq 0$ and $\sgn(x)_j = 0$ otherwise. Also, for $j\in\bbZ$, $e_j$ is the vector whose $j^{th}$ entry is 1 and is zero elsewhere. Given $\alpha\in\bbC^N$, let $\mathrm{diag}(\alpha)$ denote the diagonal matrix whose diagonal is $\alpha$.

First, by standard compactness arguments, a minimizer to (\ref{eq:noise_tv}) necessarily exists for any choice of $\Omega$, so we will henceforth derive error bounds given a minimizer of (\ref{eq:noise_tv}). To outline our proof strategy, we begin with a brief description of a typical compressed sensing approach towards establishing recovery guarantees.
In order to show that one can guarantee the recovery of a signal $x\in\bbC^N$ in a manner which is stable to noisy observations $y = \rP_\Omega \rA x + \eta$ with $\nm{\eta}_2\leq \delta$  and inexact sparsity (up to some  $\Delta \subset \br{1,\ldots, N}$) by solving
\be{\label{eq:cs_typical}
\min_{z\in\bbC^N} \nm{z}_1 \text{ subject to } \nm{\rP_\Omega \rA z - y}_2 \leq \delta,
}
standard compressed sensing arguments carry out the following two steps.
\begin{enumerate}
\item[(i)] Show that stable and robust recovery holds provided that $\rP_\Omega \rA\rP_\Delta$ is close to an isometry, and there exists some $\rho = \rA^* \rP_\Omega w$ (often referred to as the dual certificate) satisfying $\nm{\rP_\Delta \rho - \rP_\Delta \sgn(x)}_2 \leq \beta_1$, $\nm{\rP_\Delta^\perp \rho}_\infty \leq \beta_2$ and $\nm{w}_2 \leq \sqrt{s}\cdot \gamma$ for some appropriate $\beta_1, \beta_2 \in (0,1)$ and $\gamma>0$.
\item[(ii)] Derive conditions on the sampling index set $\Omega$ such that the sufficient conditions from (i) are satisfied, i.e. $\rP_\Omega \rA \rP_\Delta$ is close to an isometry and the dual certificate can be constructed.
\end{enumerate}
The minimization problem which we are interested in is slightly different from (\ref{eq:cs_typical}) since we minimize  $\nm{\rD \cdot}_1$ rather than $\nm{\cdot}_1$, nonetheless, the proofs of Theorems \ref{thm:unif_samp} and \ref{thm:min_sep_thm} will follow the strategy outlined above. In particular, we will first show that stable gradient recovery implies stable signal recovery then follow the above procedure to show that stable gradient recovery can be achieved under the hypotheses of our theorems.

Note that the difficulty in proving recovery estimates is often in the construction of the dual certificate in step (ii) which, in our case, where $\rA$ is the discrete Fourier transform, is the problem of constructing  a trigonometric polynomial $f(t) = \sum_{j\in \Omega} w_j e^{2\pi i t j}$ such that $f$ almost interpolates $\rP_\Delta(\sgn(\rD x))$ on $\tilde \Delta = \br{j/N: j\in\Delta}$ and such that the absolute value of $f$ is sufficiently small away from the support set $\tilde \Delta$. However, we will exploit some existing constructions from the study of (\ref{eq:cs_typical}), where $\rA$ is the discrete Fourier transform -- the proof of Theorem \ref{thm:unif_samp} will utilize a dual certificate from \cite{candes2011probabilistic}, which was constructed using the golfing scheme introduced by Gross in \cite{gross2011recovering,kueng2014ripless}; the proof of Theorem \ref{thm:min_sep_thm} will utilize a dual certificate constructed in \cite{tang2012compressive} through interpolation via a squared Fej\'{e}r kernel.
The bulk of our proofs will essentially show that the existence of these dual certificates is sufficient to guarantee stable gradient recovery.

\subsection{Stable gradient recovery implies stable signal recovery}
As explained, the crucial step of our proofs will be to show that for $z\in\bbC^N$,  and $\gamma>0$ such that $\nm{\rP_{\Omega}\rA z}_2\leq \gamma$
$$
\nm{z}_2 \leq C_N (\gamma + \nm{z}_{TV}),
$$
for some $C_N>0$ which may depend on $\Omega$ and $N$. We first show in Section \ref{sec:poincare} that a result of this form follows easily from the Poincar\'{e} inequality, and we will use the results of Section \ref{sec:poincare} to establish the proofs for Theorems \ref{thm:unif_samp}, \ref{cor:unif_samp_2D} and  \ref{thm:min_sep_thm} where $\Omega$ is chosen uniformly at random. In the case where $\Omega$ is constructed by a combination of uniform random sampling and variable density sampling, Section \ref{sec:stab_rip} will demonstrate that we can obtain a stronger result via  the restricted isometry property, and its main result will be used to establish the stronger statements of Theorems \ref{thm:near_optimal} and \ref{cor:unif_samp_2D_opt}.

\subsubsection{Stability via the Poincar\'{e} inequality}\label{sec:poincare}
In Theorems \ref{thm:unif_samp},  \ref{thm:min_sep_thm} and  \ref{cor:unif_samp_2D}, the constraint satisfied by the minimizer $\hat x$ to be analysed is of the form $\nm{\rP_\Omega \rA \hat x - y}_2\leq \sqrt{m}\cdot\delta$. Since it is always assumed that  $0\in\Omega$, we know that the error $z = \hat x - x$ necessarily satisfies 
$$\abs{(\rA z)_0} =\abs{  \sum_{j=1}^N z_j} \leq
\nm{\rP_\Omega \rA \hat x - y}_2 + \nm{\rP_\Omega \rA  x - y}_2\leq 2\delta\sqrt{m}.$$ 
We will utilize this fact and the Poincar\'{e} inequality in Lemma \ref{lem:stab_grad_stab_sig}, which shows that it suffices to only derive error bounds for the recovered gradient, $\nm{\rD z}_1$, when proving Theorems \ref{thm:unif_samp}, \ref{cor:unif_samp_2D} and \ref{thm:min_sep_thm}. We first state the discrete Poincar\'{e} inequality, which is a direct corollary of the classical Sobolev embedding inequality for functions of bounded variation \cite{ambrosio2000functions}.

\begin{lemma}[Poincar\'{e} inequality, see \cite{needell2013stable, ambrosio2000functions}]\label{lem:poincare}

\begin{itemize}
\item[(i)] Let $z\in\bbC^N$ be such that $\sum_{j=1}^N z_j = 0$. Then
$$
\nmu{  z}_2 \leq \sqrt{N}\cdot \nmu{z}_{TV}
$$
\item[(ii)] Let $z\in\bbC^{N\times N}$ be such that $\sum_{k=1}^N\sum_{j=1}^N z_{j, k} = 0$. Then
$$
\nmu{  z}_2 \leq \nmu{z}_{TV}.
$$

\end{itemize}

\end{lemma}

\begin{lemma}\label{lem:stab_grad_stab_sig} Let $N\in\bbN$. Let $m\in\bbN$ be such that $m\leq N$ and let $\delta>0$.
\begin{itemize}
\item[(i)]  Let $\rD$ be the finite differences operator on $\bbC^N$.
Let $z\in\bbC^N$ and suppose that $\absu{  \sum_{j=1}^N z_j} \lesssim \delta\sqrt{m}$. Then, $$
\frac{\nmu{z}_2}{\sqrt{N}} \lesssim \delta + \nmu{\rD z}_1.
$$
\item[(ii)]  Let $\rD$ be the finite differences operator on $\bbC^{N\times N}$. Let $z\in\bbC^{N\times N}$ and suppose that $\absu{  \sum_{j=1}^N z_j} \lesssim \delta\sqrt{m}$. Then, $$
\nmu{z}_2 
\lesssim \delta + \nmu{\rD z}_1.
$$
\end{itemize}

\end{lemma}
\prf{
We prove only (i), the proof of (ii) is identical except for the application of (ii) of the Poincar\'{e} inequality. 
First note that
$\tilde z = (\tilde{z}_j)_{j=1}^N$, where $\tilde{z}_j = z_j - \frac{1}{N}\sum_{j=1}^N z_j$ has mean zero. So, by the Poincar\'{e} inequality,
$$
\frac{1}{\sqrt{N}} \nmu{\tilde z}_2 \lesssim \nmu{\rD z}_1.
$$
Therefore,
$$
\frac{\nmu{z}_2}{\sqrt{N}} \lesssim \abs{\frac{1}{N}\sum_{j=1}^N z_j} + \nmu{\rD z}_1.
$$
Since $\abs{  \sum_{j=1}^N z_j} \leq 2\delta\sqrt{m}$, this implies that
$$
\frac{\nmu{z}_2}{\sqrt{N}} \lesssim \frac{\delta\sqrt{m}}{N} + \nmu{\rD z}_1
\leq \delta + \nmu{\rD z}_1.
$$

}

\subsubsection{Stability via the restricted isometry property}\label{sec:stab_rip}

To prove Theorem \ref{thm:near_optimal}, we will first show that since $\Omega$ includes $m$ samples drawn uniformly at random, we can obtain a stable and robust error bound on the recovered gradient. Then, we will show that stable gradient recovery implies stable signal recovery in a strong sense when $\Omega$ contains $m$ i.i.d. samples drawn in accordance with the probability distribution $p$ defined in Theorem \ref{thm:near_optimal}.
The main result of this section  will establish the latter, specifically, we will show that if we obtain a bound on $\nm{\hat x - x}_{TV}$ for a minimizer $\hat x$ of (\ref{eq:noise_tv}), then we have a bound on $\nm{\hat x - x}_2$.
\begin{proposition}\label{prop:gradtosignal_stability}
Let $N = 2^R$ for some $R\in\bbN$.
Let $\rA$ be the (non-unitary) discrete Fourier transform on $\bbC^N$ as defined in Theorem \ref{thm:original}. Suppose that $z$ is such that $\nm{\rP_{\Omega}\rA z}_2 \leq \sqrt{m}\gamma$, where $\Omega = \br{k_1,\ldots, k_m}$ is chosen i.i.d. such that  for each $j=1,\ldots, m$, 
$$
\bbP(k_j = n) = p(n)^{-1}, \quad p(n) = C \log(N)\max\br{1, \abs{n}}
$$
for some appropriate $C>0$ 
and $m\gtrsim 
s \cdot (\log(N) + \log(\epsilon^{-1}))$. Then with probability exceeding $1-\epsilon$, 
$$
\nm{z}_2 \lesssim \sqrt{\log(N)}\cdot \gamma +  \frac{ \nm{z}_{TV}\cdot\sqrt{N}\cdot \log^2(s)\log(N)(\log(s)+\log\log(N))}{s}.
$$
\end{proposition}
The proof of this result follows the proof of a synonymous two dimensional result from \cite{needell2013stable} (see also \cite{ward2013stable}); we will require the following two ingredients.
\begin{enumerate}
\item The discrete version of the fact that the decay in the Haar wavelet coefficients of a function of bounded variation can be controlled by its bounded variation norm.
\item Bounds concerning when a matrix constructed from a discrete Fourier system and a discrete Haar system satisfies a restricted isometry property (defined in Definition \ref{def:rip}).
\end{enumerate}  Note however that the arguments in \cite{needell2013stable} deal only with the recovery of two (or higher) dimensional  vectors and cannot be directly applied to deduce Proposition {\ref{prop:gradtosignal_stability}.

\subsubsection*{The decay of Haar coefficients and total variation}

In this section, we present some results which demonstrate how the Haar coefficients in a nonlinear approximation of a signal are controlled by its total variation norm. First, we begin with some definitions.
\begin{definition}\label{def:BV}
Let $n\in\bbN$. Given $f\in L^1([0,1)^n)$, the total variation of $f$ is defined to be
$$
\nm{f}_{V} = \sup \br{\int_{[0,1)^n} u(t) \phi(t) \mathrm{d}t : \, \phi \in C^1_c([0,1)^n), \nm{\phi}_{L^\infty([0,1)^n)} \leq 1}
$$
where $C^1_c([0,1)^2)$ denotes the space of continuously differentiable functions of compact support on $[0,1)^n$. The space of bounded variation functions on $[0,1)^n$ is defined to be
$$
BV([0,1)^n) = \br{f \in L^1([0,1)^n): \nm{f}_{V}<\infty}.
$$

\end{definition}

\begin{definition}[The Haar wavelet system]\label{def:Haar}
We  define the Haar transform for functions in $L^2[0,1)$ and the discrete Haar transform for vectors in $\bbC^N$.
Let $\Phi = \chi_{[0,1)}$ and let $\Psi = \chi_{[0,1/2)} - \chi_{[1/2,1)}$, where given an interval $I$, $\chi_I$ is the characteristic function on $I$.  Let
$$
\Psi_{j,k} = \sqrt{2^j}\Psi(2^j\cdot - k), \quad j\in\bbN, k=0,\ldots, 2^j-1.
$$
Then,
$$
\br{\Phi}\cup \br{\Psi_{j,k}: j\in\bbN\cup\br{0}, k=0,\ldots, 2^j-1}
$$
forms an orthonormal basis for $L^2[0,1)$. Order the functions in order of increasing dilation factor such that
\spls{\label{wav_ord}
\br{\varphi_j}_{j\in\bbN} = \br{ \Phi, \Psi, \Psi_{1,0}, \Psi_{1,1}, \ldots, \Psi_{j,0},\Psi_{j,1}\ldots, \Psi_{j,2^j-1},\Psi_{j+1,0},\ldots}
}
The Haar transform on $L^2[0,1)$ is defined by 
$$
\cW:L^2[0,1) \to \ell^2(\bbN), \quad f\mapsto (\ip{f}{\varphi_j})_{j\in\bbN}.
$$
For $N=2^J$ for some $J\in\bbN$, define
$$
\cT:L^2[0,1) \to \bbC^N, \qquad f \mapsto N^{1/2} \left(\int_{(k-1)/N}^{k/N} f(t) \mathrm{d}t\right)_{k=1}^N.
$$
Let ${ H_j} = \cT \varphi_j$. Then $\br{ H_j}_{j=1}^N$ forms an orthonormal basis for the vector space $\bbC^N$. 
The discrete Haar transform on $\bbC^N$ is defined by 
\be{\label{eq:DWT}
 \rW:\bbC^N \to \bbC^N, \quad z \mapsto (\ip{z}{{ H_j}})_{j=1}^N.
 }
\end{definition}

\begin{lemma}\cite[equation (9.53)]{mallat2008wavelet}\label{lem:haar_decay}
Let $f\in BV[0,1)$ and suppose that $\int_0^1 f(t)\mathrm{d}t = 0$. Let $c_j$ be the $j^{th}$ largest entry in magnitude of $\cW f$, the Haar coefficients of $f$. Then, there exists some constant $C$, independent of $f$ such that
$$
\abs{c_j} \leq \frac{C\cdot \nm{f}_{V}}{j^{3/2}}
$$
\end{lemma}

\begin{lemma}\label{lem:discrete_haar_decay}
Let $z\in\bbC^N$ be mean zero (i.e. $\sum_{j=1}^N z_j = 0$).
Let $h = \rW z$ be the discrete Haar coefficients of $z$ and let $\pi :\br{1,\ldots, N}\to \br{1,\ldots, N}$ be a permutation such that
$\abs{h_{\pi(j)}} \geq \abs{h_{\pi(j+1)}}$ for $j=1,\ldots, N-1$. Then, for some constant $C$ independent of $z$ and $h$,
$$
\abs{h_{\pi(j)}}\leq \frac{C\cdot \sqrt{N}\cdot \nm{z}_{TV}}{j^{3/2}}.
$$
\end{lemma}
\prf{
Define $f\in BV[0,1)$ by
$$
f(t) = \sqrt{N}\cdot z_j, \quad \frac{j-1}{N}\leq t<\frac{j}{N}, \quad j=1,\ldots, N.
$$
Then, $h = W z = (\cW f)_{j=1}^N$. Note also that $(\cW f)_k = 0$ for all $k>N$. Also, $\nm{f}_{V} \leq \sqrt{N}\nm{z}_{TV}$ and $\int_0^1 f(t)\mathrm{d}t =0$. Therefore, we can apply Lemma \ref{lem:haar_decay} to obtain
$$
\abs{h_{\pi(j)}} \leq \abs{(\cW f)_{\pi(j)}} \leq \frac{C \cdot\nm{f}_{V}}{j^{3/2}} = \frac{C \cdot\sqrt{N}\cdot \nm{z}_{TV}}{j^{3/2}}.
$$
}

\subsubsection*{The restricted isometry property}

\begin{definition}\label{def:rip}
Let $\rU\in\bbC^{m\times N}$. Let $s\leq N$ and let $\delta \in (0,1)$, $\rU$ is said to satisfy the restricted isometry property (RIP) of order $s$ and level $\delta$, if
$$
(1-\delta) \nm{z}^2_2 \leq \nm{\rU z}_2^2 \leq (1+\delta)\nm{z}_2^2
$$
for all $s$-sparse vectors $z\in\bbC^N$.
\end{definition}

\defn{\label{def:bos}
Let $T\in\bbC^N$ be endowed with a probability measure $\nu$, then the set of functions $\br{\psi_j:T\to \bbC, j=1,\ldots, N}$ is said to be a bounded orthonormal system  with respect to $\nu$ and bound $K$ if
$\int_T \psi_j(x) \psi_k(x) \mathrm{d}\nu(x) = \delta_{j,k}$ where $\delta_{j,k}$ is the Knonecker delta, and $max_{j=1}^N \nm{\psi_j}_\infty \leq K$. A random sample of an orthonormal system is the vector 
$(\psi_1(x),\ldots, \psi_N(x))$ where $x$ is a random variable drawn in accordance with $\nu$.
}

\begin{theorem}\cite[Theorem 12.32]{foucart2013mathematical}\label{thm:BOS}
Let $\rA\in\bbC^{m\times N}$ be a matrix whose rows are independent random samples of an orthonormal system with bound $K$. For $\epsilon, \delta\in (0,1)$, if
$$
m \geq C \cdot K^2 \delta^{-2} s \max\br{
\log^2(s)\log(K^2\delta^{-2}s\log(N))\log(N), \, \log(\epsilon^{-1})}
$$
for some constant $C$, 
then with
 probability exceeding $1-\epsilon$, $m^{-1/2}A$ satisfies the RIP of order $s$ at level $\delta$.
\end{theorem}

\begin{lemma}\label{lem:RIP}

Let $N= 2^J$ for some $J\in\bbN$ and let $\rU = \frac{1}{\sqrt{m}} \rX \rP_\Omega \rV \rW^*$ where $\rV$ is the unitary discrete Fourier transform on $\bbC^N$, and $\rX = \mathrm{diag}\left((p(k))_{k=-N/2+1}^{N/2}\right)$ is such that for $k=-N/2+1,\ldots,N/2$,
$$p(k) = C \sqrt{\log(N)}\max\br{1, \sqrt{\abs{k}}},$$  $C$ is an appropriate constant such that $p^{-2}$ is a probability measure  on $\br{-N/2+1,\ldots,N/2}$. 
 Then, with probability exceeding $1-\epsilon$, $\rU$ satisfies the RIP of order $s$ and level $\delta$ provided that  $\Omega$ consists of $m$ i.i.d. indices chosen in accordance with the probability measure $p^{-2}$ and
$$
m \geq C \cdot \delta^{-2} s \max\br{
\log^2(s)\log(\delta^{-2}s\log^2(N))\log^2(N), \, \log(\epsilon^{-1})}
$$

\end{lemma}
\prf{
This result was essentially derived in \cite{ward2013stable}, and the proof is virtually identical to the proof of  Theorem 4 in Section VI of \cite{ward2013stable}. However,  this result was not explicitly stated in \cite{ward2013stable} as their results did not consider the case of one-dimensional total variation. Furthermore, we will use Theorem \ref{thm:BOS}, which is a slightly more general version of the result used in \cite{ward2013stable} and this will allow us to state the probability that the RIP is satisfied in terms of $\epsilon$.
we will outline only the key arguments about why this result is true. 
\begin{itemize}
\item First, suppose that we are given two orthonormal bases $\br{\varphi_j}_{j=1}^N$ and $\br{\psi_j}_{j=1}^N$ and suppose that $\max_{k=1}^N \abs{\ip{\varphi_k}{\psi_j}} \leq \kappa_j$. Define a probability measure on $\br{1,\ldots, N}$ by $\nu(j) = \kappa_j^2/\nm{\kappa}^2_2$.
\item Let $d_j = \nm{\kappa}_2/\kappa_j$, let $\Psi\in\bbC^{N\times N}$ be the matrix whose $j^{th}$ row is the vector $\psi_j$ and let $\Phi\in\bbC^{N\times N}$ be the matrix whose $j^{th}$ row is the vector $\varphi_j$. Then,  by letting $\eta_l (d_j \ip{\varphi_l}{\psi_j})_{j=1}^N$ for each $l=1,\ldots, N$, 
one can apply Definition \ref{def:bos} to verify that $\br{\eta_l}_{l=1}^N$ is a bounded orthonormal system with respect to $\nu$ and bound $\nm{\kappa}_2$.
\item Then, by Theorem \ref{thm:BOS}, with probability at least $1-\epsilon$, the matrix $\rP_\Omega \Psi\Phi^*$, where $\Omega$ consists of $m$ i.i.d. indices from $\br{1,\ldots, N}$ chosen in accordance with $\nu$, satisfies the RIP of order $s$ and level $\delta$ provided that $$
m \geq C \cdot \nm{\kappa}_2^2 \delta^{-2} s \max\br{
\log^2(s)\log(\nm{\kappa}_2^2\delta^{-2}s\log(N))\log(N), \, \log(\epsilon^{-1})}.
$$
So, to prove this result, we need only derive bounds on $\kappa_j$ and $\nm{\kappa}_2$ for the case of the Fourier and Haar bases, where $\varphi_j = H_j$, with $H_j$ as defined in Definition \ref{def:Haar}, and $\psi_k = N^{-1/2}( e^{2\pi i k j/N})_{j=1}^{N}$, and we will index the system of $\Psi$ with $k=-N/2,\ldots, N/2$.
 By \cite[Section VI, Corollary 2]{ward2013stable}, $\abs{\ip{\psi_k}{\varphi_j}} \leq 3\sqrt{2\pi}/\sqrt{\abs{k}}$ for $k\neq 0$ and one can check that $\abs{\ip{\psi_k}{\varphi_j}} =1$ for $k=0$. Thus, in the Fourier with Haar case, we can let $\kappa_j = \min\br{ 3\sqrt{2\pi}/\sqrt{\abs{j}}, \, 1}$, and 
 $$
 \nm{\kappa}_2^2 = 1+2\sum_{j=1}^{N/2}\frac{18\pi }{j}\leq 36\pi +1+\int_{1}^{N/2} \frac{18\pi}{x}\mathrm{d}x\leq 36(1+\pi\log(N)).
 $$

\end{itemize}

}

\subsubsection*{Proof of Proposition \ref{prop:gradtosignal_stability}}
We will first show how a combination of the RIP and the relationship between the decay of Haar wavelet coefficients and the total variation of a signal can lead to stronger control (in comparison with the standard Poincar\'{e} result) of $\nm{\cdot}_2$ by $\nm{\cdot}_{TV}$, then the proof of Proposition \ref{prop:gradtosignal_stability} will be completed  by applying  specific bounds from Lemma \ref{lem:RIP}.
\begin{lemma}\label{lem:strong_sobolev}
Let $N=2^J$ for some $J\in\bbN$ and let $\rU = \rV \rW^*$ where $\rW$ is the discrete Haar transform on $\bbC^N$ defined in (\ref{eq:DWT}) and $\rV\in\bbC^{m\times N}$. Let $z\in\bbC^N$. Suppose that $\rU$ satisfies the RIP of order $2r+1$ and level $\delta$, and suppose that $\nm{\rV z}_2 \leq \epsilon$. Then,
$$
\nm{z}_2 \lesssim \frac{\epsilon}{1-\delta} +   \frac{1+\delta}{1-\delta}\cdot \frac{ \nm{z}_{TV}\cdot\sqrt{N}}{r}.
$$
\end{lemma}
\begin{proof}
Let $c  = \rW z$. Let $c^L, c^H\in\bbC^N$ be such that $c^L = \rP_{\br{1}}c$ and $c^H = \rP_{\br{2,\ldots, N}} c$. Note that $c^L$ has only one non-zero entry, which is the scaling coefficient $\ip{z}{H_0} = N^{-1/2}\sum_j z_j$. Decompose $z$ as $z = z^L + z^H$, where $z^L$ is the vector whose entries are all of the constant value $N^{-1} \sum_{j} z_j$. Then $z^H$ is mean-zero, $\rW z^H=c^H$ and $\rW z^L=c^L$. Note also that $\nm{z^H}_{TV} = \nm{z}_{TV}$. We will apply Lemma \ref{lem:discrete_haar_decay} to $z^H$.

Let $\pi :\br{1,\ldots, N}\to \br{1,\ldots, N}$ be a permutation such that
$\absu{c^H_{\pi(j)}} \geq \absu{c^H_{\pi(j+1)}}$ for $j=1,\ldots, N-1$.
Let $\Delta_0$ index the largest $r$ entries of $c^H$ in magnitude, let $\Delta_1$ index the next largest $r$ entries of $c^H$ in magnitude, and so on.
Then, by applying Lemma \ref{lem:discrete_haar_decay},
\spl{\label{eq:cone1}
\nm{\rP_{\Delta_0}^\perp c^H}_1 = \sum_{j=r+1}^N \abs{c^H_{\pi(j)}}
\leq C \nm{z}_{TV} \sqrt{N} \sum_{j=r+1}^N \frac{1}{j^{3/2}}
\leq  \frac{2 C \nm{z}_{TV}\sqrt{N}}{\sqrt{r}},
}
where $C$ is the constant from Lemma \ref{lem:discrete_haar_decay}. Similarly, we can also obtain an $\ell^2$ bound,
\spl{\label{eq:cone2}
\nm{\rP_{\Delta_0}^\perp c^H}_2 = \sqrt{\sum_{j=r+1}^{N}\abs{c_{\pi(j)}}^2}
\leq  C \nm{z}_{TV} \sqrt{N \sum_{j=r+1}^N \frac{1}{j^3}} = \frac{C \nm{z}_{TV} \sqrt{N} }{\sqrt{2}\cdot r}.
}
Recalling the assumption that $\epsilon \geq \nm{\rV z}_2$, we have that
\spls{
\epsilon \geq \nm{\rV z}_2 &= \nm{\rV \rW^* \rW z}_2 
= \nm{\rU (c^L + \rP_{\Delta_0}c^H + \rP_{\Delta_0}^\perp c^H)}_2\\
&\geq \nm{\rU(c^L + \rP_{\Delta_0}c^H + \rP_{\Delta_1}c^H)}_2 -
\sum_{j\geq 2} \nm{\rU \rP_{\Delta_j}c^H}_2\\
&\geq (1-\delta)\nm{c^L + \rP_{\Delta_0}c^H + \rP_{\Delta_1}c^H}_2 -
(1+\delta)\sum_{j\geq 2} \nm{ \rP_{\Delta_j}c^H}_2\\
&\geq (1-\delta)\nm{c^L + \rP_{\Delta_0}c^H}_2 -
\frac{(1+\delta)}{\sqrt{r}}\sum_{j\geq 1} \nm{ \rP_{\Delta_j}c^H}_1\\
&= (1-\delta)\nm{c^L + \rP_{\Delta_0}c^H}_2 -
\frac{(1+\delta)}{\sqrt{r}}\nm{ \rP_{\Delta_0}^\perp c^H}_1.
}
Rearranging and combining with (\ref{eq:cone1}) yields
\eas{
\nm{c^L + \rP_{\Delta_0}c^H}_2 \leq \frac{\epsilon}{(1-\delta)} +\frac{(1+\delta)\cdot \nm{ \rP_{\Delta_0}^\perp c^H}_1}{(1-\delta)\cdot \sqrt{r}}
& \lesssim \frac{\epsilon}{(1-\delta)} +\frac{(1+\delta) \nm{z}_{TV}\sqrt{N}}{(1-\delta)\cdot r} .
}
Finally, combining with (\ref{eq:cone2}) yields
\eas{
\nm{z}_2 = \nm{c^L + \rP_{\Delta_0}c^H}_2 + \nm{\rP_{\Delta_0}^\perp c^H}_2
\lesssim \frac{\epsilon}{(1-\delta)} +\frac{(1+\delta) \nm{z}_{TV}\sqrt{N}}{(1-\delta)\cdot r} .
}
\end{proof}

\begin{proof}[Proof of Proposition \ref{prop:gradtosignal_stability}]
 Let $\tilde \rA = N^{-1/2} \rA$ and note that it is unitary. We may assume that 
 $$
 \frac{s}{\log^2(s)\log(N)\log(s\log(N))}\geq 1
 $$
 since the result of Proposition \ref{prop:gradtosignal_stability} follows from Lemma \ref{lem:stab_grad_stab_sig} otherwise. 
Let
  $\rX := \mathrm{diag}\left((p(k))_{k=-N/2+1}^{N/2}\right)$ such that for $k=-N/2+1,\ldots,N/2$,
$$p(k) = C \sqrt{\log(N)}\max\br{1, \sqrt{\abs{k}}},$$ where $C$ is an appropriate constant such that $p^{-2}$ is a probability measure  on $\br{-N/2+1,\ldots,N/2}$.  Then, recall from Lemma \ref{lem:RIP} that by our choice of $m$, $\rU = \frac{1}{\sqrt{m}} \rX \rP_\Omega \tilde \rA \rW^*$  satisfies the RIP of  level $1/2$ and order 
$$
r \leq \frac{s}{\log^2(s)\log(N)\log(s\log(N))}.
$$ Also, since $\nm{p}_\infty \lesssim \sqrt{N\log(N)}$, it follows that
$$
\nm{\frac{1}{\sqrt{m}} \rX \rP_\Omega \tilde \rA z}_2 \lesssim \sqrt{\log(N)}\cdot 
\nm{\frac{1}{\sqrt{m}}\rP_\Omega  \rA z}_2 \leq \sqrt{\log(N)}\cdot \gamma.
$$
So, by Lemma \ref{lem:strong_sobolev},
$$
\nm{z}_2 \lesssim \sqrt{\log(N)}\cdot \gamma +  \frac{ \nm{z}_{TV}\cdot\sqrt{N}\cdot \log^2(s)\log(N)\log(s\log(N))}{s}.
$$
\end{proof}

\subsubsection*{The two dimensional case}
We have the following strong Sobolev inequality in the case of two dimensional vectors. The proof of this result is similar to the one dimensional case and directly applies the results of \cite{needell2013stable}. 
\begin{proposition}\label{prop:gradtosignal_stability_2d}
Let $N = 2^R$ for some $R\in\bbN$.
Let $\rA$ be the (non-unitary) discrete Fourier transform on $\bbC^N$ as defined in (\ref{def:2d_DFT}). Suppose that $z$ is such that $\nm{\rP_\Omega \rA z}_2 \leq \sqrt{m}\gamma$, where $\Omega = \br{k_1,\ldots, k_m}\subset \br{-N/2+1,\ldots, N/2}^2$ is chosen i.i.d. such that  for each $j=1,\ldots, m$, 
$$
\bbP(k_j = (n, m)) = p(n, m)^{-1}, \quad p(n,m) = C \log(N)\max\br{1,\abs{n}^2 + \abs{m}^2}
$$
for some appropriate $C>0$ 
and $m\gtrsim s \cdot (\log(N) + \log(\epsilon^{-1}))$. Then with probability exceeding $1-\epsilon$, 
$$
\nm{z}_2 \lesssim \sqrt{\log(N)}\cdot \gamma +  \frac{ \nm{z}_{TV}\cdot\log(N^2/s)\cdot \log^2(s)\log(N)(\log(s)+\log\log(N))}{\sqrt{s}}.
$$
\end{proposition}

The key difference between the proof of this proposition and its one dimensional counterpart in Proposition \ref{prop:gradtosignal_stability} is the relationship between Haar coefficients and bounded variation norms for bivariate functions. We will first define the bivariate Haar system, then present the results from \cite{needell2013stable, ward2013stable} which will allow us to deduce Proposition \ref{prop:gradtosignal_stability_2d}.

\begin{definition}[The bivariate Haar system]
Let $$\tilde \Phi(x,y) =\Phi(x) \Phi(y), \quad \tilde\Psi^v(x,y) = \Phi(x) \Psi(y), \quad \tilde \Psi^h(x,y) = \Psi(x)\Phi(y), \quad \tilde\Psi^d(x,y) = \Psi(x)\Psi(y).$$ Also, for $j\in\bbN$ and $k\in\br{0,\ldots, 2^j-1}^2$, let $\tilde{\Psi}_{j,k}^e = 2^j \tilde \Psi^e(2^j \cdot -k)$ for each $e\in\br{v,d,h}$. Then,
$$
\br{\tilde \Phi}\cup \br{\tilde \Psi^e_{j,k}:  e\in\br{v,d,h}, \, j\in\bbN, \, k\in\br{0,\ldots, 2^j-1}^2}
$$
forms an orthonormal basis for $L^2\left([0,1)^2\right)$. Let $\br{\tilde \varphi_j}_{j\in\bbN^2}$ denote these basis elements after ordering in (any) increasing order of dilation factor, such that
$$
\br{\tilde \varphi_{j_1,j_2} : j_1, j_2=1, \ldots, 2^{J}}= \br{\tilde \Phi}\cup \br{\tilde \Psi^e_{j,k}:  e\in\br{v,d,h}, \, j=0,\ldots, J-1, \, k\in\br{0,\ldots, 2^j-1}^2}
$$
Given $N=2^J$ for some $J\in\bbN$, define
$$
\tilde \cT:L^2\left([0,1)^2\right) \to \bbC^{N\times N}, \qquad f\mapsto
N \left(\int_{(k_1-1)/N}^{k_1/N}\int_{(k_2-1)/N}^{k_2/N} f(t_1,t_2) \mathrm{d}t_1\mathrm{d}t_2\right)_{k_1,k_2 = 1}^N.
$$
Then, let $\tilde H_{j_1,j_2} := \tilde \cT \varphi_{j_1,j_2}$, and $\br{\tilde H_{j_1,j_2}: j_1,j_2=1, \ldots, N}$ forms an orthonormal basis for the vector space $\bbC^{N\times N}$. Let the two dimensional discrete Haar transform be denoted by
$$
\tilde \rW: \bbC^{N\times N} \to \bbC^{N\times N} ,\qquad z \mapsto \left(\ip{z}{\tilde H_{j_1,j_2}}\right)_{j_1,j_2=1}^{N}.
$$
\end{definition}

All the lemmas used to deduce Proposition \ref{prop:gradtosignal_stability} can be directly extended to two dimensions, except for Lemmas \ref{lem:haar_decay} and \ref{lem:discrete_haar_decay}. The following lemma reveals the relationship between bivariate Haar coefficients and the total variation norm which will be used instead of the latter lemma. Lemma \ref{lem:discrete_haar_decay_2d} is the finite dimensional version of a result on bounded variation functions from \cite{petrushev1999nonlinear}. Although this paper deals only with the one and two dimensional case, analogous statements to the following lemma also exist for higher dimensions (see \cite{needell2013near}), so one can extend the results of this section to multidimensional cases.

\begin{lemma}\cite{needell2013stable}\label{lem:discrete_haar_decay_2d}
Let $z\in\bbC^{N\times N}$ be mean zero (i.e. $\sum_{k,j=1}^N z_{k,j} = 0$).
Let $c_j$ be the $j^{th}$ largest entry in magnitude of $\tilde \rW z$,  the discrete Haar coefficients of $z$. Then, for some constant $C$ independent of $z$ and $h$,
$$
\abs{c_j}\leq \frac{C\cdot  \nm{z}_{TV}}{j}.
$$
\end{lemma}
\begin{remark}
Note that \cite{needell2013stable} actually proved that
$
\abs{c_j}\leq C\cdot  \nm{z}_{TV}'/j,$
where $\nm{\cdot }_{TV}'$ is the isotropic total variation norm with Neumann boundary conditions. However, Lemma \ref{lem:discrete_haar_decay_2d} holds since $\nm{z}'_{TV} \leq \nm{z}_{TV}$ for all $z\in\bbC^{N\times N}$.
\end{remark}

Now, applying this bound on the bivariate Haar coefficients (just as in the proof of Lemma \ref{lem:strong_sobolev}) yields the following result.
\begin{lemma}\cite{needell2013stable}\label{lem:strong_sobolev_2d}
Let $N=2^J$ for some $J\in\bbN$ and let $\rU = \rV \rW^*$ where $\rW$ is the discrete Haar transform on $\bbC^{N\times N}$ and $\rV: \bbC^{N\times N}\to \bbC^m$. Let $z\in\bbC^{N\times N}$. Suppose that $\rU$ satisfies the RIP of order $2r+1$ and level $\delta$, and suppose that $\nm{\rV z}_2 \leq \epsilon$. Then,
$$
\nm{z}_2 \lesssim \frac{\epsilon}{1-\delta} +   \frac{1}{1-\delta}\cdot \frac{ \nm{z}_{TV}\cdot\log(N^2/s)}{\sqrt{r}}.
$$
\end{lemma}

The final lemma which we require to deduce Proposition \ref{prop:gradtosignal_stability_2d} is the following, which considers how one can combine the discrete Fourier transform and the discrete Haar transform such that the RIP is satisfied.

\begin{lemma}\cite{ward2013stable} \label{lem:RIP_2d}
Let $N= 2^J$ for some $J\in\bbN$ and let $\rU = \frac{1}{\sqrt{m}} \rX \rP_\Omega \rV \tilde \rW^*$ where $\rV$ is the unitary discrete Fourier transform on $\bbC^{N\times N}$, $\tilde \rW$ is the discrete bivariate Haar transform and $\rX:\bbC^{N\times N} \to \bbC^{N\times N}$ is such that $\rX z =  \tilde p \circ z$, where $\circ$ denotes pointwise multiplication and $\tilde p \in\bbC^{N\times N}$ is defined by 
$$\tilde p = (p(k_1,k_2))_{k_1,k_2=-N/2+1}^{N/2}, \qquad p(k_1,k_2) = C \sqrt{\log(N)}\max\br{1, \sqrt{\abs{k_1}^2+\abs{k_2}^2}},$$ where $C$ is an appropriate constant such that $p^{-2}$ is a probability measure  on $\br{-N/2+1,\ldots,N/2}^2$.
 Then, with probability exceeding $1-\epsilon$, $\rU$ satisfies the RIP of order $s$ and level $\delta$ provided that  $\Omega$ consists of $m$ i.i.d. indices chosen in accordance with the probability measure $p^{-2}$ and
$$
m \geq C \cdot \delta^{-2} s \max\br{
\log^2(s)\log(\delta^{-2}s\log^2(N))\log^2(N), \, \log(\epsilon^{-1})}
$$

\end{lemma}

\begin{proof}[Proof of Proposition \ref{prop:gradtosignal_stability_2d}]
 Let $\rV = N^{-1} \rA$ and note that it is unitary. Recall the definition of $\tilde p$ and $\rX$ from Lemma \ref{lem:RIP_2d}. Since  $m=\ord{s \cdot \log (N)\cdot (1+\log(\epsilon^{-1}))}$, it follows from Lemma \ref{lem:RIP_2d} that with probability at least $1-\epsilon$, $\rU = \frac{1}{\sqrt{m}} \rX \rP_\Omega \rV \rW^*$  satisfies the RIP of order 
$$
r= \frac{s}{\log^2(s)\log(N)(\log(s)+\log\log(N))}
$$and level $1/2$. Also, since $\nm{\tilde p}_\infty \lesssim N \sqrt{\log(N)}$
$$
\nm{\frac{1}{\sqrt{m}} \rX \rP_\Omega \rV z}_2 \lesssim \sqrt{\log(N)}\cdot 
\nm{\frac{1}{\sqrt{m}}\rP_\Omega  \rA z}_2 \leq \sqrt{\log(N)}\cdot \gamma.
$$
So, by Lemma \ref{lem:strong_sobolev_2d},
$$
\nm{z}_2 \lesssim \sqrt{\log(N)}\cdot \gamma +  \frac{ \nm{z}_{TV}\cdot\log(N^2/s)\cdot \sqrt{\log^2(s)\log(N)(\log(s)+\log\log(N))}}{\sqrt{s}}.
$$
\end{proof}

\subsection{Proofs of Theorem \ref{thm:near_optimal}, Theorem \ref{cor:unif_samp_2D_opt}, Theorem \ref{thm:unif_samp} and Theorem \ref{cor:unif_samp_2D}}

We first prove Theorem \ref{thm:unif_samp}, then mention the modifications required to prove Theorem \ref{thm:near_optimal}, Theorem \ref{cor:unif_samp_2D_opt} and Theorem \ref{cor:unif_samp_2D}.
To begin, we recall some definitions from \cite{candes2011probabilistic}.
\begin{definition}
Let  $\rU\in\bbC^{N\times N}$. 
\begin{enumerate}
\item For $\rU$ such that $N^{-1/2} \rU$ is an isometry, the coherence of $\rU$ is 
 $\mu(\rU) := \max_{i,j} \abs{\rU_{i,j}}^2$.
 \item Given $\Delta\subset \br{1,\ldots, N}$, $\delta\in(0,1)$ and $r\in\bbN$, $\rU$ is said to satisfy the weak restricted isometry property (RIP) if  for all $v$ supported on $\Delta \cup \Gamma$ with $\abs{\Gamma} \leq r$,
$$
(1-\delta) \nmu{v}_2^2 \leq \nmu{\rU v}_2^2 \leq (1+\delta) \nmu{v}^2_2.
$$
\end{enumerate}

\end{definition}
Note that given two orthonormal bases $\Phi = \br{\phi_j}_{j=1}^N$ and $\Psi = \br{\psi_j}_{j=1}^N$ for $\bbC^N$, the columns of the matrix $\rU = \sqrt{N}\left(\ip{\psi_j}{\phi_k}\right)_{k,j=1}^N$ form a bounded orthonormal system with respect to the uniform measure $\nu$ on $[N]$ with $\nu(S) = \abs{S}/N$ for $S\subset [N]$. The coherence of this matrix $\rU$ is such that $\mu(\rU) \in [1, N]$ and can be understood as a measure for the correlation between the two bases $\Phi$ and $\Psi$. To prove Theorem \ref{thm:unif_samp}, we will make use of the fact that the  discrete Fourier transform matrix satisfies $\mu(\rU) = 1$ and the following result, which was essentially derived in \cite{candes2011probabilistic} (see assumptions (i)-(iii) at the start of the proof of Theorem 1.2 in \cite{candes2011probabilistic}). This result shows the existence of a dual certificate, which we will show is a sufficient dual vector to guarantee the recovery bounds in Theorem \ref{thm:unif_samp}.

The actual proof of Theorem \ref{thm:unif_samp} is similar to the proof of Theorem 1.2 in \cite{candes2011probabilistic}, however, as the minimization problem (\ref{eq:noise_tv}) is slightly different from the setup in \cite{candes2011probabilistic}, for completeness, we will repeat many steps of the argument from \cite{candes2011probabilistic}. 
\begin{proposition} \cite{candes2011probabilistic} \label{prop:dual_existence_inc}
Let $\epsilon \in (0,1)$.
Let $\rU\in\bbC^{N\times N}$ be such that $N^{-1/2}\rU$ is unitary. Let $m\in\bbN$ with $0<m\leq N$. Let $\Lambda \subset \br{1,\ldots, N}$ and let $x_0 \in\bbC^N$.  Let $s = \abs{\Lambda}$. Let $\Gamma \subset \br{1,\ldots, N}$ be $m$ indices drawn uniformly at random.   Let $\rU_{\Gamma,\Lambda} =  m^{-1/2} \rP_\Gamma \rU \rP_\Lambda$. If 
$$
m \geq C \cdot \left(1+\log(\epsilon^{-1})\right) \cdot \mu(U) \cdot  \log(N) \cdot  s,
$$ for some appropriate numerical constant $C$, then the following hold with probability exceeding $1-\epsilon$.
\begin{itemize}
\item[(i)] $\nmu{(\rU_{\Gamma,\Lambda}^* \rU_{\Gamma,\Lambda})^{-1}}_{2\to 2} \leq 2$
\item[(ii)] $m^{-1/2}\max_{i\in\Lambda^c} \nmu{\rU_{\Gamma,\Lambda}^* \rP_\Gamma \rU e_i}_2 \leq 1$
\end{itemize}
and there exists $\rho = \rU^* \rP_\Gamma w$ such that 
\begin{itemize}
\item[(iii)] $\nmu{\rP_\Lambda \rho - \sgn(\rP_\Lambda x_0)}_2 \leq 1/4$
\item[(iv)] $\nmu{\rP_\Lambda^\perp \rho}_\infty \leq 1/4$
\item[(v)] there exists some numerical constant $C_0$ such that $\nmu{w}_2 \leq  C_0 \cdot \abs{\Lambda}^{1/2}\cdot m^{1/2}$.
\item[(vi)] $m^{-1/2} \rP_\Gamma \rU$ satisfies the  weak RIP with respect to $\Lambda$, $\delta = 1/4$ and 
$$
r= \left\lfloor \frac{m}{C\left(1+\log(\epsilon^{-1})\right) \mu(\rU) \log(N) \log(m) \log^2(s)} \right\rfloor.
$$

\end{itemize}

\end{proposition}
\begin{remark}
Instead of the sampling without replacement model stated in Proposition \ref{prop:dual_existence_inc}, the version proved in \cite{candes2011probabilistic} actually considered a slightly different probability model, where $\Gamma$ is drawn independently, so that its elements are not necessarily unique. However, their proofs rely on arguments from \cite{gross2011recovering} and \cite{rudelson2008sparse} which allow for the same statements to be made in the case of sampling without replacement (see also \cite{gross2010note}).
\end{remark}

\begin{proof}[Proof of Theorem \ref{thm:unif_samp}]
First, the discrete Fourier transform $\rA$ is such that $N^{-1/2}\rA$ is unitary and $\mu(\rA)=1$. So, 
by letting
$$
\rU := \rA, \quad \Gamma := \Omega', \quad \Lambda := \Delta, \quad x_0 :=x
$$
in Proposition \ref{prop:dual_existence_inc},
conditions (i)-(vi) of Proposition \ref{prop:dual_existence_inc} are true with probability exceeding $1-\epsilon$ provided that the number of samples $m$ is chosen in accordance with (\ref{eq:num_samples_unif}). 
The rest of this proof will show that these conditions imply stable gradient recovery.
Let $z = \hat x - x$.

\textbf{Step I:}  We will show that $$
\nmu{\rP_\Delta^\perp \rD z}_1 \lesssim  \nmu{\rP_\Delta^\perp \rD x}_1
+  \delta\cdot\left(1+  \sqrt{s}\right).
$$

To do this, we first demonstrate that $\nmu{\rP_\Delta \rD z}_2$ can be controlled by $\delta$ and $\nmu{\rP_\Delta^\perp \rD z}_1$.  
Since $m^{-1/2}\rP_\Omega \rA$ satisfies the weak RIP with respect to $\delta = 1/4$ and $\Delta$, we have that $\nm{m^{-1/2}\rP_\Omega \rA\rP_\Delta}_2 \leq \sqrt{5}/2$. So, combining this with properties (i) and (ii)  of Proposition \ref{prop:dual_existence_inc}, repeated application of H\"{o}lder's inequality yields the following.
\spl{\label{eq:step1_a}
\nmu{\rP_\Delta \rD z}_2 &= \nm{(\rA_{\Omega', \Delta}^* \rA_{\Omega', \Delta})^{-1} \rA_{\Omega', \Delta}^* \rA_{\Omega', \Delta} \rP_\Delta \rD z}_2\\
&\leq  2\left(\nm{\frac{1}{m} \rP_\Delta \rA^* \rP_{\Omega'} \rA \rD z}_2 + \nm{\frac{1}{m}  \rA_{\Omega', \Delta}^* \rP_{\Omega'} \rA \rP_\Delta^\perp \rD z}_2\right)\\
&\leq \sqrt{5} \cdot \frac{1}{\sqrt{m}} \nm{\rP_{\Omega'} \rA \rD z}_2 + 2
\max_{j\in\Delta^c} \frac{1}{\sqrt{m}}\nm{ \rA_{\Omega', \Delta}^* \rP_{\Omega'} \rA  e_j}_2 \nmu{\rP_\Delta^\perp \rD z}_1 \\
&=\sqrt{5}\cdot \frac{1}{\sqrt{m}} \nmu{\rP_{\Omega'} \rA \rD z}_2 +
2 \nmu{\rP_\Delta^\perp \rD z}_1 \\
&\leq 4\sqrt{5}\cdot \delta + 2\nmu{\rP_\Delta^\perp \rD z}_1.
}
In the last line of the above computation, note that  for all $k\neq 0$, $v_k \cdot(\rA z)_k =  (\rA \rD z)_k$ where $v_k = 1-e^{2\pi i k/N}$. Also, $(\rA \rD z)_0 = 0$. Thus, since $\nmu{\rP_\Omega \rA z}_2 \leq \nm{\rP_\Omega \rA x - y}_2 +\nm{\rP_\Omega \rA \hat x - y}_2 \leq 2\delta\sqrt{m}$ by the enforced constraint in the minimization problem, we have that
 $$\nmu{\rP_{\Omega'} \rA \rD z}_2 \leq \nmu{\rP_\Omega \rA \rD z}_2 = 
 \nmu{\mathrm{diag}\left((v_k)_{k\in\Omega}\right)\cdot \rP_\Omega \rA  z}_2
 \leq 2 \nmu{\rP_\Omega \rA z}_2 \leq 4\delta\sqrt{m},$$ since $\abs{v_k} \leq 2$.

To bound $\nmu{\rP_\Delta^\perp \rD z}_1$, first observe that by algebraic manipulation,
$$
\nmu{\rD \hat x}_1\geq \nmu{\rP_\Delta^\perp \rD z}_1 - 2\nmu{\rP_{\Delta}^\perp \rD x}_1 + \nmu{\rD x}_1 + \Re\ip{\rP_\Delta \rD z}{\sgn(\rP_\Delta \rD x)},
$$
and by applying the assumption that $\hat x$ is a minimizer, so $\nm{\rD \hat x}_1 \leq \nm{\rD x}_1$, we have that $$
\nmu{\rP_\Delta^\perp \rD z}_1 \leq 2\nmu{\rP_{\Delta}^\perp \rD x}_1 + \abs{\Re\ip{\rP_\Delta \rD z}{\sgn(\rP_\Delta \rD x)}}.
$$
By properties (iii)-(v) of the dual certificate $\rho$ from Proposition \ref{prop:dual_existence_inc},
\eas{
\abs{\Re\ip{\rP_\Delta \rD z}{\sgn(\rP_\Delta \rD x)}} &\leq \abs{\ip{\rP_\Delta \rD z}{\sgn(\rP_\Delta \rD x)  - \rP_\Delta \rho}} + \abs{\ip{\rP_\Delta \rD z}{ \rP_\Delta \rho}} \\
&\leq \nmu{\rP_\Delta \rD z}_2   \cdot \frac{1}{4} + \abs{\ip{ \rD z}{\rho}} + \abs{\ip{ \rP_\Delta^\perp\rD z}{\rho}}\\
&\leq \nmu{\rP_\Delta \rD z}_2   \cdot \frac{1}{4}   + 
 \abs{\ip{ \rP_{\Omega'} \rA \rD z}{w}} + \nmu{\rP_\Delta^\perp \rD z}_1 \cdot \frac{1}{4} \\
 &\leq \delta\cdot\left( \sqrt{5} +
 4 C_0\cdot \sqrt{s}\right) +\frac{3}{4}\cdot \nmu{\rP_\Delta^\perp \rD z}_1 
}
where we have applied the Cauchy-Schwarz inequality with the fact that $\nmu{\rP_{\Omega'} \rA \rD z}_2 \leq 4\delta\sqrt{m}$ and $\nmu{w}_2 \leq  C_0 \cdot \sqrt{s m}$. So,
\be{\label{eq:step1_b}
\nmu{\rP_\Delta^\perp \rD z}_1 \leq 8\nmu{\rP_\Delta^\perp \rD x}_1
+  \delta\cdot\left( 4\sqrt{5} +
 16 \cdot C_0\cdot \sqrt{s}\right).
}

\textbf{
Step II:} Assume that $r\geq 1$ (recall that the $m^{-1/2}\rP_{\Omega'} \rA$ satisfies the weak RIP with respect to $r$). Let $h = \rD z$ and partition $\Delta^c$ into subsets of at most cardinality $r$ by letting $\Delta_1$ be the indices of the $r$ largest entries of $\rP_\Delta^\perp h$, $\Delta_2$ be the next $r$ largest entries and so on. Let $\tilde \Delta = \Delta\cup \Delta_1$. Since $m^{-1/2}\rP_{\Omega'} \rA$ satisfies the weak RIP condition (vi) from Proposition \ref{prop:dual_existence_inc},
\spl{\label{eq:bound_Delta_w}
\nmu{\rP_{\tilde \Delta} h}_2^2 &\leq \frac{4}{3}\nmu{\rA_{\Omega', \tilde \Delta} \rP_{\tilde \Delta} h}_2^2\\
&= \frac{4}{3\sqrt{m}}\left(\ip{\rA_{\Omega', \tilde \Delta} \rP_{\tilde \Delta} h}{\rP_{\Omega'} \rA h} - \ip{\rA_{\Omega', \tilde \Delta} \rP_{\tilde \Delta} h}{\rP_{\Omega'} \rA\rP_{\tilde \Delta}^\perp h}\right).
}
By applying the weak RIP condition and since $\nmu{\rP_{\Omega'} \rA h}_2 \leq 4\cdot \sqrt{m}\cdot \delta$ (as shown in Step I),
\be{\label{eq:first_term}
\abs{m^{-1/2}\ip{\rA_{\Omega', \tilde \Delta} \rP_{\tilde \Delta} h}{\rP_{\Omega'} \rA h}} \leq \sqrt{\frac{5}{4}} \cdot \nmu{\rP_{\tilde \Delta} h}_2 \cdot m^{1/2} \nmu{\rP_{\Omega'} \rA h}_2 \leq 2\sqrt{5} \cdot \nmu{\rP_{\tilde \Delta} h}_2 \cdot \delta.
}
Also, using the  standard compressed sensing result (see \cite[proof of Theorem 1.2]{candes2011probabilistic})
that
\bes{
\sum_{j\geq 2} \nmu{\rP_{\Delta_j} h}_2 \leq \frac{1}{\sqrt{r}} \nmu{\rP_\Delta^\perp h}_1,
}
one obtains
\be{\label{eq:2nd_term}
m^{1/2}\abs{ \ip{\rA_{\Omega', \tilde \Delta} \rP_{\tilde \Delta} h}{\rP_{\Omega'} \rA\rP_{\tilde \Delta}^\perp h}} \leq
\frac{1}{2\sqrt{r}} \nmu{\rP_{\tilde \Delta} h}_2 \nmu{\rP_{\Delta}^\perp h}_1.
}
Therefore, by combining (\ref{eq:bound_Delta_w}), (\ref{eq:first_term}) and (\ref{eq:2nd_term}),
\be{\label{eq:step1}
\nmu{\rP_{\tilde \Delta} h}_2 \leq \left(\frac{8\sqrt{5}}{4}\cdot \delta + \frac{2\nmu{\rP_\Delta^\perp h}_1}{3\sqrt{r}}\right)
}
and 
$$
\nmu{h}_2 \leq \nmu{\rP_{\tilde \Delta} h}_2 + \sum_{j\geq 2} \nmu{\rP_{\Delta_j} h}_2 \leq \left(\frac{8\sqrt{5}}{4}\cdot \delta + \frac{7\nmu{\rP_\Delta^\perp h}_1}{6\sqrt{r}}\right).
$$
Combining with the result of step I and observing that $r^{-1} \leq 1$ yields
\be{\label{eq:bound_grad_l2}
\nmu{h}_2 \leq \delta (15+ 19C_0\sqrt{s}) + \frac{10\nmu{\rP_\Delta^\perp \rD x}_1}{\sqrt{r}}
}
and by (\ref{eq:step1}) and the bound on $\nm{\rP_\Delta^\perp h}_1$ from step I,
\be{\label{eq:bound_grad}
\nmu{h}_1 \leq \sqrt{s} \cdot \nmu{\rP_\Delta h}_2 + \nmu{\rP_{\Delta}^\perp h}_1 \leq \sqrt{s}\nmu{\rP_{\tilde \Delta} h}_2 + \nmu{\rP_\Delta^\perp h}_1
\lesssim  \left(\delta(1+ C_0\sqrt{s}) + \sqrt{\frac{s}{r}} \nmu{\rP_\Delta^\perp \rD x}_1\right).
}

Recalling  $r$ from (vi) of Proposition \ref{prop:dual_existence_inc}, we have that
$$
\frac{s}{r} \leq \log(m) \log^2(s)
$$
and
\be{\label{eq:bound_grad_l1}
\nmu{h}_1 
\lesssim  \delta(1+ C_0\sqrt{s}) + \log^{1/2}(m) \log(s) \nmu{\rP_\Delta^\perp \rD x}_1.
}
Note that if $r\not\geq 1$, then (\ref{eq:bound_grad}) and (\ref{eq:bound_grad_l1}) are still true by combining (\ref{eq:step1_a}) and (\ref{eq:step1_b}) from Step I.

Finally, having shown stable gradient recovery, an application of (i) of Lemma \ref{lem:stab_grad_stab_sig} with (\ref{eq:bound_grad_l1}) concludes this proof.

\end{proof}

\subsubsection{Remark on Theorem \ref{thm:near_optimal}, Theorem \ref{cor:unif_samp_2D_opt} and Theorem \ref{cor:unif_samp_2D}}
The proofs of  Theorems \ref{thm:near_optimal}, \ref{cor:unif_samp_2D_opt} and \ref{cor:unif_samp_2D} are almost identical to the proof of Theorem \ref{thm:unif_samp} and we mention only the necessary modifications in this section.

For the proof of Theorem \ref{cor:unif_samp_2D}, first recall that the proof of Theorem \ref{thm:unif_samp}  relies on Proposition \ref{prop:dual_existence_inc}, the coherence of the discrete Fourier transform and the Poincar\'{e} inequality. In two dimensions, since the two dimensional discrete Fourier transform $\rA$ also has coherence $\mu(\rA)=1$, we may apply the two dimensional statement of Lemma \ref{lem:stab_grad_stab_sig} in the last sentence of the above proof. 

For the proof of Theorem \ref{thm:near_optimal},  we first use the assumption that $\Omega$ consists of a uniform random subsampled part $\Omega_1$ to derive the stable gradient recovery bounds of (\ref{eq:bound_grad_l2}) and (\ref{eq:bound_grad}) as in the proof of Theorem \ref{thm:unif_samp} above. Then, in the last sentence, instead of applying  Lemma \ref{lem:stab_grad_stab_sig} (which was a consequence of the Poincar\'{e} inequality), note that by the choice of $\Omega_2\subset \Omega$, we have that $\nm{\rP_{\Omega_2} \rA (x-\hat x)}_2 \leq \delta\sqrt{m}$ and the hypothesis of Proposition \ref{prop:gradtosignal_stability} is satisfied. So, Proposition \ref{prop:gradtosignal_stability} may be applied to reach the conclusion of Theorem \ref{thm:near_optimal}.

Similarly, the proof of the two dimensional case in Theorem \ref{cor:unif_samp_2D_opt} is the same as the proof of Theorem \ref{cor:unif_samp_2D}, but, we apply the two dimensional result Proposition \ref{prop:gradtosignal_stability_2d} instead of the two dimensional statement of Lemma \ref{lem:stab_grad_stab_sig} in the last part of the proof.

\subsection{Proof of Theorem \ref{thm:min_sep_thm}}
Throughout this section, for $M\in\bbN$,  let $[M] := \br{-M,\ldots, M}$ and let $\Delta := \br{t_1,\cdots, t_s} \subset \br{1,\ldots, N}$ be such that $t_1<t_2<\cdots< t_s$ and $\nu_{\min}(\Delta, N) \geq \frac{1}{M}$, where $\nu_{\min}$ is as defined in Definition \ref{def:min_sep}.

For the proof of Theorem \ref{thm:unif_samp}, we showed that stable gradient recovery is implied by conditions (i) to (vi) of Proposition \ref{prop:dual_existence_inc}, however, the following result shows that a weaker stable gradient recovery statement is still possible even when the last condition (vi) relating to the weak RIP is missing.

\begin{proposition}\label{prop:dual}
Suppose that $0\in\Omega$ and the following conditions hold.
\begin{itemize}
\item[(i)] there exists $\rU\in\bbC^{N\times N}$ such that $\tilde \rL = \frac{1}{m} \rP_\Delta \rU \rP_\Omega \rA \rP_\Delta$ is invertible on the subspace $ \mathrm{span}\br{e_j: j\in\Delta}$, $\nmu{\tilde \rL^{-1}}_{2\to 2} \leq \frac{4}{3}$ and $\nmu{m^{-1/2}\rP_\Delta \rU \rP_\Omega}_{2\to 2} \leq \frac{5}{4}$.
\item[(ii)] $\max_{j\in\Delta^c} \nmu{\frac{1}{\sqrt{m}}\rP_{\Omega } \rA \rP_\Delta^\perp e_j}_{2}\leq 6$.
\end{itemize}
and there exists $\rho = \rA^* \rP_\Omega w$ such that
\begin{itemize}
\item[(iii)] $\nmu{\rP_\Delta^\perp \rho}_\infty \leq c_0$.
\item[(iv)] $\nmu{\rP_\Delta \rho - \sgn(\rP_\Delta \rD x)}_2 \leq c_1 $
\item[(v)] $\nmu{w}_2 \leq c_2 \frac{\sqrt{s}}{\sqrt{m}}$
\end{itemize}
with constants $c_0,c_1,c_2>0$ such that $C_0  = \left(10 c_1 +  c_0 \right) <1$.
Let $\hat x$ be a minimizer of (\ref{eq:noise_tv}). Then 
$$
\nmu{\rD\hat x -\rD x}_2
\lesssim (1-C_0)^{-1}\left( (c_1 +  c_2\sqrt{s}) \cdot \delta + \nmu{\rP_\Delta^\perp\rD x}_1\right).
$$
and
$$
\frac{\nmu{\hat x - x}_2}{\sqrt{N}}
\lesssim (1-C_0)^{-1}\left( (c_1 \sqrt{s}+  c_2 s) \cdot \delta + \sqrt{s} \nmu{\rP_\Delta^\perp\rD x}_1\right).
$$

\end{proposition}

\begin{proof}
Let $z = \hat x - x$. We first demonstrate that $\nmu{\rP_\Delta \rD z}_2$ can be controlled by $\delta$ and $\nmu{\rP_\Delta^\perp \rD z}_1$.
\eas{
\nm{\rP_\Delta \rD z}_2 &= \nm{\tilde \rL^{-1} \tilde \rL \rP_\Delta \rD z}_2\\
&\leq  \frac{4}{3}\left(\nm{\frac{1}{m} \rP_\Delta \rU \rP_\Omega \rA \rD z}_2 + \nm{\frac{1}{m} \rP_\Delta \rU \rP_\Omega \rA \rP_\Delta^\perp \rD z}_2\right)\\
&\leq \frac{4}{3} \cdot\nm{\frac{1}{\sqrt{m}} \rP_\Delta \rU \rP_\Omega }_{2\to 2}\cdot\left( \frac{1}{\sqrt{m}} \nm{\rP_\Omega \rA \rD z}_2 +
\max_{j\in\Delta^c} \frac{1}{\sqrt{m}}\nm{\rP_{\Omega} \rA \rP_\Delta^\perp e_j}_2 \nm{\rP_\Delta^\perp \rD z}_1 \right)\\
&=\frac{5}{3} \cdot\left( \frac{1}{\sqrt{m}} \nm{\rP_\Omega \rA \rD z}_2 +
6\nm{\rP_\Delta^\perp \rD z}_1 \right)\\
&\leq \frac{20 \delta}{3} + 10\nm{\rP_\Delta^\perp \rD z}_1
}
Similarly to the proof of Theorem \ref{thm:unif_samp}, the last line of the above calculation follows because for $k\neq 0$, $v_k \cdot(\rA z)_k =  (\rA \rD z)_k$ where $v_k = 1-e^{2\pi i k/N}$. Also, $(\rA \rD z)_0 = 0$. Thus, 
 $$\nmu{\rP_\Omega \rA \rD z}_2 = 
 \nmu{\mathrm{diag}\left((v_k)_{k\in\Omega}\right)\cdot \rP_\Omega \rA  z}_2
 \leq 2 \nmu{\rP_\Omega \rA z}_2 \leq 4\delta\sqrt{m},$$ since $\abs{v_k} \leq 2$ and $\nmu{\rP_\Omega \rA z}_2 \leq 2\delta\sqrt{m}$ by the enforced constraint in the minimization problem.

To bound $\nmu{\rP_\Delta^\perp \rD x}_1$, first observe that since $\hat x$ is a minimizer and
$$
\nmu{\rD \hat x}_1\geq \nmu{\rP_\Delta^\perp \rD z}_1 - 2\nmu{\rP_{\Delta}^\perp \rD x}_1 + \nmu{\rD x}_1 + \Re\ip{\rP_\Delta \rD z}{\sgn(\rP_\Delta \rD x)},
$$
we have that
$$
\nmu{\rP_\Delta^\perp \rD z}_1 \leq 2\nmu{\rP_{\Delta}^\perp \rD x}_1 + \abs{\Re\ip{\rP_\Delta \rD z}{\sgn(\rP_\Delta \rD x)}}.
$$
By the assumed properties (iii)-(v) of the dual vector,
\eas{
\abs{\Re\ip{\rP_\Delta \rD z}{\sgn(\rP_\Delta \rD x)}} &\leq \abs{\ip{\rP_\Delta \rD z}{\sgn(\rP_\Delta \rD x)  - \rP_\Delta \rho}} + \abs{\ip{\rP_\Delta \rD z}{ \rP_\Delta \rho}} \\
&\leq \nmu{\rP_\Delta \rD z}_1   \cdot c_1 + \abs{\ip{ \rD z}{\rho}} + \abs{\ip{ \rP_\Delta^\perp\rD z}{\rho}}\\
&\leq \nmu{\rP_\Delta \rD z}_1   \cdot c_1  + 
 \abs{\ip{ \rP_\Omega \rA \rD z}{w}} + \nmu{\rP_\Delta^\perp \rD z}_1 \cdot c_0\\
 &\leq \delta\cdot\left(\frac{20 c_1}{3} +
 4 c_2\cdot \sqrt{s}\right) +\left(10 c_1 +  c_0 \right)\cdot \nmu{\rP_\Delta^\perp \rD z}_1 
}
where we use the fact that $\nmu{\rP_\Omega \rA \rD z}_2 \leq 4\delta\sqrt{m}$.
So,
$$
\nmu{\rD z}_1 \leq \sqrt{s}\nmu{\rP_\Delta\rD z}_2 + \nmu{\rP_\Delta^\perp \rD z}_1 \leq (1-C_0)^{-1}\cdot \left(  \delta\cdot\left(\frac{20 c_1 \sqrt{s}}{3} +
 4 c_2\cdot s\right)  +2 \sqrt{s} \nmu{\rP_\Delta^\perp \rD x}_1\right)
$$
Finally, the bound on $\nm{z}_2$ follows from a direct application of Lemma \ref{lem:stab_grad_stab_sig}.

\end{proof}

It remains to show that conditions (i) to (v) of Proposition \ref{prop:dual} are satisfied with high probability when $\Omega$ is chosen uniformly at random in accordance with Theorem \ref{thm:min_sep_thm}. For ease of analysis, it has become customary in compressed sensing theory to deduce recovery statements for the uniform sampling models by first proving statements of some alternative sampling model. One approach (considered in \cite{candes2011probabilistic, gross2011recovering}) is sampling with replacement where each sample is drawn independently, another popular approach is to consider a Bernoulli sampling model. To understand the use of the Bernoulli model, first recall the argument from \cite[Section II.C]{candes2006robust}, which shows that the probability that one of the conditions (i)-(v) of Theorem \ref{thm:min_sep_thm} fails for $\Omega$ chosen uniformly at random is up to a constant bounded from above by the probability that one of these conditions fails when $\Omega$ is chosen in accordance with a Bernoulli sampling model: $\Omega = \br{\delta_j \cdot j : j=-2M,\ldots, 2M}$ where $\delta_j$ are independent random variables such that $\bbP(\delta_j = 1) = q$ and $\bbP(\delta_j = 0) = 1-q$, where $q = m/M$ with $m$ as defined in (\ref{eq:num_samples_min_sep}). We denote such a choice of $\Omega$ by $\Omega \sim \mathrm{Ber}(q, 2M)$. Thus, it  suffices to show that the conditions hold choosing $\Omega$ in accordance with this Bernoulli sampling model.  We  assume throughout the rest of this section that $\Omega \sim \mathrm{Ber}(q, 2M)$.

Proposition \ref{prop:dual} presents conditions under which the gradient of $x$ (and hence $x$ by the Poincar\'{e} inequality) can be stably recovered by solving (\ref{eq:noise_tv}). These conditions are essentially the conditions required for stable recovery of a discrete signal from its low frequency Fourier data, and the proof of Theorem \ref{thm:min_sep_thm} will heavily rely on results from the analysis of super-resolution in \cite{candes2014towards,candes2013super,tang2012compressive}.
As mentioned before, the problem of constructing the dual certificate which satisfies the conditions of Proposition \ref{prop:dual} is a problem of interpolating a trigonometric polynomial. A common approach of tackling these problems is via the use of the squared Fej\'{e}r  kernel. Furthermore, $\rP_{[2M]} \rA \rP_\Delta$ is a Vandermonde type matrix, and its invertability can also be analysed via the squared Fej\'{e}r  kernel \cite{potts2010parameter}.

We first recall a discrete version of a result from \cite{candes2014towards} which shows that $\rP_{[2M]} \rA \rP_\Delta$ has a left inverse.

\begin{lemma}\cite{candes2014towards} \label{lem:fejer_leftinv}
 Let $K_M(t)$ be the squared Fej\'{e}r  kernel
$$
K_M(t) = \left(\frac{\sin(\pi M t)}{M \sin(\pi t)}\right)^4
= \frac{1}{M} \sum_{j=-2M}^{2M} g_M(j) e^{-\frac{2\pi i tj}{N}}
$$
where
$$
g_M(j) = \frac{1}{M} \sum_{k=\max(j-M,-M)}^{\min(j+M,M)} \left(1-\abs{\frac{k}{M}}\right)\left(1-\abs{\frac{j}{M} - \frac{k}{M}}\right).
$$
Define $\rL = (\rL_{j,k} )_{j,k=1}^s\in\bbC^{s\times s}$ by $\rL_{j,k} = K_M(t_j-t_k)$.
Then, 
\begin{enumerate}
\item $
\rL = \rP_\Delta \rA^* \rV \rA \rP_\Delta
$
where $\rV= (v_{j,k})_{j,k=-\lfloor N/2\rfloor+1}^{\lceil N/2\rceil}$, is a diagonal matrix such that $v_{k,k}=g_M(k)$ for $k=-2M,\ldots, 2M$ and has zero entries otherwise. 
\item $\nmu{g_M}_\infty \leq 1$
\item $\nmu{\rL - \rI}_{2\to 2} \leq 6.253\times 10^{-3}$ which implies that
\begin{enumerate}
\item $\rL$ is invertible and $\nmu{\rL^{-1}}_{2\to 2} \leq  0.993747$
\item $\nmu{\rL}_{2\to 2} \leq 1+6.253\times 10^{-3}$.
\end{enumerate}
\end{enumerate}
\end{lemma}

In the following two lemmas, we will prove that   conditions (i) and (ii) of Proposition \ref{prop:dual} are satisfied with high probability under the conditions of Theorem \ref{thm:min_sep_thm} by making use of the matrix Bernstein inequality \cite{tropp2012user} and Lemma \ref{lem:fejer_leftinv}.

\begin{lemma}\label{lem:cond_i} Consider the setup of Lemma \ref{lem:fejer_leftinv}.
Let $\Omega \sim \mathrm{Ber}\left(q, 2M\right)$ with $q= \frac{m}{M}$.
For each $\epsilon \in (0,1]$, if
$$
m \gtrsim  \log\left(\frac{s}{\epsilon}\right) \cdot s.
$$
Then with probability exceeding $1-\epsilon$,
\begin{itemize}
\item[(i)] $\tilde \rL = \frac{1}{m} \rP_\Delta \rA^* \rV \rP_\Omega \rA \rP_\Delta$ is invertible. 
\item[(ii)] $\nmu{\tilde \rL}_{2\to 2} \leq \frac{5}{4}$ and $\nmu{\tilde \rL^{-1}}_{2\to 2} \leq \frac{4}{3}$.
Note that since $\rV$ is a diagonal matrix, it is self adjoint and $\tilde \rL = \frac{1}{m} \rP_\Delta \rA^* \rV^{1/2} \rP_\Omega \rV^{1/2}\rA \rP_\Delta$. So, $\nmu{ \frac{1}{\sqrt{m}} \rP_\Delta \rA^* \rV^{1/2} \rP_\Omega}_{2\to 2} = \nmu{\tilde \rL}_{2\to 2}^{1/2} \leq \sqrt{\frac{5}{4}}$ and $\nmu{\rV^{1/2}}_{2\to 2} \leq 1$ since $\nmu{g_M}_\infty \leq 1$.
\end{itemize} 
If $q=1$, then (i) and (ii) hold with probability 1. 
 
\end{lemma}
\prf{
First, if $q=1$, then $\tilde L = L$ and the result follows from Lemma \ref{lem:fejer_leftinv}.
For the case when $q<1$, observe that from Lemma \ref{lem:fejer_leftinv}, we have that $\nmu{\rL - \rI}_{2\to 2} \leq 6.253\times 10^{-3}$, which implies that
$\nmu{\rL}_{2\to 2} - 6.253\times 10^{-3} \leq \nmu{\rL \rL^{-1}}_{2\to 2}$ and
$\nmu{\rL^{-1}}_{2\to 2} \geq 1- \frac{6.253\times 10^{-3}}{\nmu{\rL}_{2\to 2}} \geq 1- \frac{6.253\times 10^{-3}}{1- 6.253\times 10^{-3} } \geq 0.99$.
 Thus, if $\nmu{\tilde \rL - \rL}_{2\to 2} \leq \frac{1}{4\cdot0.99} $, then
$\nmu{\tilde \rL - \rL}_{2\to 2} \leq \frac{1}{4\cdot \nmu{\rL^{-1}}_{2\to 2}} $ and
\begin{enumerate}
\item  $\tilde \rL$ is invertible and $\nmu{\tilde \rL^{-1}}_{2\to 2} \leq \frac{4}{3}\cdot 0.99$.
\item $\nmu{\tilde \rL}_{2\to 2} \leq \frac{1}{4}+ 0.99$.
\end{enumerate}
So, it suffices to show that with probability exceeding $1-\epsilon$,
$$
\nmu{\tilde \rL - \rL}_{2\to 2} \leq \frac{1}{4\cdot0.99}.
$$
We will do so using the matrix Bernstein inequality.
Let $(\delta_j)_{j=-2M}^{2M}$ be independent Bernoulli random variables such that $\bbP(\delta_j = 1) = q$ and $\bbP(\delta_j = 0) = 1-q$. Let $\tilde K_M(t) =  \frac{1}{m}\sum_{j=-2M}^{2M} \delta_j g_M(j) e^{-\frac{2\pi i t j}{N}}$ and observe that $(\tilde \rL)_{j,k} =  \tilde K_M(t_j - t_k)$.

Let $\eta_j = \left(e^{-2\pi i t_1 j}, e^{-2\pi i t_2 j},\ldots,  e^{-2\pi i t_s j} \right)^T$. Then
\eas{
\tilde \rL - \rL &= \sum_{j=-2M}^{2M} \frac{g_M(j) \delta_j}{m}    \left( \eta_j\otimes \overline\eta_j\right)-  \sum_{j=-2M}^{2M} \frac{g_M(j)}{M}  \left( \eta_j\otimes \overline \eta_j\right)\\
&= \sum_{j=-2M}^{2M} \frac{g_M(j)}{M}  \left(\frac{\delta_j}{q} -1\right)  \left( \eta_j\otimes \overline \eta_j\right)
= \sum_{j=-2M}^{2M} X_j,
}
where $\otimes$ is the Kronecker product, and $X_j := \frac{g_M(j)}{M}  \left(\frac{\delta_j}{q} -1\right)  \left( \eta_j\otimes \overline \eta_j\right)$ are independent random self-adjoint matrices of zero mean.
In order to apply the matrix Bernstein inequality, we require bounds on $\max_j \nm{X_j}_{2\to 2}$ and $\nmu{\sum_{j=-2M}^{2M}\bbE(X_j^2)}_{2\to 2}$.
\begin{enumerate}
\item $$\nmu{X_j}_{2\to 2} \leq \frac{\abs{g_M(j)}}{M} \cdot \frac{1}{q}\cdot \nmu{\eta_j}_2^2 \leq \frac{s}{m}, $$
since $\nmu{\eta_j}_2^2 = s$ and  $\nmu{g_M(j)}_\infty \leq 1$. So, 
$$
R = \max_{\abs{j}\leq 2M} \nmu{X_j}_{2\to 2}\leq \frac{s}{m}.
$$
\item \eas{
&\sigma^2 = \nm{\sum_{j=-2M}^{2M}\bbE(X_j^2)}_{2\to 2} = \nm{ \sum_{j=-2M}^{2M} \left(\frac{1}{q}-1\right) \left(\frac{g_M(j)^2 \nm{\eta_j}^2_2}{M^2}\right) \left(\eta_j\otimes \overline \eta_j\right)}_{2\to 2}\\
&\leq \left(\frac{1}{q}-1\right) \left(\frac{g_M(j) \cdot s}{M}\right)  \nm{ \sum_{j=-2M}^{2M} \frac{g_M(j)}{M}  \left(\eta_j\otimes \overline \eta_j\right)}_{2\to 2}
\leq \left(\frac{1}{q}-1\right) \left(\frac{ s}{M}\right) \nmu{\rL}_{2 \to 2}
}
since $\nmu{g_M}_\infty \leq 1$ and $\nmu{\eta_j}_2^2 = 2$. Finally, because $\nmu{\rL}_{2\to 2} \leq 1.1$, it follows that $$ \sigma^2 \leq 1.1\left(\frac{1}{q}-1\right) \left(\frac{ s}{M}\right).$$

\end{enumerate}

Let $\gamma = \frac{1}{4\cdot 0.99} $. Then by matrix Bernstein,
\eas{
\bbP\left(\nmu{\tilde \rL - \rL}_{2\to 2} \geq \gamma \right) \leq s \cdot\exp\left(\frac{-\gamma^2/2}{\sigma^2 + R \gamma/3}\right) \leq \epsilon
}
provided that
$$
\log\left(\frac{s}{\epsilon}\right)\cdot \left(1.1+ \frac{\gamma}{2}\right)\cdot \frac{s}{m} \leq \frac{\gamma^2}{2}
$$

}

\begin{lemma}\label{lem:ii}

Let $\epsilon \in (0,1]$ and suppose that $\Omega \sim \mathrm{Ber}(q, 2M)$ with $q=m/M$ and 
$$
m\gtrsim  \log\left(\frac{N}{\epsilon}\right).
$$
Then,
 $$\bbP\left(\max_{j=1}^N \nm{\frac{1}{m}\rP_{\br{j}} \rA^* \rP_{\Omega } \rA \rP_{\br{j}}}_{2\to 2}\geq  5.5\right) \leq \epsilon.$$
 If $q=1$, then $\max_{j=1}^N \nmu{\frac{1}{m}\rP_{\br{j}} \rA^* \rP_{\Omega } \rA \rP_{\br{j}}}_{2\to 2}\leq 5.5$ holds with probability 1.
\end{lemma}
\prf{
First, if $q=1$, then for each $j=1,\ldots, N$,
\be{\label{eq:determi_bd}
\nm{\frac{1}{M} \rP_{\br{j}} \rA^* \rP_{[2M]} \rA \rP_{\br{j}}}_{2\to 2} = \frac{4M+1}{M}\leq 5
}
and so
$$\bbP\left(\max_{j=1}^N \nm{\frac{1}{m}\rP_{\br{j}} \rA^* \rP_{\Omega } \rA \rP_{\br{j}}}_{2\to 2}\geq  5.5\right) =0.$$
For $q<1$, let $(\delta_k)_{k=-2M}^{2M}$ be independent Bernoulli random variables such that $\bbP(\delta_j = 1) = q$ and $\bbP(\delta_j = 0) = 1-q$.
 Then
\eas{
\frac{1}{m}\rP_{\br{j}} \rA^* \rP_\Omega \rA \rP_{\br{j}} &- \frac{1}{M} \rP_{\br{j}} \rA^* \rP_{[2M]} \rA \rP_{\br{j}} = \sum_{k=-2M}^{2M} \frac{ \delta_k}{m} -  \sum_{k=-2M}^{2M} \frac{1}{M}  \\
&= \sum_{k=-2M}^{2M} \frac{1}{M}  \left(\frac{\delta_k}{q} -1\right)   
= \sum_{k=-2M}^{2M} X_k
}
where $X_k = \frac{1}{M}  \left(\frac{\delta_k}{q} -1\right)  $ are independent random variables of zero mean. 
So, combining with (\ref{eq:determi_bd}) gives that for each $j=1,\ldots, N$,
\eas{ 
&\bbP\left( \nm{\frac{1}{qM}\rP_{\br{j}} \rA^* \rP_{\Omega } \rA \rP_{\br{j}}}_{2\to 2}\geq 5.5\right) 
\leq \bbP\left( \abs{\sum_{k=-2M}^{2M} X_k} \geq  0.5  \right)
}
We will apply Bernstein's inequality \cite{foucart2013mathematical} to bound the right hand side of the above inequality. Observe that
$\abs{X_k} \leq \frac{1}{qM} = \frac{1}{m}$ and $$\sum_{k=-2M}^{2M} \bbE(X_k^2) = \frac{2M+1}{M^2}\cdot \left(\frac{1}{q}-1\right) \leq \frac{3}{m}.$$
Therefore,
$$
\bbP\left( \abs{\sum_{k=-2M}^{2M} X_k} \geq  0.5  \right)
\leq 2\exp \left( -\frac{1/8}{3/m+ 1/6m}\right).
$$
So, by applying the union bound, the conclusion follows.

}

We will consider the existence of $\rho = \rA^* \rP_\Omega w$ which satisfies conditions (iii) - (iv) in Proposition \ref{prop:dual} with $c_1 = 0$, $c_2=1$ and some constant $c_0 <1$. This dual certificate is actually identical to a discrete version of the dual certificate constructed in \cite{tang2012compressive}. We simply recall a few results from \cite{tang2012compressive} and provide a bound for $\nmu{w}_2$.

\begin{lemma}\cite[Section IV.C.]{tang2012compressive}\label{lem:CS_off_grid}
Let $(\delta_k)_{k=-2M}^{2M}$ be independent Bernoulli random variables such that $\bbP(\delta_j = 1) = q$ and $\bbP(\delta_j = 0) = 1-q$ with $q=m/M$.
Let $\Delta = \br{t_1,\ldots, t_s}$ and let $x\in\bbC^N$. If $\nu_{\min}(\Delta, N) \geq \frac{1}{M}$ and 
 $$m \gtrsim \max\br{ \log^2\left(\frac{M}{\epsilon}\right), \, \log\left(\frac{N}{\epsilon}\right)  ,\, s\cdot \log \left(\frac{s}{\epsilon}\right)\cdot \log\left(\frac{M}{\epsilon}\right) },
$$ then with probability exceeding $1-\epsilon$,
there exists constants $(\alpha_k)_{k=1}^s, (\beta_k)_{k=1}^s \in\bbR^s$ such that
the trigonometric polynomial
$$
Q(t) = \sum_{k=1}^s \alpha_k \overline K_M(t-t_k) + \sum_{k=1}^s \beta_k \overline K_M'(t-t_k)
$$
where $\overline K_M(t) = \sum_{\abs{j}\leq 2M} \delta_j g_M(j) e^{-2\pi i tj/N}
$ and $\overline K_M'$ is the first derivative of $\overline K_M$, satisfies the following.
\begin{enumerate}
\item $Q(t ) =( \sgn(\rP_\Delta x))_{t}$ for each $t\in\Delta$,
\item $\abs{Q(t)} < 1$ for all $t\not\in\Delta$
\item $\sqrt{\sum_{k=1}^s \abs{\alpha_k}^2 + \sum_{k=1}^s \abs{K_M''(0)}\abs{\beta_k}^2} \leq 2\sqrt{s}\cdot q^{-1}\cdot 1.568$.
\end{enumerate}
where $K_M''$ is the second derivative of the squared Fej\'{e}r  kernel and $K_M''(0) = -\frac{4\pi^2(M^2-1)}{3}$. Furthermore,for each $k\in\br{1,\ldots, N}\setminus \Delta$,
\be{\label{ineq:stability}
\abs{Q(k) } \leq \max \br{ 1- \frac{0.92 (M^2-1)}{N^2}, 0.99993}.
}

\end{lemma}
\begin{remark}
This lemma is essentially proved in \cite{tang2012compressive}. The inequality (\ref{ineq:stability}) is a result of combining the result of Proposition 4.12 and equation (IV.38) from \cite{tang2012compressive}.
\end{remark}

\begin{lemma}
Suppose that $M \geq 10$ and 
 $$m \gtrsim \max\br{ \log^2\left(\frac{M}{\epsilon}\right), \, s\cdot \log \left(\frac{s}{\epsilon}\right)\cdot \log\left(\frac{M}{\epsilon}\right) }.
$$
 Let $E$ be the event that conditions (i)-(v) of Proposition \ref{prop:dual} are satisfied, with
$$c_0 := \max \br{ 1- \frac{0.92 (M^2-1)}{N^2}, 0.99993}, \quad c_1 := 0, \quad c_2 := 1$$ and $\rU = \rA^* \rV$ where $\rV$ is the diagonal matrix defined in Lemma \ref{lem:fejer_leftinv} and $\Omega = \br{0} \cup \Omega'$ where $\Omega' \sim \mathrm{Ber}(m/M, 2M)$.
Then $\bbP(E) > 1-\epsilon$.
 
 \end{lemma}
 
 \prf{
Let $\rho = (Q(j))_{j=1}^N$ where $Q$ is as defined in Lemma \ref{lem:CS_off_grid}. Recall conditions (i) to (v) of Proposition \ref{prop:dual}.
Let $E_1$ be the event that the conclusions of Lemma \ref{lem:cond_i} holds (so (i) holds), $E_2$ be the event that (ii) holds and $E_3$ be the event that the conclusions of Lemma \ref{lem:CS_off_grid} holds (so $\rho$ satisfies both conditions (iii) and (iv)).  Suppose that the events $E_1$ and $E_3$ imply condition (v), then  to show that $\bbP(E) > 1-\epsilon$, it suffices to show that
\be{\label{eq:remains_to_show}
\bbP(E_1^c) \leq \epsilon/3, \quad \bbP(E_2^c) \leq \epsilon/3, \quad \bbP(E_3^c) \leq \epsilon/3.
}
 Observe that if the inequalities in (\ref{eq:remains_to_show}) are satisfied for the index set $\Omega'$, then they are satisfied for the index set $\Omega \supset \Omega'$. By the choice of $\Omega'$, $\bbP(E_1^c) \leq \epsilon/3$ follows from Lemma \ref{lem:cond_i}, $\bbP(E_2^c) \leq \epsilon/3$ follows from Lemma \ref{lem:ii} and $\bbP(E_3^c) \leq \epsilon/3$ follows from Lemma \ref{lem:CS_off_grid}.

It remains to demonstrate that if $E_1$ and $E_3$ both occur, then
$\rho = \rA^* \rP_\Omega w$ with $\nmu{w}_2 \leq \frac{\sqrt{s}}{\sqrt{m}}$ (so condition (v) holds):
By definition, for each $j=1,\ldots, N$,
\eas{
\rho_j &= \frac{1}{M} \sum_{\abs{l}\leq 2M} \delta_l g_M(l) e^{2\pi i jl/N} \sum_{k=1}^s \alpha_k e^{2\pi i t_k l/N} + 
\frac{1}{M} \sum_{\abs{l}\leq 2M} (-2\pi i l) \delta_l g_M(l) e^{2\pi i jl/N} \sum_{k=1}^s \beta_k e^{2\pi i t_k l/N}\\
&= \sum_{l\in\Omega} w_l e^{2\pi i jl/N}
}
where 
$$
w_l :=\frac{ g_M(l)}{M} \left(\sum_{k=1}^s \alpha_k e^{2\pi i t_k l/N}  -2\pi i l \sum_{k=1}^s \beta_k e^{2\pi i t_k l/N}\right).
$$
To bound $\nm{\rP_\Omega w}_2$, observe that
\eas{\frac{1}{q}\nmu{\rP_\Omega w}_2^2
&= \sum_{j\in\Omega} \frac{1}{q}\abs{w_j}^2
\leq \frac{2}{q}\sum_{j\in\Omega} \abs{\frac{ g_M(j)}{M}\sum_{k=1}^s \alpha_k e^{2\pi i t_k j/N}}^2 + \frac{2}{q}\sum_{j\in\Omega}\abs{\frac{ g_M(j)}{M}\cdot 4\pi^2 \abs{j}^2\cdot \sum_{k=1}^s \beta_k e^{2\pi i t_k j/N}}^2\\
&\leq \frac{2}{q}\sum_{j\in\Omega} \abs{\frac{ g_M(j)}{M}\sum_{k=1}^s \alpha_k e^{2\pi i t_k j/N}}^2 + \frac{2}{q}\sum_{j\in\Omega}\abs{\frac{ g_M(j)}{M}\cdot 14\abs{K_M''(0)}\cdot \sum_{k=1}^s \beta_k e^{2\pi i t_k j/N}}^2,
}
if we assume that $M\geq 10$ and since $K_M''(0) = -\frac{4\pi^2(M^2-1)}{3}$. Furthermore,
since $\tilde \rL =  m^{-1} \rP_\Delta \rA^* \rV \rP_\Omega \rA \rP_\Delta$,
\eas{\frac{1}{q}\nmu{\rP_\Omega w}_2^2
&\leq 2 \frac{\nmu{g_M}_\infty}{M}\ip{\tilde \rL \alpha}{\alpha}+   2 \frac{\nmu{g_M}_\infty}{M}\cdot 14\abs{K_M''(0)}\cdot \ip{\tilde \rL \beta}{\beta} \\
&\leq \frac{28}{M}\nmu{\tilde \rL}_{2\to 2} \left(\nmu{\alpha}_2^2 +  \abs{K_M''(0)} \nmu{\beta}_2^2\right),
}
where we have used $\nm{g_M}_\infty \leq 1$.   By assumption, $\nmu{\tilde \rL}_{2\to 2} \leq 5/4$ (from Lemma \ref{lem:cond_i}) and $$\sqrt{\sum_{k=1}^s \abs{\alpha_k}^2 + \sum_{k=1}^s \abs{K_M''(0)}\abs{\beta_k}^2} \leq 2\sqrt{s}\cdot q^{-1}\cdot 1.568$$ (from Lemma \ref{lem:CS_off_grid}), therefore,
$$
\nmu{\rP_\Omega w}_2 \lesssim \sqrt{\frac{q}{M}}\cdot \frac{\sqrt{s}}{q}= \sqrt{\frac{s}{m}}.
$$

}

\section{Concluding remarks}
This paper studied one type of variable density sampling which concentrates near low frequency Fourier samples. The recovery guarantees derived are optimal up to $\log$ factors. We also showed that in the case where the discontinuities of the underlying signal are sufficiently far apart, one  need only sample from low Fourier frequencies to ensure exact recovery. Our results provided some initial justification for the use of variable density sampling patterns over uniform random sampling patterns. First, variable density sampling guarantees near-optimal stability, and, although we do not prove that uniform random sampling cannot achieve the same stability guarantees, our numerical results suggest that dense sampling near the zero frequency does indeed substantially improve stability. Second, Theorem \ref{thm:min_sep_thm} demonstrated that the number of samples required can be substantially decreased under the additional assumption that the discontinuities of the underlying signal are sufficiently far apart. It is possible that this result is only a special case of a much more general theory, since our numerical examples suggest that variable density sampling patterns can be optimized to account for the gradient sparsity structure of the underlying signal.
 
\section{Acknowledgements}
 This work was supported by the UK Engineering and Physical Sciences
Research Council (EPSRC) grant EP/H023348/1 for the University of Cambridge
Centre for Doctoral Training, the Cambridge Centre for Analysis. The author would like to thank Ben Adcock and Anders Hansen for useful discussions. This paper  also significantly benefited from the comments of the anonymous referees. Finally, the author acknowledges Kwai Fan Yip for proofreading help and invaluable comments.

 \addcontentsline{toc}{section}{References}
\bibliographystyle{abbrv}
\bibliography{References}

\end{document}